\documentclass[oneside,11pt,reqno]{amsart}
 
\usepackage{amsmath, amscd}

\usepackage{amsfonts}
\usepackage{amssymb}
\usepackage{amsxtra}
\usepackage{tikz-cd}

\usepackage{amsmath}
\usepackage{amssymb}
\usepackage{graphicx, overpic}
\usepackage{caption}
\usepackage{subcaption}
\usepackage{pinlabel}
\usepackage{floatrow}
\usepackage{soul}
\usepackage{tikz}
\usetikzlibrary{decorations.pathreplacing}
\usetikzlibrary{decorations.markings}
\usepackage{color}
\usetikzlibrary{shapes.gates.logic.US,trees,positioning,arrows}
 \usepackage{float}
 
 \usepackage{eucal}
\usepackage[all]{xy}
\usepackage{multicol}

\usepackage{mathtools}

\newtheorem{thm}{Theorem}[section]
\newtheorem{prop}[thm]{Proposition}
\newtheorem{lem}[thm]{Lemma}

\theoremstyle{definition}
\newtheorem{defn}[thm]{Definition}
\newtheorem{rem}[thm]{Remark}
\newtheorem{conv}[thm]{Convention}

\newtheorem{ques}[thm]{Question}

\newtheorem*{thm_dense}{Theorem 5.7}
\newtheorem*{prop_ct2}{Proposition 5.9}
\newtheorem*{prop-MC}{Proposition 4.28}
\newtheorem*{thm_div}{Theorem 2.13}

\renewcommand{\bar}[1]{\overline{#1}}

\newcommand{\set}[2]{\{\,{#1} \mid {#2} \,\}}

\renewcommand{\emptyset}{\varnothing}

\newcommand{\field}[1]{\mathbb{#1}}
\newcommand{\Z}{\field{Z}}

\newcommand{\R}{\field{R}}

\newcommand{\N}{\field{N}}

\newcommand{\of}{\circ}
\renewcommand{\implies}{\Rightarrow}

\DeclareMathOperator{\Dist}{Dist}

\DeclareMathOperator{\CAT}{CAT}

\DeclareMathOperator{\Div}{Div}

\DeclareMathOperator{\diam}{diam}

\newcommand{\showcomments}{yes}

\usepackage{ifthen}
\newsavebox{\commentbox}
{\ifthenelse{\equal{\showcomments}{yes}}%
{\footnotemark
    \begin{lrbox}{\commentbox}
    \begin{minipage}[t]{1.25in}\raggedright\sffamily\tiny
    \footnotemark[\arabic{footnote}]}
{\begin{lrbox}{\commentbox}}}%
{\ifthenelse{\equal{\showcomments}{yes}}%
{\end{minipage}\end{lrbox}\marginpar{\usebox{\commentbox}}}
{\end{lrbox}}}

\begin{document}

\title[Divergence of finitely presented subgroups of CAT(0) groups]{Divergence of finitely presented subgroups\\of CAT(0) groups}

\begin{abstract}
 We construct families of $\CAT(0)$ groups containing finitely presented groups whose divergence functions are of the form $r^\alpha$ for a dense set of exponents $\alpha \in [2,\infty)$ and $r^q\log(r)$ for integers $q \geq 2$. The same construction also yields examples of $\CAT(0)$ groups containing contracting elements which are not contracting in certain finitely presented subgroups. 
\end{abstract} 

\author{Noel Brady}
\address{University of Oklahoma, Norman, OK 73019-3103, USA}
\email{nbrady@ou.edu}

\author{Hung Cong Tran}
\address{University of Oklahoma, Norman, OK 73019-3103, USA}
\email{Hung.C.Tran-1@ou.edu}

\maketitle

\section{Introduction}

In classical geometry, one can distinguish between euclidean space and hyperbolic space by measuring the rates at which geodesics diverge apart. In particular, in the zero curvature (euclidean) case the divergence is linear and in the negative curvature (hyperbolic) case the divergence is exponential. In \cite{MR1253544} Gromov discussed geodesic divergence for more general non-positively curved metric spaces and asked (6.B$_2$(h), page 90) if such divergence is always linear or exponential. In \cite{MR1254309} Gersten gave the first examples of $\CAT(0)$ groups whose geodesic divergence is quadratic. In the same paper, Gersten defined divergence as a quasi-isometry invariant of geodesic metric spaces and of finitely generated groups.

Roughly speaking, Gersten's notion of divergence is a function $\delta(r)$ which measures the diameter of the sphere of radius $r$ in the path-metric inherited from the complement of the open ball of radius $r$. In order to obtain a quasi-isometry invariant, one needs to consider the path-metric inherited from the complement of an open ball of radius $\rho r$ for some $0<\rho\leq 1$ and one works with equivalence classes of one-parameter families of functions $\{\delta_\rho(r)\}$. Precise definitions are given in section~\ref{sec:two}. 

In \cite{MR1909513} Macura studied divergence of free-by-cyclic groups with polynomially growing monodromies and gave the first example of a $\CAT(0)$ group with cubic divergence and the first examples of finitely presented groups with polynomial divergence of arbitrary degree. Examples of $\CAT(0)$ groups with polynomial divergence of arbitrary degree are provided in \cite{MR3032700} and \cite{MR3421592}. 

There has been intense study of divergence functions for particular families of finitely presented groups over the years. All these examples have divergence functions which are either polynomial or exponential. The first examples of finitely presented groups whose divergence is neither polynomial nor exponential appear in \cite{BT2020}. The groups in \cite{BT2020} make essential use of distorted infinite cyclic subgroups and so cannot be subgroups of $\CAT(0)$ groups. 
The research in the current paper was motivated by the following open question (paraphrased from \cite{MR3421592}).




\begin{ques}
\label{q1}
Can the divergence of a $\CAT(0)$ group be other than a polynomial or an exponential function? 
\end{ques}
This question has been answered in the negative for certain collections of $\CAT(0)$ groups; namely, right-angled Artin groups (see \cite{MR2874959}), right-angled Coxeter groups (see \cite{MR3314816,IL,MR3623669,Sisto,IL2}, $\CAT(0)$ $3$--manifold groups (see \cite{MR1302334}), and $\CAT(0)$ free-by-cyclic groups (see \cite{MR1909513}). However, there is no current strategy for handling $\CAT(0)$ groups in general. 

The main result of this paper (stated in Theorem~\ref{c2c2} below) is the construction of $\CAT(0)$ groups containing finitely presented subgroups that exhibit a wide range of divergence behavior beyond polynomial and exponential. 

\begin{thm_dense}
\textit{There exist $\CAT(0)$ groups containing finitely presented subgroups whose divergence is equivalent to $r^\alpha$ for a dense set of exponents $\alpha \in [2,\infty)$ and to $r^q\log(r)$ for integers $q\geq 2$.}
\end{thm_dense}

These subgroups provide further instances of the remarkable diversity of geometric behavior of subgroups of $\CAT(0)$ groups. 

There are two very useful properties of geodesics in negatively curved spaces. The first is that quasi-geodesics lie in bounded neighborhoods of geodesics between their endpoints; this was a critical ingredient in the proof of Mostow rigidity. This has been generalized to give the notions of Morse quasi-geodesic in a space and Morse element in a group. The second is the strong control on the diameter of images under nearest point projection onto geodesics. This has been generalized in \cite{ABD} to give the notion of contracting quasi-geodesic in a metric space and contracting element in a group. These group theory generalizations are related as follows. Every contracting element is Morse. Moreover, these notions are equivalent in $\CAT(0)$ and hyperbolic groups.
In Proposition~\ref{ppppp3} we show that the finitely presented subgroups in Theorem~\ref{c2c2} are not $\CAT(0)$ because they contain Morse elements which are not contracting.


\begin{prop-MC}[paraphrased]
\textit{The finitely presented subgroups in Theorem~\ref{c2c2} contain Morse elements which are not contracting. In particular, they are not $\CAT(0)$. }
\end{prop-MC}

In Remark~\ref{dehn} we use Dehn functions to give an alternate reason why the subgroups in Theorem~\ref{c2c2} are not $\CAT(0)$ themselves.

%
%

 It is not hard to see from the definition of contracting elements that if $H$ is an undistorted subgroup of a group $G$ then all group elements in $H$ which are contracting in $G$ are also contracting in $H$. However, we prove that this fact is not true for distorted subgroups of $\CAT(0)$ groups. These are the first examples in which the property of being a contracting element is not inherited on passing to a finitely presented subgroup. 

\begin{prop_ct2}
\textit{There exist $\CAT(0)$ groups containing contracting elements which are not contracting in some finitely presented subgroups.}
\end{prop_ct2}



The construction in this paper is similar in outline to that of \cite{BT2020}, although there are significant technical differences. 
The key geometric idea behind the construction in \cite{BT2020} is to start with a base $\Z^2 = \langle a_0, a_1 | [a_0,a_1]=1\rangle$
which is amalgamated along the $\langle a^{-1}_0a_1\rangle$ subgroup to a suitable group in order to distort the metric along this direction. 
An important feature is that the $a_0$ and $a_1$ lines remain undistorted in the amalgam. Then one builds a sequence of groups as iterated HNN extensions in a similar fashion to the examples of Gersten and Macura. This construction gives rise to examples of finitely presented groups with new divergence behavior; namely, $r^{\alpha}$ and $r^q\log(r)$. 


The main challenge is to find a suitable replacement for the base plane with distorted direction of \cite{BT2020}. We construct an analogue of this base plane using products of finite rank free groups in such a way that the distorted paths are far from periodic and so that infinite cyclic subgroups remain undistorted. These can then be embedded into $\CAT(0)$ groups. These embeddings rely on the fact that there are $\CAT(0)$ free-by-cyclic groups with distorted fibers.

The analogue of the base plane of \cite{BT2020} is a direct product of three free groups with certain diagonal subgroups taking the place of the $a_0$, $a_1$ directions and a skew-diagonal group as the transverse direction. This construction has a combinatorial shortcoming: analogues of triangles built from $a_0$, $a_1$ and transverse direction sides can only be constructed using palindromic words in the free groups. This is not a serious shortcoming since power function distortions of palindromic paths can be provided 
using the snowflake construction of \cite{MR3705143} where palindromic automorphisms play a central role, and exponential distortions can be provided using a free-by-cyclic group with palindromic monodromy of exponential growth.

A significant portion of the paper is devoted to proving the following technical divergence result. The product of three free groups (the analogue of the base plane in \cite{BT2020}) appears in the definition of $G_1$; so does the skew-diagonal free group analogue of $a_0^{-1}a_1$. The diagonal free group analogues of $a_0$ and $a_1$ appear in the definition of $G_2$. 
Note that a key contribution to the divergence comes from the inverse of a distortion function.

\begin{thm_div}
\textit{Let $H$ be a finitely presented group that contains a free subgroup $F$ of rank $p$ generated by $d_i$ for $1\leq i \leq p$. Assume that the distortion function $\Dist_F^H$ admits an exponentially bounded sequence of palindromic certificates in $F$. Let $F_x$, $F_y$, and $F_z$ be other free groups of rank $p$ generated by $x_i$, $y_i$, and $z_i$ for $1\leq i \leq p$ respectively.}

\textit{For each integer $m \geq 1$ there are finitely presented one-ended groups
\begin{align*}
    & G_1 \; =\; \bigl[H\ast_{\langle d_i=x_iy_i^{-1}\rangle} (F_x\times F_y \times F_z)\bigr] \times {\langle s_1 \rangle};\\& G_2\;=\; \langle G_1, s_2 | s_2^{-1}(x_iz_i)s_2=y_{i}z_i\rangle; \\
    & {\hbox{and}}\\
    & G_m \; =\; \langle G_{m-1}, s_m | s_m^{-1}s_1s_m=s_{m-1}\rangle \quad \quad {\hbox{ for $m \geq 3$}}
\end{align*}
with the following properties. The subgroup $\langle s_m \rangle$ is undistorted and the following table holds.
}
\begin{center}
\begin{tabular}{|p{0.3in}|p{1.4in}|p{1.4in}|}
\hline
\hfill \rule[-0.3cm]{0cm}{0.8cm} {\bf $m$} & \hspace{0.3in} {\bf ${\rm Div}_{\langle s_m\rangle}^{G_m}$ } & \hspace{0.3in} {\bf ${\rm Div}_{G_m}$}\\
\hline
\hline
\rule[-0.3cm]{0cm}{0.8cm} \hfill $2$ & \hspace{0.3in}$r(\Dist_F^H)^{-1}(r)$ \ & \hspace{0.3in}$r^2$\\
\hline
\rule[-0.3cm]{0cm}{0.8cm} \hfill $\geq 3$ & \hspace{0.2in}$r^{m-1}(\Dist_F^H)^{-1}(r)$ &\hspace{0.2in}$r^{m-1}(\Dist_F^H)^{-1}(r)$\\
\hline
\end{tabular}
\end{center}

\end{thm_div}

Using $\CAT(0)$ free-by-cyclic groups as building blocks, we construct a sequence of $\CAT(0)$ groups $B_m$ which mirrors the construction of the groups $G_m$ above and we prove that the groups $B_m$ contain the groups $G_m$. We refer the reader to section~\ref{sbasic} for the constructions of $\CAT(0)$ groups $B_m$. For the $r^\alpha$ divergence (resp. $r^q\log{r}$) in Theorem~\ref{c2c2} we use the snowflake groups of \cite{MR3705143} (resp. certain free-by-cyclic group with palindromic automorphism of exponential growth) as the subgroup $H$ in the theorem above.

This paper is organized as follows. Section 2 contains definitions and a statement of the main divergence theorem. Section 3 includes fundamental geometric and algebraic properties of the groups $G_m$ which play a key role in the proof of the main divergence theorem. The proof of this theorem is given in section 4. In section 5 the sequence of $\CAT(0)$ groups $(B_m)$ is defined and the inclusions $G_m \leq B_m$ are established. 

\subsection*{Acknowledgements} 
The first author was supported by Simons Foundation collaboration grant \#430097. 
The second author was supported by an AMS-Simons Travel Grant. 

\section{Definitions and statement of the divergence results}
\label{sec:two}
\begin{conv}
\label{cv}
Let $\mathcal{M}$ be the collection of all functions from positive reals to positive reals. Let $f$ and $g$ be arbitrary elements of $\mathcal{M}$. We say that $f$ is \emph{dominated} by $g$, denoted $f \preceq g$, if there are positive constants $A, B, C$ such that $f(x) \leq g(Ax) + Bx$ for all $x > C$. We say that $f$ is \emph{equivalent} to $g$, denoted $f\sim g$, if $f\preceq g$ and $g\preceq f$. This defines an equivalence relation on $\mathcal{M}$. 

One can define an equivalence relation between indexed families of functions as follows. 
Let $\{\delta_\rho\}$ and $\{\delta'_\rho\}$ be two families of functions of $\mathcal{M}$, indexed over $\rho \in (0,1]$. The family $\{\delta_\rho\}$ is \emph{dominated} by the family $\{\delta'_\rho\}$, denoted $\{\delta_\rho\}\preceq \{\delta'_\rho\}$, if there exists constant $L \in (0,1]$ such that $\delta_{L\rho}\preceq \delta'_\rho$ for all $\rho \in (0,1]$. We say $\{\delta_\rho\}$ is \emph{equivalent} to $\{\delta'_\rho\}$, denoted $\{\delta_\rho\}\sim \{\delta'_\rho\}$, if $\{\delta_\rho\}\preceq \{\delta'_\rho\}$ and $\{\delta'_\rho\}\preceq \{\delta_\rho\}$. If $f$ is a function in $\mathcal{M}$, the family $\{\delta_\rho\}$ is \emph{equivalent} to $f$ if there is $b\in (0,1]$ such that $\delta_\rho$ is equivalent to $f$ for each $\rho \in (0,b]$.

\end{conv}

We now recall Gersten's definition of divergence from \cite{MR1254309}. Let $X$ be a geodesic space and $x_0$ one point in $X$. Let $d_{r,x_0}$ be the induced length metric on the complement of the open ball with radius $r$ about $x_0$. If the point $x_0$ is clear from context, we will use the notation $d_r$ instead of $d_{r,x_0}$.

\begin{defn}[Group divergence]
Let $X$ be a geodesic space with a fixed point $x_0$. For each $\rho\in (0,1]$ we define a function $\delta_\rho\!: (0, \infty) \to (0, \infty)$ as follows. For each $r>0$, let $\delta_\rho(r) = \sup d_{\rho r}(x_1, x_2)$ where the supremum is taken over all $x_1$ and $x_2$ on the sphere $S(x_0,r)$ such that $d_{\rho r}(x_1,x_2)<\infty$. The family of functions $\{\delta_\rho\}$ is the \emph{divergence} of $X$.

Using Convention~\ref{cv} one can show that up to equivalence the divergence of $X$ does not depend on the choice of $x_0$ and it is a quasi-isometry invariant (see \cite{MR1254309}). The \emph{divergence} of a finitely generated group $G$, denoted ${\rm Div}_{G}$, is the divergence of the Cayley graph $\Gamma(G,S)$ for some (any) finite generating set $S$. 
\end{defn}

We now recall the definition of quasi-geodesic divergence and undistorted cyclic subgroup divergence.

\begin{defn}[Quasi-geodesic divergence]
Let $X$ be a geodesic space and $\alpha\!:(-\infty,\infty)\to X$ be an $(L,C)$--bi-infinite quasi-geodesic. The \emph{divergence} of $\alpha$ in $X$ is the function ${\rm Div}_\alpha\!:(0,\infty) \to (0,\infty)$ defined as follows. Let $r>0$ be an arbitrary number. If there is no path outside the open ball $B(\alpha(0),r/L-C)$ connecting $\alpha(-r)$ and $\alpha(r)$, we define ${\rm Div}_\alpha(r)=\infty$. Otherwise, we define ${\rm Div}_\alpha(r)$ to be the infimum over the lengths of all paths outside the open ball $B(\alpha(0),r/L-C)$ connecting $\alpha(-r)$ and $\alpha(r)$. 
\end{defn}

\begin{defn}[Cyclic subgroup divergence]
Let $G$ be a finitely generated group and $\langle c \rangle $ be an undistorted, infinite cyclic subgroup of $G$. Let $S$ be a finite generating set of $G$ that contains $c$. Since $\langle c\rangle$ is undistorted, every bi-infinite path with edges labeled by $c$ is a quasi-geodesic in the Cayley graph $\Gamma(G,S)$. 
The \emph{divergence} of the cyclic subgroup $\langle c \rangle$ in $G$, denoted ${\rm Div}_{\langle c \rangle}^{G}$, is defined to be the divergence of such a bi-infinite quasi-geodesic.
\end{defn}

Using Convention~\ref{cv} the divergence of the cyclic subgroup $\langle c \rangle$ in $G$ does not depend on the choice of finite generating set $S$. We leave the proof of this fact as an exercise for the reader. 

We now review the concept of contracting quasi-geodesics and contracting elements in a finitely generated group.


\begin{defn} [Contracting quasi-geodesic]
Let $X$ be a geodesic space. A bi-infinite quasi-geodesic $\alpha$ in $X$ is \emph{contracting} in $X$ if there exist a map $\pi_\alpha\!:X\to \alpha$ and constants $0< A< 1$ and $D\geq 1$ satisfying:
\begin{enumerate}
\item $\pi_\alpha$ is $(D,D)$--coarsely Lipschitz (i.e., $\forall x_1, x_2 \in X, d(\pi_\alpha(x_1),\pi_\alpha(x_2)) \leq Dd\bigl(x_1, x_2\bigr) + D$). 
\item For any $y\in \alpha$, $d\bigl(y,\pi_\alpha(y)\bigr)\leq D$.
\item For all $x\in X$, if we set $R=A d(x,\alpha)$, then $\diam\bigl(\pi_\alpha\bigl(B(x,R)\bigr)\bigr)\leq D$.
\end{enumerate}
\end{defn}

\begin{defn}[Contracting element]
Let $G$ be a finitely generated group and $g$ an infinite order element in $G$. Let $S$ be a finite generating set of $G$ that contains $g$. Let $\alpha$ be a bi-infinite path in the Cayley graph $\Gamma(G,S)$ with edges labeled by $g$. The element $g$ in $G$ is \emph{contracting} if the path $\alpha$ is a contracting quasi-geodesic.
\end{defn}

It is not hard to see that the notion of contracting quasi-geodesic is a quasi-isometry invariant. In particular, the concept contracting elements in a finitely generated group do not depend on the choice of finite generating set $S$. The following proposition shows that the divergence of a contracting quasi-geodesic is always at least quadratic.

\begin{prop}[See Proposition 5.3 in \cite{BT2020}]
\label{lem_CS}
Let $\alpha\!: (-\infty,\infty)\to X$ be an $(L,C)$--quasi-geodesic in a geodesic space $X$. Assume that $\alpha$ is also an $(A,D)$--contracting quasi-geodesic. Then the divergence of $\alpha$ is at least quadratic.
\end{prop}




\begin{defn}[Lipschitz equivalence of functions]
Let $f,g\!:[0,\infty)\to [0,\infty)$ be two non-decreasing function. We say that $f$ and $g$ are \emph{Lipschitz equivalent} if there is a constant $M\geq 1$ such that for each $r\geq 1$ we have 
$$g(r)\leq M f(Mr) \text{ and } f(r)\leq M g(Mr).$$
\end{defn}

\begin{defn}[Distortion function]
\label{edf}
Let $H$ be a finitely generated group with a finite generating set $T$ and let $F$ be a subgroup of $H$ with a generating set $R$. The \emph{distortion function} $\Dist_F^H\!:\N\to \N$ is defined as follows:
$$\Dist_F^H(n)=\max\set{|g|_R}{g\in F, |g|_T\leq n} \text{ for $n\in\N$}.$$
We note that the function $\Dist_F^H$ is non-decreasing. We can also consider the function $\Dist_F^H$ as a non-decreasing function from $[0,\infty)$ to $[0,\infty)$ by linear interpolation over the open unit intervals. 

The function $\Dist_F^H$ is said to \emph{admit an exponentially bounded sequence of certificates} $(g_m)$ in $F$ if there is a constant $C>1$ such that the following hold:
\begin{enumerate}
    \item $|g_m|_R\to \infty$ as $m\to \infty$; 
    \item $|g_{m+1}|_R\leq C|g_m|_R$ for each $m$; and
    \item $\Dist_F^H\bigl(|g_m|_T/C\bigr)\leq C|g_{m}|_R$ for each $m$.
\end{enumerate}
The certificate elements $g_m$ will be used to certify upper bounds on group divergence functions later on.

Up to Lipschitz equivalence, the distortion function $\Dist_F^H$ does not depend on the choice of finite generating sets $R$ and $T$. 
Furthermore, the fact that a distortion function admits an exponentially bounded sequence of certificates also does not depend on the choice of finite generating sets $R$ and $T$. 

Note that the function $\Dist_F^H(r)$ grows at least linearly and so is Lipschitz equivalent to the monotonic increasing function $\Dist^H_F(r) + r$. Henceforth, we will always assume that $\Dist^H_F(r)$ is a bijective function.

\end{defn}



\begin{rem}
The reader may recall that the distortion of $F_{\{a,b\}}$ in $F_{\{a,b\}} \rtimes_\varphi \Z$ where $\varphi$ has exponential growth is exponential. The upper bound for this distortion function is proven algebraically using successive applications of Britton's lemma for HNN extensions. The sequence of elements $(\varphi^n(a))$ is exponentially bounded in the $F_{\{a,b\}}$--metric (and is linearly bounded in the 
$F_{\{a,b\}} \rtimes_\varphi \Z$--metric) and provides a lower bound for the distortion function. 

This same sequence $(\varphi^n(a))$ is an exponentially bounded sequence of certificates in the sense of Definition~\ref{edf}. However, note that the inequality in item (3) of that definition is the opposite of what would be used if we were establishing a lower bound for the distortion function. This is important. As was stated in Definition~\ref{edf}, this sequence will be used later on to certify an upper bound for a divergence function. The next lemma reformulates conditions (1)--(3) of that definition in a format which will be easy to apply to give upper bounds on group divergence. 
\end{rem}

\begin{lem}
\label{hohohhh}
Let $H$ be a finitely generated group with a finite generating set $T$ and let $F$ be a subgroup of $H$ with a generating set $R$. Up to the Lipschitz equivalence relation we can choose the distortion $\Dist_F^H$ such that $$\Dist_F^H(r)\geq r \text{ for each } r\geq 1 \text{ and } (\Dist_F^H)^{-1}\bigl(|g|_R\bigr)\leq |g|_T \text{ for each } g\in F.$$ 
Moreover, if $\Dist_F^H$ admits an exponentially bounded sequence of certificates $(g_m)$ in $F$, then there are constants $D>1$ and $r_0\geq 1$ such that for each $r\geq r_0$ there is a group element $u$ in the sequence $(g_m)$ satisfying $$r/D\leq |u|_R\leq r \text{ and }|u|_T\leq D~(\Dist_F^H)^{-1}(r).$$ 
\end{lem}

\begin{proof}

We note that the function $\Dist_F^H(r)$ grows at least linearly and so is Lipschitz equivalent to the monotonic increasing function $\Dist^H_F(r) + r$. Replacing $\Dist_F^H(r)$ by $\Dist^H_F(r) + r$ (if necessary), we can assume that $\Dist_F^H(r)\geq r$ for each $r\geq 1$. We also observe that $$\Dist_F^H\bigl(|g|_T\bigr)\geq |g|_R \text{ for each $g\in F$}.$$ Therefore, $(\Dist_F^H)^{-1}\bigl(|g|_R\bigr)\leq |g|_T$.

We now prove second statement. Since the distortion function $\Dist_F^H$ admits the exponentially bounded sequence of certificates $(g_m)$, there is a constant $C>1$ such that the following hold:
\begin{enumerate}
    \item $|g_m|_R\to \infty$ as $m\to \infty$; 
    \item $|g_{m+1}|_R\leq C|g_m|_R$ for each $m$; and
    \item $\Dist_F^H\bigl(|g_m|_T/C\bigr)\leq C|g_{m}|_R$ for each $m$.
\end{enumerate}

Let $r_0=C|g_1|_R\geq 1$ and $D=C^3>1$. Let $r\geq r_0$ be an arbitrary number. Then there is a positive integer $n$ such that $$C^n|g_1|_R\leq r \leq C^{n+1}|g_1|_R.$$ Since $|g_m|_R\to \infty$ as $m\to \infty$ and $|g_{m+1}|_R\leq C|g_m|_R$ for each $m$, we can also choose $u=g_{m_0}$ for some $m_0$ such that $$C^{n-2}|g_1|_R\leq |u|_R\leq C^{n-1}|g_1|_R.$$

By the choice of $r_0$, $D$, and $n$ we observe that 
$$|u|_R\geq C^{n-2}|g_1|_R \geq(C^{n+1}|g_1|_R)/C^3\geq r/C^3\geq r/D$$ and $$|u|_R\leq C^{n-1}|g_1|_R \leq (C^{n}|g_1|_R)/C\leq r/C.$$
Moreover, 
$$\Dist_F^H\bigl(|u|_T/D\bigr)\leq \Dist_F^H\bigl(|u|_T/C\bigr)\leq C|u|_R\leq C(r/C)\leq r $$
Therefore, $|u|_T/D\leq (\Dist_F^H)^{-1}(r)$ or $|u|_T\leq D~(\Dist_F^H)^{-1}(r)$.

\end{proof}

\begin{defn}
Let $F$ be a free group of rank $p$ with a generating set $R=\set{d_i}{1\leq i \leq p}$. A word in $R$ is \emph{palindromic} if it reads the same left-to-right as right-to-left. A group element $g\in F$ is \emph{palindromic} with respect to $R$ if the reduced word in $R$ that represents $g$ is palindromic. 
\end{defn}

We now state the main divergence result of the paper. The reason for the technical conditions in this statement will become clear as we fill in details of the proof in the next two sections.

\begin{thm}\label{main_thm}
Let $H$ be a finitely presented group that contains a free subgroup $F$ of rank $p$ generated by $d_i$ for $1\leq i \leq p$. Assume that the distortion function $\Dist_F^H$ admits an exponentially bounded sequence of palindromic certificates in $F$. Let $F_x$, $F_y$, and $F_z$ be other free groups of rank $p$ generated by $x_i$, $y_i$, and $z_i$ for $1\leq i \leq p$ respectively.

For each integer $m \geq 1$ there are finitely presented one-ended groups
\begin{align*}
    & G_1 \; =\; \bigl[H\ast_{\langle d_i=x_iy_i^{-1}\rangle} (F_x\times F_y \times F_z)\bigr] \times {\langle s_1 \rangle};\\& G_2\;=\; \langle G_1, s_2 | s_2^{-1}(x_iz_i)s_2=y_{i}z_i\rangle; \\
    & {\hbox{and}}\\
    & G_m \; =\; \langle G_{m-1}, s_m | s_m^{-1}s_1s_m=s_{m-1}\rangle \quad \quad {\hbox{ for $m \geq 3$}}
\end{align*}
with the following properties. The subgroup $\langle s_m \rangle$ is undistorted and the following table holds.
\begin{center}
\begin{tabular}{|p{0.3in}|p{1.4in}|p{1.4in}|}
\hline
\hfill \rule[-0.3cm]{0cm}{0.8cm} {\bf $m$} & \hspace{0.3in} {\bf ${\rm Div}_{\langle s_m\rangle}^{G_m}$ } & \hspace{0.3in} {\bf ${\rm Div}_{G_m}$}\\
\hline
\hline
\rule[-0.3cm]{0cm}{0.8cm} \hfill $2$ & \hspace{0.3in}$r(\Dist_F^H)^{-1}(r)$ \ & \hspace{0.3in}$r^2$\\
\hline
\rule[-0.3cm]{0cm}{0.8cm} \hfill $\geq 3$ & \hspace{0.2in}$r^{m-1}(\Dist_F^H)^{-1}(r)$ &\hspace{0.2in}$r^{m-1}(\Dist_F^H)^{-1}(r)$\\
\hline
\end{tabular}
\end{center}

\end{thm}

The proof of this theorem occupies the next two sections. In section~\ref{sec:three} we establish some algebraic properties of $G_m$; the reason for the palindromic condition becomes clear in Lemma~\ref{a1}. In section~\ref{sec:four} we establish upper bounds and lower bounds for $\Div_{G_m}$ and for the divergence of $\langle s_m\rangle$ in $G_m$. 


\section{Algebraic and geometric properties of the groups $G_m$}
\label{sec:three}

We fix generating sets $S_m$ for the groups $G_m$ in Theorem~\ref{main_thm}. 
\begin{defn}[Generating sets]\label{gene}
The group $G_1$ is defined in the Theorem~\ref{main_thm} as an amalgam 
$$
G_1 \; =\; \bigl[H\ast_{\langle d_i=x_iy_i^{-1}\rangle} (F_x\times F_y \times F_z)\bigr] \times {\langle s_1 \rangle};
$$
where group $H$ has finite generating set $T$, three free groups $F_x$, $F_y$, and $F_z$ of rank $p$ generated by the finite generating sets $R_x=\set{x_i}{1\leq i \leq p}$, $R_y=\set{y_i}{1\leq i \leq p}$, and $R_z=\set{z_i}{1\leq i \leq p}$ respectively. Group $H$ contains the free subgroup $F$ with the generating set $R=\set{d_i}{1\leq i \leq p}$ and we assume that $T$ contains $R$.

For each $1\leq i \leq p$ we let $a_i=x_iz_i$ and $b_i=y_iz_i$. Then two subgroups $F_{xz}$ and $F_{yz}$ generated by two sets $R_{xz}=\set{a_i}{1\leq i \leq p}$ and $R_{yz}=\set{b_i}{1\leq i \leq p}$ respectively are free of rank $p$. Let $S_1=T\cup R_x \cup R_y \cup R_z\cup R_{xz} \cup R_{yz} \cup \{s_1\}$. Then $S_1$ is a finite generating set for $G_1$. Moreover, we also can write 
\begin{align*}
    & G_1 \; =\; \bigl[H\ast_{\langle d_i=a_ib_i^{-1}\rangle} (F_x\times F_y \times F_z)\bigr] \times {\langle s_1 \rangle};\\& G_2\;=\; \langle G_{1}, s_2 | s_2^{-1}a_is_2=b_i\rangle.
\end{align*}

For each integer $m \geq 3$ the groups $G_m$ are defined recursively by 
$$
G_m \; =\; \langle G_{m-1}, s_m \, |\, s_m^{-1}s_1s_m=s_{m-1} \rangle
$$
and so have recursively defined finite generating sets 
$S_m=S_{m-1}\cup \{s_m\}$ for groups $G_m$ for each $m\geq 2$.
\end{defn}

The following lemma will be used repeatedly to prove the group divergence by induction in Theorem~\ref{main_thm}. In particular, the last two inequalities in the following lemma will be used to show that certain paths lies outside certain open balls.

\begin{lem}\label{ll0}
The inclusion maps $i_0\!:\Gamma(F_{xz},R_{xz})\hookrightarrow \Gamma(G_{1},S_{1})$ and $i_1\!:\Gamma(F_{yz},R_{yz})\hookrightarrow \Gamma(G_{1},S_{1})$ are isometric embedding maps. Moreover, if $g_0$, $g_1$, and $h$ are group elements in $F_{xz}$, $F_{yz}$, and $H$ respectively, then
$$|g_0h|_{S_1}\geq |g_0|_{{S_1}} \text{ and } |g_1h|_{S_1}\geq |g_1|_{{S_1}}.$$
\end{lem}

\begin{proof}
We observe that there is a group homomorphism $\Psi\!: G_1\to F_z$ taking each $z_i$ to itself and taking other generators to the identity. We observe that $\Psi_{|F_{xz}}$ be a group monomorphism taking each $a_i$ to $z_i$ and $\Psi(F_{xz})=F_z$. Similarly, $\Psi_{|F_{yz}}$ is also a group monomorphism taking each $b_i$ to $z_i$ and $\Psi(F_{yz})=F_z$. Therefore, the inclusion maps $i_0\!:\Gamma(F_{xz},R_{xz})\hookrightarrow \Gamma(G_{1},S_{1})$ and $i_1\!:\Gamma(F_{yz},R_{yz})\hookrightarrow \Gamma(G_{1},S_{1})$ are isometric embedding maps.

Now let $g_0$, $g_1$, and $h$ be group elements in $F_{xz}$, $F_{yz}$, and $H$ respectively. Then,
$$|g_0h|_{S_1}\geq |\Psi(g_0h)|_{R_z} \text{ and } |g_1h|_{S_1}\geq |\Psi(g_1h)|_{R_z}.$$
Also, $|\Psi(g_0h)|_{R_z}=|\Psi(g_0)|_{R_z}=|g_0|_{R_{xz}}=|g_0|_{S_1} $ and $|\Psi(g_1h)|_{R_z}=|\Psi(g_1)|_{R_z}=|g_1|_{R_{yz}}=|g_0|_{S_1} $. Therefore, 
$$|g_0h|_{S_1}\geq |g_0|_{S_1} \text{ and } |g_1h|_{S_1}\geq |g_1|_{S_1}.$$
\end{proof}

\begin{lem}\label{ll1}
Let $m\geq 2$ be an integer. Then the inclusion map $i\!:\Gamma(G_{m-1},S_{m-1})\hookrightarrow \Gamma(G_{m},S_{m})$ is an isometric embedding.
\end{lem}

\begin{proof}
For $m=2$ the group $G_2$ is an isometric HNN extension in the sense of \cite{MR1668335}, with base group $G_{1}$ and edge groups $F_{xz}$ and $F_{yz}$. By Lemma~\ref{ll0} two groups $F_{xz}$ and $F_{yz}$ are isometrically embedded into $G_1$ with the given generating sets. By Lemma 2.2(2) in \cite{MR1668335} (taking $G=H=G_{1}$) the inclusion $G_{1} \hookrightarrow G_2$ is an isometric embedding. 

For each $m\geq 3$ the group $G_m$ is also an isometric HNN extension, with base group $G_{m-1}$ and edge groups $\langle s_1\rangle$ and $\langle s_{m-1}\rangle$. The homomorphism $G_{m-1} \to {\mathbb Z}$ taking all the $s_j$ ($1\leq j \leq m-1$) to a generator of ${\mathbb Z}$ and all other generators of $G_{m-1}$ to the identity shows that each $\langle s_j\rangle$ is a retract of $G_{m-1}$ and so are isometrically embedded subgroups. By Lemma 2.2(2) in \cite{MR1668335} again (taking $G=H=G_{m-1}$) the inclusion $G_{m-1} \hookrightarrow G_m$ is an isometric embedding. 
\end{proof}





The following lemma will be used to define the concepts of $1$--rays, $(1,k)$--rays, and $k$--corners (see Definition~\ref{fcner} and Definition~\ref{corner2}) which appear in the proof of the lower bound of group divergence in Theorem~\ref{main_thm}.

\begin{lem}
\label{lcool}
Let $1\leq i< j \leq m$ be integers. Let $p,q$ be arbitrary integers. Then $|s_i^{p} s_j^{q}|_{S_m}=|p|+|q|$.
\end{lem}

\begin{proof}

There is a group homomorphism $\Psi\!: G_{j} \to \mathbb{Z}^2$ taking $s_{j}$ to the generator $(0,1)$, taking each $s_k$ to the generator $(1,0)$ for $1\leq k \leq j-1$, and taking each element in $S_{j}-\{s_1,s_2,\cdots,s_{j}\}$ to the identity $(0,0)$. Therefore, $\Psi(s_i^{p} s_{j}^{q})=(p,q)$. This implies that
$$|s_i^{p} s_{j}^{q}|_{S_{j}}\geq |\Psi(s_i^{p} s_{j}^{q})|=|p|+|q|.$$ 

Since $s_i$ and $s_j$ are also elements in the finite generating set $S_j$, we have $$|s_i^{p} s_{j}^{q}|_{S_{j}}=|p|+|q|.$$ 

Also, the inclusion map $\Gamma(G_{j},S_{j})\hookrightarrow \Gamma(G_{m},S_{m})$ is an isometric embedding by Lemma~\ref{ll1}. Therefore, $|s_i^{p} s_{j}^{q}|_{S_{m}}=|s_j^{p} s_{j}^{q}|_{S_{j}}=|p|+|q|$.
\end{proof}

The next lemma is used in Lemma~\ref{coolhyp} to help determine the structure of $k$--corners.

\begin{lem}
\label{lcool2}
Let $1\leq i< j <k\leq m$ be integers. Let $p,q,r$ be arbitrary integers such that $s_i^{p}s_j^{q}s_k^{r}=e$ in $G_m$. Then $p=q=r=0$.
\end{lem}

\begin{proof}
We first prove that $r=0$. In fact, there is a group homomorphism $\Phi\!: G_{k} \to \mathbb{Z}$ taking $s_{k}$ to $1$ and each generators in $S_k-\{s_k\}$ to $0$. Therefore, $r=\Phi(s_i^{p}s_j^{q}s_k^{r})=\Phi(e)=0$. Thus, $s_i^{p}s_j^{q}=e$ which implies that $p=q=0$ by Lemma~\ref{lcool}. 
\end{proof}

\begin{rem}
\label{bbb}
Recall from Definition~\ref{gene} that $F$ is a free subgroup of rank $p$ of $H$ with a finite generating set $R=\set{d_i}{1\leq i \leq p}$ and that $R$ is a subset of the finite generating set $T$ of $H$. We choose the distortion $\Dist_F^H$ as in Lemma~\ref{hohohhh} and let $f=(\Dist_F^H)^{-1}$. Then it is clear that $f(r)\leq r$ for each $r\geq 1$. By Lemma~\ref{hohohhh} we also have
\begin{enumerate}
    \item $f\bigl(|g|_R\bigr)\leq |g|_T$ for each $g\in F$; and
    \item There is constants $D>1$ and $r_0\geq 1$ such that for each $r\geq r_0$ there is a palindromic group element $u$ satisfying $r/D\leq |u|_R\leq r$ and $|u|_T\leq Df(r)$. 
\end{enumerate}
We will fix the positive numbers $r_0$, $D$, and the function $f=(\Dist_F^H)^{-1}$ in the above statement for the remaining of the paper.
\end{rem}




Two elements $g_0 \in F_{xz}$ and $g_1 \in F_{yz}$ determine a pair of geodesic edge-paths based at the identity (one from $e$ to $g_0$ and the other from $e$ to $g_1$); we call this pair of edge-paths a \emph{corner}. The next lemma establishes a lower bound on the divergence of such corners in the group $G_1$.

\begin{lem}
\label{st1}
Let $g_0$ be a group element in $F_{xz}$ and let $g_1$ be a group element in $F_{yz}$. Then $$|g_1^{-1}g_0|_{S_1}\geq f\bigl(|g_0|_{S_1}\bigr).$$
\end{lem}

\begin{proof}
There is a group homomorphism $\Psi\!: G_{1} \to H$ taking each $x_{i}$ to $d_i$, taking each $y_i$, $z_i$ and $s_1$ to the identity $e$, and taking each element in $T$ to itself. Thus, $\Psi$ takes each $a_i$ to $d_i$ and takes each $b_i$ to the identity $e$. Therefore, $\Psi(g_1^{-1}g_0)=\Psi(g_0)$. Also $\Psi|F_{xz}$ is a monomorphism and $\Psi(F_{xz})=F$. Therefore, $|\Psi(g_0)|_R=|g_0|_{R_{xz}}$. We remind the reader that $R=\set{d_i}{1\leq i \leq p}$ is a finite generating set $F$ and $R_{xz}=\set{a_i}{1\leq i \leq p}$ is a finite generating set $F_{xz}$. Since the finite generating set $S_1$ of $G_1$ containing the finite generating set $R_{xz}$ of $F_{xz}$, we have $|\Psi(g_0)|_R=|g_0|_{R_{xz}}\geq|g_0|_{S_1}$. Also, $\Psi(g_0)$ is a group element in $F$. Therefore, we have $$|\Psi(g_0)|_T\geq f\bigl(|\Psi(g_0)|_R\bigr)\geq f\bigl(|g_0|_{S_1}\bigr).$$
This implies that
$$|g_1^{-1}g_0|_{S_1}\geq |\Psi(g_1^{-1}g_0)|_T=|\Psi(g_0)|_T\geq f\bigl(|g_0|_{S_1}\bigr).$$
\end{proof}

 The next lemma establishes the existence of corners in $G_1$ whose divergences agree with the lower bounds of Lemma~\ref{st1}.

\begin{lem}
\label{a1}
For each positive integer $r\geq r_0$ there are elements $g_0$ in $F_{xz}$ and $g_1$ in $F_{yz}$ such that the following hold:
\begin{enumerate}
    \item $s_2^{-1}g_0s_2=g_1$; 
    \item $r/D\leq |g_0|_{R_{xz}}\leq r$ and $r/D\leq |g_1|_{R_{yz}}\leq r$;
    \item There is a path $\gamma$ connecting $g_0$ and $g_1$ with length at most $Df(r)$ and each vertex in $\gamma$ is a group element $g_0h$ for $h\in H$; in particular, $\gamma$ avoids the $r/D$--ball about $e$.
\end{enumerate}
\end{lem}

\begin{proof}
Let $u$ be a palindromic group element in $F$ such that $r/D\leq |u|_R\leq r$ and $|u|_T\leq Df(r)$. Then $u$ is represented by a reduced palindromic word $w$ as follows
$$w=d_{i_1}^{\epsilon_1}d_{i_2}^{\epsilon_2}\cdots d_{i_\ell}^{\epsilon_\ell} \text{ where } \epsilon_i\neq 0 \text{ and }|\epsilon_1|+|\epsilon_2|+\cdots+|\epsilon_\ell|=|u|_R. $$
We define

$$g_0=a_{i_1}^{\epsilon_1}a_{i_2}^{\epsilon_2}\cdots a_{i_\ell}^{\epsilon_\ell} \text{ and } g_1=b_{i_1}^{\epsilon_1}b_{i_2}^{\epsilon_2}\cdots b_{i_\ell}^{\epsilon_\ell}. $$
Then $g_0$ and $g_1$ are group elements in $F_{xz}$ and $F_{yz}$ respectively, and
$$r/D\leq |g_0|_{R_{xz}}=|g_1|_{R_{yz}}=|u|_R\leq r.$$
Since $s_2^{-1}a_{i_j}s_2=b_{i_j}$ for $1\leq j \leq \ell$, we have $s_2^{-1}g_0s_2=g_1$. 

We note that $x_{i_j}$ commutes $y_{i_k}$ for $1\leq k,l\leq \ell$ and $d_{i_j}=x_{i_j}y_{i_j}^{-1}$ for $1\leq j \leq \ell$. Therefore, 
$$u=d_{i_1}^{\epsilon_1}d_{i_2}^{\epsilon_2}\cdots d_{i_\ell}^{\epsilon_\ell}=(x_{i_1}y_{i_1}^{-1})^{\epsilon_1}(x_{i_2}y_{i_2}^{-1})^{\epsilon_2}\cdots (x_{i_\ell}y_{i_\ell}^{-1})^{\epsilon_\ell}=(x_{i_1}^{\epsilon_1}x_{i_2}^{\epsilon_2}\cdots x_{i_\ell}^{\epsilon_\ell})(y_{i_1}^{-\epsilon_1}y_{i_2}^{-\epsilon_2}\cdots y_{i_\ell}^{-\epsilon_\ell}).$$
Similarly, we also have 
$$g_0=(x_{i_1}^{\epsilon_1}x_{i_2}^{\epsilon_2}\cdots x_{i_\ell}^{\epsilon_\ell})(z_{i_1}^{\epsilon_1}z_{i_2}^{\epsilon_2}\cdots z_{i_\ell}^{\epsilon_\ell}) \text{ and } g_1=(y_{i_1}^{\epsilon_1}y_{i_2}^{\epsilon_2}\cdots y_{i_\ell}^{\epsilon_\ell})(z_{i_1}^{\epsilon_1}z_{i_2}^{\epsilon_2}\cdots z_{i_\ell}^{\epsilon_\ell}).$$
Since $y_{i_1}^{\epsilon_1}y_{i_2}^{\epsilon_2}\cdots y_{i_\ell}^{\epsilon_\ell}$ is a palindromic word, we have
$(y_{i_1}^{\epsilon_1}y_{i_2}^{\epsilon_2}\cdots y_{i_\ell}^{\epsilon_\ell})(y_{i_1}^{-\epsilon_1}y_{i_2}^{-\epsilon_2}\cdots y_{i_\ell}^{-\epsilon_\ell})=e$.
This implies that $g_0=g_1u$ or $g_1=g_0u^{-1}$.

Let $\gamma_0$ be a geodesic in the Cayley graph $\Gamma(H,T)\subset \Gamma(G_1,S_1)$ connecting $e$ and $u^{-1}$. Then 
$$\ell(\gamma_0)=|u|_T\leq Df(r).$$ Moreover, the path $\gamma=g_0\gamma_0$ connects $g_0$ and $g_1$, the length of $\gamma$ is at most $Df(r)$, and each vertex in $\gamma$ is a group element $g_0h$ for $h\in H$. By Lemma~\ref{ll0} the path avoid $\gamma$ is $r/D$--ball about $e$. 

\end{proof}

The last lemma in this section is very straightforward; it will be used multiple times in section~\ref{sec:four}.

\begin{lem}
\label{s1}
Let $m$ be a positive integer and let $n$ be non-zero integer. Then 
\begin{enumerate}
    \item There is a path in $\Gamma(G_m,S_m)$ which connects $s_1^{-n}$ and $s_1^{n}$, lies outside the open ball $B(e,|n|)$, and has length at most $4|n|$.
    \item For each element $s\in S_1\cup S_1^{-1}$, two group elements $s_1^{2n}$ and $ss_1^{2n}$ can be connected by an edge labeled $s$ in $\Gamma(G_m,S_m)$.
\end{enumerate}
\end{lem}

\begin{proof}
By Lemma~\ref{ll1} we only need to prove Statement (1) for $m=1$. 
In fact, Statement (1) for $m=1$ can be proved from the fact the $\Z^2$-group generated by $a_1$ and $s_1$ is isometrically embedded in $G_1$. 
Statement (2) holds due to the fact that $s_1$ and $s$ commutes.
\end{proof}

\section{Divergence in $G_m$}
\label{sec:four}

The purpose of this section is twofold. In the first two subsections upper and lower bounds for the divergence of $G_m$ are established, thus proving the claims in the right column of the table in Theorem~\ref{main_thm}. The third subsection is devoted to establishing the divergence of the cyclic subgroups $\langle s_m\rangle$ and $\langle s_1s_2\rangle$ in $G_m$. The former completes the proof of Theorem~\ref{main_thm} and is also used in subsection~\ref{sec:fivethree} to prove Proposition~\ref{prop:contract}. In the case $F$ has nonlinear distortion in $H$, the latter is used to prove that the $G_m$ are not $\CAT(0)$ groups for $m \geq 2$. 

\subsection{The upper bound}

In this subsection we establish upper bounds for the divergence of $G_m$. A fundamental observation (Lemma~\ref{ll1}) is that the group $G_i$ isometrically embeds into $G_{i+1}$. Therefore paths which avoid the open $r$--ball (are open $r$--ball avoidant) about the identity in $G_i$ remain open $r$--ball avoidant in $G_{i+1}$. This allows us to prove upper bounds on the divergence of $G_m$ by induction on $m$. 

The following four lemmas will be used in the proof of the upper bound of the group divergence in Theorem~\ref{main_thm}. The upper bounds are established in Proposition~\ref{upperkey} for $m \geq 3$ and in Proposition~\ref{upperkey2} for $m=2$. 

\begin{lem}[see Lemma 4.1 in \cite{BT2020}]
\label{basic}
Let $X$ be a geodesic space and $x_0$ be a point in $X$. Let $r$ be a positive number and let $x$ be a point on the sphere $S(x_0,r)$. Let $\alpha_1$ and $\alpha_2$ be two rays with the same initial point $x$ such that $\alpha_1\cup\alpha_2$ is a bi-infinite geodesic. Then either $\alpha_1$ or $\alpha_2$ has an empty intersection with the open ball $B(x_0,r/2)$.
\end{lem}

The second lemma provides an upper bound on the divergence of $s_1, s_2$ corners in $G_2$.

	\begin{figure}[h]
		\centering
		\begin{tikzpicture}[scale=0.95]
		\filldraw[black] (0,0) circle (2pt) node[xshift=0em, yshift=-0.8em]{$e$};
		\filldraw[black] (10,0) circle (2pt) node[xshift=0.6em, yshift=-0.8em]{$s_2^{n_2}$};
		\filldraw[black] (-4,3) circle (2pt) node[xshift=-0.7em, yshift=0em]{$s_1^{n_1}$};
		\filldraw[black] (11,4) circle (2pt);
		\filldraw[black] (10.5,4) circle (1pt);
		\filldraw[black] (6.5,4) circle (1pt);
		\filldraw[black] (6,4) circle (1pt);
		\filldraw[black] (3.5,4) circle (1pt);
		\filldraw[black] (4,4) circle (1pt);
		\filldraw[black] (-0.5,4) circle (1pt);
		\filldraw[black] (-1,4) circle (1pt);
		\filldraw[black] (0.5,0) circle (1pt);
		\filldraw[black] (4.5,0) circle (1pt);
		\filldraw[black] (5,0) circle (2pt) node[xshift=-0.05em, yshift=-0.8em]{$s_2^i$};
		\filldraw[black] (5.5,0) circle (1pt);
		\filldraw[black] (9.5,0) circle (1pt);
		 \draw[thick, blue] (6,4) arc (45:135:1.4142135) node[midway, black, yshift=0.6em]{$\beta_i$};
	 	 \draw[thick, blue] (1.5,4) arc (45:135:1.4142135) node[midway, black, yshift=0.6em]{$\beta_1$};
		  \draw[thick, blue] (-1,4) arc (90:126.5:5) node[midway, black, xshift=-0.4em, yshift=0.4em]{$\beta$};
		\draw[thick] (0, 0) -- (10, 0);
		
		\draw[thick] (0.5, 0) -- (-0.5, 4) node[midway, xshift=0.7em]{$g_1$};
		\draw[thick] (-1,4) -- (0,0) node[midway, xshift=-0.6em]{$g_0$};
		\draw[thick, blue] (-0.5, 4) -- (-1,4) node[midway, black,  yshift=0.6em]{$e_1$};
		
		\draw[thick, blue] (3.5, 4) -- (4,4) node[midway, black,  yshift=0.6em]{$e_i$};
		\draw[thick, blue] (6, 4) -- (6.5,4) node[midway, black, yshift=0.6em]{$e_{i+1}$};
		\draw[thick, blue] (10.5, 4) -- (11,4) node[midway, black,  yshift=0.6em]{$e_r$};
		\draw[thick] (4.5, 0) -- (3.5, 4) node[midway, xshift=-0.6em, yshift=0.3em]{$g_0$};
		\draw[thick] (4,4) -- (5,0) node[midway, xshift=0.7em, yshift=0.3em]{$g_1$};
		
		\draw[thick] (6.5,4) -- (5.5,0) node[midway, xshift=0.7em, yshift=0.3em]{$g_1$};
		\draw[thick] (5, 0) -- (6, 4) node[midway, xshift=-0.6em, yshift=0.3em]{$g_0$};
		
		\draw[thick] (0, 0) -- (-4, 3);
	
		\draw[thick] (10.5,4) -- (9.5, 0) node[midway, xshift=-0.6em, yshift=0.3em]{$g_0$};
		\draw[thick, blue] (10, 0) -- (11, 4) node[midway, black, xshift=0.6em, yshift=0.3em]{$\gamma$};
		\draw[dashed] (-1,4.1) -- (-1, 5.5);
		\draw[dashed] (11, 4.1) -- (11, 5.5);
		\draw [decorate,decoration={brace,amplitude=5pt,raise=3ex}] (-1,5) -- (11,5) node[midway,yshift=2.4em]{$\gamma'$};
		\end{tikzpicture}
		\caption{The $r$--ball avoiding path (in blue) in the proof of Lemma~\ref{cool2} }
		\label{fig:avoid}
	\end{figure}
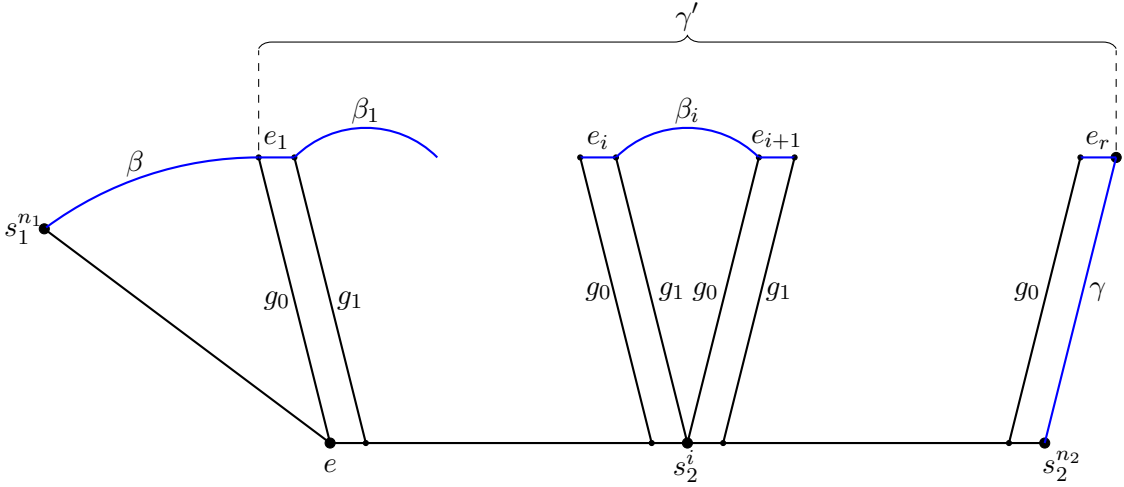

\begin{lem}\label{cool2} Let $r_0 \geq 1$ be the constant in Remark~\ref{bbb}. 
There is a positive constant $N_2$ with the following property. 
For all integers $r\geq r_0$, there exist paths in $\Gamma(G_2,S_2)$ connecting each of $s_1^{\pm r}$ to each of $s_2^{\pm r}$ which lie outside of the open ball $B(e,r)$ and whose lengths are bounded above by $N_2r(f(N_2r)+1)$. 
\end{lem}

\begin{proof} For concreteness, we connect $s_1^{n_1}$ to $s_2^{n_2}$ where $|n_1|=|n_2| = r$ by paths avoiding the open ball $B(e,r)$ in the case $n_2>0$. 
 The path $\overline{\gamma}$ which avoids the $r$--ball is constructed from three parts: $\overline{\gamma} = \beta \cup \gamma' \cup \gamma$. The reader can refer to Figure~\ref{fig:avoid} while reading this proof. 

\smallskip

\noindent
\emph{The sub-path $\gamma'$.} 
Let $D$ be the constant in Remark~\ref{bbb} and let $m=4Dr$. By Lemma~\ref{a1}, there are group elements $g_0$ in $F_{xz}$ and $g_1$ in $F_{yz}$ such that the following hold:
\begin{enumerate}
    \item $s_2^{-1}g_0s_2=g_1$; 
    \item $m/D\leq |g_0|_{R_{xz}}\leq m$ and $m/D\leq |g_1|_{R_{yz}}\leq m$;
    \item There is a path $\gamma$ connecting $g_0$ and $g_1$ with the length at most $Df(m)$ and each vertex in $\gamma$ is a group element $g_0h$ for $h\in H$;
\end{enumerate}
By Lemma~\ref{ll0} and Lemma~\ref{ll1} we have $|g_0|_{R_{xz}} = |g_0|_{S_1}=|g_0|_{S_2}$ and $|g_1|_{R_{yz}} = |g_1|_{S_1}=|g_1|_{S_2}$. Since $m/D=4r$, we have
$$4r\leq |g_0|_{S_2}\leq 4Dr \text{ and } 4r\leq |g_1|_{S_2}\leq 4Dr.$$
Another application of Lemma~\ref{ll0} and Lemma~\ref{ll1} implies that the path $\gamma$ lies outside the open ball $B(e,4r)$ in the group $G_2$.

For $i\in\{1,2,\cdots,r\}$ let $x_i=s_2^{i-1}g_0$ and let $y_i=s_2^{i}g_1$. Since $s_2^{-1}g_0s_2=g_1$, we have

$$y_i=s_2^{i}g_1=s_2^{i-1}(s_2g_1)=s_2^{i-1}(g_0s_2)=(s_2^{i-1}g_0)s_2=x_is_2.$$
Therefore, we can connect $x_i$ and $y_i$ by an edge $e_i$ labeled by $s_2$. Since $$d_{S_2}(y_i,s_2^i)=|g_1|_{S_2}\geq 4r \text{ and } d_{S_2}(e,s_2^i)=i\leq r,$$
the edge $e_i$ lies outside the open ball $B(e,3r-1)$ (therefore outside the open ball $B(e,r)$).


Since $y_i=s_2^{i}g_1$ and $x_{i+1}=s_2^{i}g_0$ for each $i\in\{1,2,\cdots,i-1\}$,
the path $\beta_i=s_2^i\gamma$ (by which we mean the left translate of the path $\gamma$ by $s_2^i$) lies outside the open ball $B(s_2^i,4r)$, connects $y_i$ and $x_{i+1}$, and has the length at most $Df(m)$. Also, since $d_{S_2}(e,s_2^i)=i\leq r$, each path $\beta_i$ lies outside the open ball $B(e,3r)$ and hence outside of $B(e,r)$. Let $$\gamma'=(e_1\cup \beta_1)\cup (e_2\cup \beta_2)\cup \cdots \cup (e_{r-1}\cup \beta_{r-1})\cup e_r.$$ Then $\gamma'$ is a path outside the open ball $B(e,r)$ that connect $x_1=g_0$ and $y_{n_2}=s_2^{n_2}g_1$ and has the length at most $r\bigl(D f(m)+1\bigr)$.

\smallskip

\noindent
\emph{The sub-path $\beta$.} 
Let $m_1=|g_0|_{R_{xz}}$. Then $$4r= m/D\leq m_1\leq m.$$
Since $g_0$ is a group element in $F_{xz}$ and subgroup $F_{xz}\times \langle s_1 \rangle$ is isometrically is embedded into $G_2$, there is a path $\beta$ outside the open ball $B(e,r)$ which connects $s_1^{n_1}$ and $g_0$ and has length exactly $m_1+|n_1|$. 

\smallskip

\noindent
\emph{The sub-path $\eta$.} 
Let $\eta$ be the path labeled by the reduced word in $R_{yz}$ representing $g_1$ which connects $s_2^{n_2}$ and $y_{n_2}$. Then the length of $\eta$ is exactly $|g_1|_{R_{yz}}\leq m$. We observe that each vertex in $\eta$ is a group element $s_2^{n_2}u$ for some $u\in F_{yz}$. Also the group automorphism $G_{2} \to {\mathbb Z}$ taking all the $s_2$ to a generator of ${\mathbb Z}$ and all other generators of $G_{2}$ to the identity shows that $s_2^{n_2}u$ lies outside the open ball $B(e,n_2)$(=$B(e,r)$). In other words, $\eta$ lies outside the open ball $B(e,r)$.

Let $\bar{\gamma}=\beta\cup \gamma'\cup \eta$. Then $\bar{\gamma}$ is a path outside the open ball $B(e,r)$ that connects $s_1^{n_1}$ and $s_2^{n_2}$. Moreover, 
\begin{align*}
    \ell(\bar{\gamma})&=\ell(\beta)+\ell(\gamma')+\ell(\eta)\\&\leq (m_1+|n_1|)+r\bigl(D f(m)+1\bigr)+m\\&\leq (4Dr+r)+r\bigl[D f\bigl(4Dr\bigr)+1\bigr]+(4Dr).
\end{align*}
Since $f$ is a non-decreasing function, we can choose a number $N_2$ depending on $D$ such that the length of $\bar{\gamma}$ is at most $N_2r\bigl(f(N_2r)+1\bigr)$.

\end{proof}

\begin{lem}\label{cool3}
Let $n$ and $r$ be non-zero integers such that $|n|=r$ and let $\epsilon \in \{-1,+1\}$. Then there is a path in $\Gamma(G_2,S_2)$ which connects $s_1^{2n}$ and $s_2^{\epsilon}s_1^{2n}$, lies outside the open ball $B(e,r)$, and has length at most $8r+1$.
\end{lem}
\begin{proof}
We prove the lemma for the case $\epsilon=1$; the proof for the case $\epsilon=-1$ is similar. Since $s_2b_1^{2n}=a_1^{2n}s_2$, we can connect $a_1^{2n}$ and $s_2b_1^{2n}$ by an edge $e_1$ labeled by $s_2$ in $\Gamma(G_2,S_2)$. Since the $\Z^2$-group generated by $a_1$ and $s_1$ is isometrically embedded into $G_2$, we can connect $s_1^{2n}$ and $a_1^{2n}$ by a path $\alpha$ outside the open ball $B(e,2r)$ with length at most $4|n|=4r$. Similarly, we can connect $s_1^{2n}$ and $b_1^{2n}$ by a path $\beta_0$ outside the open ball $B(e,2r)$ with length at most $4r$. Therefore, $\beta=s_2\beta_0$ is a path outside the open ball $B(e,r)$ that connects $s_2s_1^{2n}$ and $s_2b_1^{2n}$ and has the length at most $4r$. This implies that $\gamma=\alpha \cup e_1 \cup \beta$ is a path outside the open ball $B(e,r)$ that connects $s_1^{2n}$ and $s_2s_1^{2n}$ and has the length at most $8r+1$.
\end{proof}


Say that a path is \emph{open $r$--ball avoidant} if it does not intersect the open ball $B(e,r)$. 

\begin{lem}
\label{ag}
For each integer $m\geq 3$ there are constants $M_m$ and $N_m$ such that the following two statements hold:
\begin{itemize}
    \item[$(P_m)$] Given integers $n$ and $r$ satisfying $|n|=r\geq r_0$ and $\epsilon \in \{-1, 1\}$, there exists an open $r$-ball avoidant path in $\Gamma(G_m,S_m)$ connecting $s_1^{2n}$ and $s_m^{\epsilon}s_1^{2n}$ which has length at most $M_mr^{m-2} \bigl(f(M_mr)+1\bigr)$.
    
    \item[$(Q_m)$] Given integers $n_1$, $n_2$ and $r$ satisfying $|n_1|=|n_2|=r\geq r_0$, there exists an open $r$--ball avoidant path in $\Gamma(G_m,S_m)$ connecting $s_1^{n_1}$ and $s_m^{n_2}$ which has length at most $N_mr^{m-1} \bigl(f(N_mr)+1\bigr)$.

\end{itemize}
\end{lem}

\begin{figure}[h]
	\centering
	\begin{tikzpicture}[scale=0.95]
		\filldraw[black] (0,0) circle (2pt) node[xshift=-0.8em, yshift=0em]{$e$};
		\filldraw[black] (0.7071067,0.7071067) circle (2pt) node[xshift=-1.2em, yshift=-0.2em]{$s_1^n$};
		\filldraw[black] (0.7071067,-0.7071067) circle (2pt) node[xshift=0.6em, yshift=-0.8em]{$s_m^n$};
		
		\draw[dashed] (0,0) circle [radius=1];
		\draw[dashed] (0,0) circle [radius=2];

		\draw[thick]  (0,0) -- (0.7071067,0.7071067);  
		\draw[thick, blue]  (0.7071067,0.7071067)  --  (2.828427, 2.828427);
		\filldraw[black]  (2.828427, 2.828427) circle (2pt) node[xshift=-1.2em, yshift=-0.2em]{$s_1^{4n}$};
		
		\draw[thick] (0,0) --  (0.7071067,-0.7071067); 
		
		\draw[thick, blue]  (0.7071067,-0.7071067) -- (0.7071067+2.828427+0.8,-0.7071067+2.828427-0.4); 
		\draw[thick]  (0.7071067/4,-0.7071067/4) -- (0.7071067/4+2.828427+0.3,-0.7071067/4+2.828427-0.2); 
		\draw[thick]  (3*0.7071067/4,-3*0.7071067/4) -- (3*0.7071067/4+2.828427+0.5,-3*0.7071067/4+2.828427-0.3); 
		\filldraw[black] (0.7071067+2.828427+0.8,-0.7071067+2.828427-0.4) circle (2pt) node[xshift=0.6em, yshift=-0.8em]{$s_m^ns_1^{4n}$};
		
		 \draw[thick, blue]  (2.828427, 2.828427) .. controls (1.2, 1) .. (0.7071067/4+2.828427+0.3,-0.7071067/4+2.828427-0.2);
		 \draw[thick, blue]  (3*0.7071067/4+2.828427+0.5,-3*0.7071067/4+2.828427-0.3) .. controls (2.0, 0.4) .. (0.7071067+2.828427+0.8,-0.7071067+2.828427-0.4);
		\draw[blue] (0.7071067/4+2.828427+0.3,-0.7071067/4+2.828427-0.2)  node[xshift=3.4em]{\scriptsize{$(P_m)$ estimates}};
		\draw[black] (-1.0, 3) node{$(P_m) \implies (Q_m)$};
		
		\draw[black] (-7.5, 3) node{$(Q_m) \implies (P_{m+1})$};
		
		\draw[thick]  (-8,0) -- (-8+2*0.7071067,2*0.7071067);  
		\draw[thick]  (-8,0) -- (-8, 2);
		\draw[thick]  (-8,0) --  (-8-0.5,0) -- (-8-0.5, 2);
		\draw[thick, blue] (-8-0.5, 2) -- (-8, 2);
		
		\filldraw[black] (-8.5,0) circle (2pt) node[xshift=-0.6em, yshift=-0.8em]{$e$};
		\filldraw[black] (-8,0) circle (2pt) node[xshift=0.4em, yshift=-0.8em]{$s_{m+1}^\epsilon$};
		\filldraw[black] (-8.5,2) circle (2pt) node[xshift=-1.0em, yshift=-0.8em]{$s_1^{2n}$};	
		\filldraw[black] (-8+2*0.7071067,2*0.7071067) circle (2pt) node[xshift=1.4em, yshift=-0.8em]{$s_{m+1}^\epsilon s_1^{2n}$};
		
		\draw[thick, blue]  (-8+2*0.7071067,2*0.7071067) .. controls (-8 + 0.35,0.7) .. (-8, 2);
		\draw[dashed] (-8.5,0) circle [radius=1.2];
		
		\draw[blue] (-7.5, 2) node[xshift=2.5em]{\scriptsize{$(Q_m)$ estimate}};
		
	\end{tikzpicture}
\caption{Inductive steps in proof of Lemma~\ref{ag}}
\label{fig:lem44}
\end{figure}
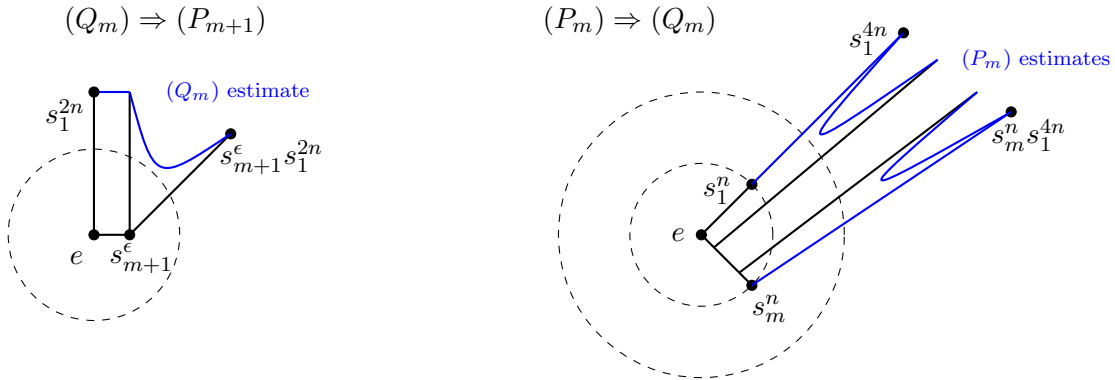

\begin{proof}
We prove the above lemma by induction on $m$ using an analogous strategy as in the proof of Lemma 4.2 in \cite{BT2020}. We first prove Statement $P_3$. Then we prove the implication $(P_m)\implies (Q_m)$ for each $m\geq 3$. Finally, we prove the implication $(Q_m)\implies (P_{m+1})$ for each $m\geq 3$. 

The proofs of the implications $(P_m)\implies (Q_m)$ for each $m\geq 3$ and $(Q_m)\implies (P_{m+1})$ for each $m\geq 3$ are analogous to the proofs of the corresponding implications in Lemma 4.2 in \cite{BT2020} and we leave these to the reader. The diagrams in Figure~\ref{fig:lem44} are included to facilitate the translation of the proofs of the corresponding statements in Lemma~4.2 of \cite{BT2020} to the current setting. 

We now prove Statement $(P_3)$ for the case of $\epsilon=1$; the proof for the case of $\epsilon=-1$ is very similar. Since $s_3s_2^{2n}=s_1^{2n}s_3$, we can connect $s_1^{2n}$ and $s_3s_2^{2n}$ by an edge $e_1$ outside the open ball $B(e,r)$. By Lemma~\ref{cool2} there is a path $\gamma_0$ outside the open ball $B(e,2r)$ that connects $s_2^{2n}$ and $s_1^{2n}$ and has length at most $2N_2r\bigl(f(2N_2r)+1\bigr)$. Therefore, $\gamma_1=s_3\gamma_0$ is a path outside the open ball $B(e,r)$ that connects $s_3s_2^{2n}$ and $s_3s_1^{2n}$ and has length at most $2N_2r\bigl(f(2N_2r)+1\bigr)$. This implies that $\gamma=e_1\cup\gamma_1$ is a path outside the open ball $B(e,r)$ that connects $s_1^{2n}$ and $s_3s_1^{2n}$ and has length at most $2N_2r\bigl(f(2N_2r)+1\bigr)+1$. Therefore, $M_3=2N_2+1$ is a desired constant.
\end{proof}

The following two propositions give the upper bounds for the divergence of the group $G_m$ for $m\geq 2$. The proof of these propositions is analogous to the proof of Proposition 4.3 and Proposition 4.4 in \cite{BT2020}. We only need to replace Lemma~3.2, Lemma~3.7, Lemma~4.1, and Lemma~4.2 in \cite{BT2020} by Lemma~\ref{ll1}, Lemma~\ref{s1}, Lemma~\ref{basic}, Lemma~\ref{cool3}, and Lemma~\ref{ag}. We note that the linear upper bound in Lemma~\ref{cool3} in this paper is different from the sublinear upper bound of Statement ($P_2$) of Lemma~4.2 in \cite{BT2020}. However, this difference does not affect the proof of Proposition~\ref{upperkey} and Proposition~\ref{upperkey2} in this paper.

The following proposition establishes the upper bound for the divergence of the group $G_m$ for $m\geq 3$. 

\begin{prop}[Upper bound for $\Div_{G_m}$, $m \geq 3$]
\label{upperkey}
Let $\{\delta_\rho\}$ be the divergence of $\Gamma(G_m,S_m)$ for $m \geq 3$. Then
$$
\delta_\rho(r) \preceq r^{m-1}f(r)
$$
for each $\rho \in (0,1/2]$.

\end{prop}

\begin{figure}[h]
	\centering
		\begin{tikzpicture}[scale=0.6]
		\filldraw[black] (0,0) circle (2pt) node[xshift=0em, yshift=-0.5em]{$e$};
		\draw[thick] (0,0) circle [radius=1.5];
		\draw[thick] (0,0) circle [radius=3];
		
		\filldraw[black] (2.1213, 2.1213)  circle (2pt) node[xshift=-1.1em, yshift=-0.2em]{$s_1^r$};
		\filldraw[black] (-2.1213, -2.1213) circle (2pt) node[xshift=1.0em, yshift=0.3em]{$s_1^{-r}$};
		\filldraw[black] (8,8) circle (2pt) node[xshift=-1.1em, yshift=-0.2em]{$s_1^{4r}$};
		
		\draw[thick] (0,0) -- (2.1213, 2.1213);
		\draw[thick, blue] (2.1213,2.1213) -- (8,8);
		
		\draw[thick, blue] (2.2,2.4) arc [start angle=46.1, end angle=225, x radius=3.3, y radius=3.3] node[midway, black, xshift= -3.0em, yshift=0em]{\scriptsize{Length $\leq 4r$}}; 
		\draw[thick, blue] (2.1213, 2.1213)  -- (2.2,2.4); 
		\draw[thick, blue]  (-2.1213, -2.1213)  -- (-2.44,-2.31); 
		
		\draw[thick] (0,0) -- (3,0);
		\filldraw[black] (3,0) circle (2pt) node[xshift=0.6em, yshift=-0.6em]{$x$};

		 \draw[thick, blue] (3,0) .. controls (2,0.6) .. (2.8,1.077);
		 \draw[thick, blue]  (2.8,1.077) -- (14 , 5.385); 
		 \filldraw[black] (14 , 5.385) circle (2pt) node[xshift=1.1em, yshift=-0.2em]{$xs_1^{4r}$};
		 
		\draw[thick] (0.4,0) -- (8.8,7.6);
		
		 \draw[thick] (2.6,0).. controls (1.55,0.55) .. (2.7,1.3077);
		 \draw[thick]  (2.7,1.3077) -- (13 , 5.8); 
		
		 \draw[thick, blue]  (8,8) .. controls (5.4,5) .. (8.8,7.6);
		  \draw[thick, blue]  (9.6,7.2) .. controls (6.6,5) .. (8.8,7.6);
		    \draw[thick, blue] (13 , 5.8) .. controls (11, 4.5) .. (14 , 5.385);
		     \draw[thick, blue] (13 , 5.8) .. controls (9, 4.5) .. (12 , 6.3);
		\draw[black] (13.0, 7.0) node{\scriptsize{Bounded by Lemma 4.4}};
		
		\end{tikzpicture}
		\caption{Divergence upper bound in the case that the positive $s_1$-ray from $x$ avoids the open $r/2$--ball.}
		\label{fig:div}
\end{figure}

\begin{proof}
We will prove that there is a constant $A_m$ such that we have
 $$\delta_\rho(r)\leq A_m r^{m-1}\bigl(f(A_mr)+1\bigr)$$ 
 for each $\rho \in (0,1/2]$ and each $r\geq r_0$. Here $r_0$ is the constant from Remark~\ref{bbb}.
 
 We can assume that $r$ is an integer. By Lemma~\ref{s1}, there is a path which connects $s_1^{-r}$ and $s_1^{r}$, lies outside the open ball $B(e,r)$, and has length at most $4r$. It suffices to show there is a constant $B_m$ depending only on $m$ such that for each point $x$ on the sphere $S(e,r)$ we can either connect $x$ to $s_1^r$ or connect $x$ to $s_1^{-r}$ by a path outside the open ball $B(e,r/2)$ with the length at most $B_m r^{m-1}\bigl(f(B_mr)+1\bigr)$. Now the proposition follows choosing a suitable constant $A_m \geq 2B_m + 4$. 

Now we establish the $B_m r^{m-1}\bigl(f(B_mr)+1\bigr)$ bound. Let $\alpha$ be a bi-infinite geodesic which contains $x$ and has edges labeled by $s_1$. Then $\alpha$ is the union of two rays $\alpha_1$ and $\alpha_2$ that share the initial point $x$. Assume that $\alpha_1$ traces each edge of $\alpha$ in the positive direction and $\alpha_2$ traces each edge of $\alpha$ in the negative direction. By Lemma~\ref{basic}, either $\alpha_1$ or $\alpha_2$ (say $\alpha_1$) lies outside the open ball $B(e,r/2)$. We now construct a path $\eta$ which connects $x$ and $s_1^r$, lies outside the open ball $B(e,r/2)$, and has the length at most $B_m r^{m-1}\bigl(f(B_mr)+1\bigr)$. We note that we can use an analogous argument to construct a similar path connecting $x$ and $s_1^{-r}$ in the case $\alpha_2$ lies outside the open ball $B(e,r/2)$.

\noindent
{\em Constructing the comb and connecting endpoints of successive teeth.} First, we connect $s_1^r$ and $s_1^{4r}$ by the geodesic $\eta_0$ labeled by $s_1^{3r}$. Then $\eta_0$ lies outside the open ball $B(e,r)$ and has length exactly $3r$. Since $|x|_{S_m}=r$, we can write $x=u_1u_2\cdots u_{r-1}u_r$ where $u_i\in S_m\cup S_m^{-1}$. Let $R_m\geq 17$ be a constant which is greater than $4^m$ times of each constant $M_j$ for $3\leq j \leq m$ in Lemma~\ref{ag}. By Lemma~\ref{ll1}, Lemma~\ref{s1}, Lemma~\ref{cool3}, and Lemma~\ref{ag} we can connect $s_1^{4r}$ and $u_1s_1^{4r}$ by a path $\eta_1$ which lies outside the open ball $B(e,2r)$ and has length at most $R_m r^{m-2}\bigl(f(R_mr)+1\bigr)$.

Similarly for $2\leq i\leq r$ we can also connect $(u_1u_2\cdots u_{i-1})s_1^{4r}$ and $(u_1u_2\cdots u_{i})s_1^{4r}$ by a path $\eta_i$ which lies outside the open ball $B(u_1u_2\cdots u_{i-1},2r)$ and has length at most $R_m r^{m-2}\bigl(f(R_mr)+1\bigr)$. Finally, we connect $xs_1^{4r}$ and $x$ using the subsegment $\eta_{r+1}$ of $\alpha_1$. Then $\eta_{r+1}$ lies outside the open ball $B(e,r/2)$ and has length exactly $4r$. 

Let $\eta=\eta_0\cup\eta_1\cup\eta_2\cup\cdots\cup \eta_r\cup \eta_{r+1}$. Then the path $\eta$ connects $s_1^r$ and $x$, lies outside the open ball $B(e,r/2)$, and has length at most $R_m r^{m-1}\bigl(f(R_mr)+1\bigr)+7r$. Since $f$ is a non-decreasing function, the length of $\eta$ is bounded above by $B_m r^{m-1}\bigl(f(B_mr)+1\bigr)$ by some appropriate choice of $B_m \geq R_m + 7$.
\end{proof}

The following proposition gives the quadratic upper bound for the divergence of the group $G_2$. The proof of this proposition proceeds exactly as the proof of Proposition~\ref{upperkey} and we leave it to the reader. The main difference is that we only need the upper bounds $16r+1$ on lengths on paths avoiding $B(e, 2r)$ which connect $s_1^{4r}$ with $s_2s_1^{4r}$ or $s_2^{-1}s_1^{4r}$ from Lemma~\ref{cool3}. The other length estimates are $s$-edges outside $B(e,2r)$ which connect $s_1^{4r}$ with $ss_1^{4r}$ for $s$ a generator of $G_1$ (see Lemma~\ref{s1}).

\begin{prop}[Quadratic upper bound for $\Div_{G_2}$]
\label{upperkey2}
Let $\{\delta_\rho\}$ be the divergence of $\Gamma(G_2,S_2)$. 
Then $$\delta_\rho(r) \preceq r^2$$
for each $\rho \in (0,1/2]$.
\end{prop}

Finally, we note that the groups $G_m$ are one-ended. The proof of this statement is exactly as the proof of Proposition 4.5 in \cite{BT2020}.

\begin{prop}
\label{one-ended}
For each $m\geq 1$ the group $G_m$ is one-ended.
\end{prop}

\begin{proof}
$G_1$ is one-ended because it is a direct product of two infinite groups.
For $m \geq 2$ note that (by retracting along geodesic rays based at $e$) it is sufficient to prove that any two points on the sphere $S(e,r)$ can be connected by a path avoiding $B(e,r/2)$ for large integers $r$. In the proof of Proposition~\ref{upperkey} and Proposition~\ref{upperkey2} we construct paths which avoid $B(e,r/2)$ and which connect an arbitrary $x\in S(e,r)$ to one of $s_1^r$
or $s_1^{-r}$, and we can connect $s_1^r$ to $s_1^{-r}$ by a path outside of $B(e,r/2)$ in the $a_1s_1$-plane. Therefore, the group $G_m$ is one-ended.
\end{proof}

\subsection{The lower bound}
\label{lbfdg}

\subsubsection{Lower bound strategy}
In this subsection, we prove the lower bound for the group divergence of $G_m$. The idea is to work by induction on $m$ using the fact that $G_i$ isometrically embeds into $G_{i+1}$. However, one has to proceed carefully. There are more paths connecting points in the complement of the open $r$--ball in $G_{i+1}$ than in the complement of the open $r$--ball in $G_i$, and so lower bounds on path lengths in the latter do not \emph{a priori} imply the same bounds on path lengths in the former. Furthermore, there are more points in the $r$--sphere in $G_{i+1}$ than in the $r$--sphere in $G_i$.

We shall establish lower bound estimates for the lengths of open $r$--ball avoidant paths connecting two points on the $r$--sphere. These lower bound estimates are obtained by cutting these paths into pieces using $s_m$-hyperplanes and estimating the lengths of these pieces. We work with a $2$--dimensional geometric model $X_m$ for $G_m$.





\subsubsection{Lower bound details}
 Let $Y_1$ be the standard presentation $2$--complex for the group $G_1$ with the finite generating set $S_1$. We construct 2-complexes $Y_m$ for the groups $G_m$ inductively as follows. 
 
The complex $Y_2$ is the total space of the graph of spaces based on a circle with one vertex. The vertex space is $Y_1$ and the edge space is $\Sigma_2$, a bouquet of oriented circles labeled $e_i$ for $1 \leq i \leq p$. There are two maps $\Sigma_2 \to Y_1$, one sending $e_i$ to $a_i$ for $1\leq i \leq p$ and the other sending $e_i$ to $b_i$ for $1\leq i\leq p$. The total space of this graph of spaces (as defined in \cite{MR564422})
$$
Y_2 \; =\; [Y_1 \, \cup\,  (\Sigma_2 \times [0,1])]/\sim
$$
has one extra 1-cell which we label $s_2$ and 2-cells with boundary words $s_2^{-1}a_is_2b_i^{-1}$ for $1 \leq i \leq p$. 
 
 For $m \geq 3$ the complex $Y_m$ is the total space of the graph of spaces based on a circle with one vertex. The vertex space is $Y_{m-1}$ and the edge space is $\Sigma_m$, an oriented circle labeled $e$. There are two maps $\Sigma_m \to Y_{m-1}$, one sending $e$ to $s_1$ and the other sending $e$ to $s_{m-1}$. The total space 
 $$
Y_m \; =\; [Y_{m-1}\, \cup\,  (\Sigma_m \times [0,1])]/\sim
$$
has one extra 1-cell which we label $s_m$ and a 2-cell with boundary word $s_m^{-1}s_1s_ms_{m-1}^{-1}$. 
 
 Throughout this section, $X_m$ denotes the universal cover of $Y_m$. We define the notions of $s_m$-hyperplane, $k$--rays and $(1,k)$--rays, $k$--corners, and $r$--avoidant paths over corners in $X_m$. These play a fundamental role in the proof of the lower bounds. 
 
 \begin{defn}[$s_m$-hyperplanes]
 \label{defn:hyperplane}
Let $\widetilde\Sigma_m$ denote the universal cover of the circle $\Sigma_m$ consisting of one 0--cell and one 1--cell. The cell structure of $\widetilde\Sigma_m$ agrees with the standard 1--dimensional cubing of $\R$ with 0--cells at $\Z$. The space $\widetilde\Sigma_m \times [0,1]$ has the product square structure. The map $$\widetilde\Sigma_m \times [0,1]\,  \to\,  \Sigma_m \times [0,1] \, \to\,  [Y_{m-1}\, \cup\,  (\Sigma_m \times [0,1])]/\sim  \, \; = \;  Y_m$$ lifts to the universal cover $X_m$. These lifts are indexed by the edges of the Bass-Serre tree of the HNN description of $G_m$. Each image of $\widetilde\Sigma_m \times \{1/2\}$ in $X_m$ is called an \emph{$s_m$-hyperplane}. We say that the $s_m$-hyperplane, $H$, is \emph{dual} to the edges of $X_m$ which it intersects; these edges are all labeled $s_m$.
Note that each $s_m$-hyperplane $H$ separates $X_m$ into two components. We will make use of open and closed star neighborhoods of $H$ in the cell structure on $X_m$. 
Note that $star(H)$ (resp.\ $ostar(H)$) is the image of $\widetilde\Sigma_m \times [0,1]$ (resp.\ $\widetilde\Sigma_m \times (0,1)$) in $X_m$.
 \end{defn}

\begin{defn}[$k$--rays and $(1,k)$--rays]
\label{fcner}
Let $m\geq 1$ be an integer. A geodesic ray $\alpha$ in the complex $X_m$ is a \emph{$k$--ray} for $1\leq k\leq m$ if all edges of $\alpha$ are labeled by $s_k$. A geodesic ray $\beta$ in the complex $X_m$ is a \emph{$(1,k)$--ray} for $2\leq k\leq m$ if one of the following holds:
\begin{enumerate}
    \item $\beta$ is a $1$--ray;
    \item $\beta$ is a concatenation of $\sigma_1 \sigma_2$, where $\sigma_1$ is a non-degenerate segment with edges labeled by $s_1$ and $\sigma_2$ is a $k$--ray.
\end{enumerate}
Other than the fact that the rays are geodesics, we do not make any assumptions on the orientations of edges in $k$--rays or in $(1,k)$--rays. 
\end{defn}

We define $k$--corners for $k\geq 2$ as follows.

\begin{defn}[$k$--corner for $k\geq 2$]
\label{corner2}
Let $2\leq k \leq m$ be integers. Let $\alpha$ be a $(1,k)$--ray and $\beta$ be a $k$--ray in the complex $X_m$ which share the same initial point $x$. Then $(\alpha,\beta)_x$ is called a \emph{$k$--corner} at $x$. 

\end{defn}

The following definition recalls Macura's notion of detour paths over corners in \cite{MR3032700}. 

\begin{defn}[An $r$--avoidant path over a $k$--corner]
Let $2\leq k\leq m$ be integers and let $(\alpha,\beta)_x$ be a $k$--corner in $X_m$. For each $r>0$ a path which lies outside the open ball $B(x,r)$ in $X_m$ connecting a vertex of $\alpha$ to a vertex of $\beta$ is called an \emph{$r$--avoidant path over the $k$--corner $(\alpha,\beta)_x$} in $X_m$.
\end{defn}

Next, we investigate how hyperplanes intersect corners. We start with a basic result about the intersection of hyperplanes and rays. 

\begin{lem}\label{lem:hyp-geod}
Let $k\geq 2$ and $H$ be an $s_k$-hyperplane in $X_k$. Then $H$ intersects each $k$--ray at most once.
\end{lem}

\begin{proof}
Suppose $H$ intersects a $k$--ray at two points. Then there is a path in $star(H) - ostar(H)$ which connects two vertices of the $k$--ray. In the case $k\geq 3$, this means that a non-zero power of $s_k$ is equal to a power of either $s_1$ or $s_{k-1}$. This contradicts Lemma~\ref{lcool}. In the case $k=2$, this intersection implies that a non-zero power of $s_2$ is equal to an element of the free group generated by $\{a_1, \ldots, a_p\}$ or $\{b_1, \ldots , b_p\}$. But this is impossible because these two elements have distinct normal forms in the HNN extension structure of $G_2$. 
\end{proof}

The following lemma is analogous to Lemma~4.11 in \cite{BT2020}. It states that an $s_k$-hyperplane can only intersect a $k$--corner in $X_k$ in a very restricted way. 

\begin{lem}
\label{coolhyp}
Let $k\geq 3$ and $(\alpha, \beta)_x$ be a $k$--corner in $X_k$. If $H$ is an $s_k$-hyperplane which intersects both $\alpha$ and $\beta$, then $H$ intersects the first edge of $\beta$. 
\end{lem}

\begin{proof}
Let $i$ be the smallest positive integer such that the $s_k$-hyperplane $H_i$ dual to the $i^{th}$ edge of $\beta$ intersects $\alpha$. We observe that if $i \geq 2$, the $s_k$-hyperplane $H_{i-1}$ dual to the $(i-1)^{th}$ edge of $\beta$ 
cannot intersect $\alpha$ by the choice of $i$ and cannot intersect $\beta$ twice by Lemma~\ref{lem:hyp-geod}, and so 
must intersect $H_i$ which contradicts the fact that $s_k$-hyperplanes are disjoint. Therefore, $i$ must equal $1$. 

Assume that some hyperplane $H_j$ dual to the $j^{th}$ edge of $\beta$ for $j\geq 2$ intersects $\alpha$. Arguing as in the preceding paragraph, we see that the hyperplane $H_2$ dual to the second edge of $\beta$ must intersect $\alpha$. We note that for $i=1, 2$ each $star(H_i)-ostar(H_i)$ consists of two bi-infinite geodesics: one has edges labeled by $s_1$ and the other has edges labeled by $s_{k-1}$. This gives a loop based at the vertex between the first and second edge of $\beta$ which is labeled by $s_{k-1}^ns_k^ps_1^m$ for $m\neq 0$ and $n\neq 0$. This contradicts Lemma~\ref{lcool2}. 

This implies that all the $s_k$-hyperplanes dual to the second or subsequent edges of $\beta$ cannot meet $\alpha$, and the lemma is proved. 
\end{proof}

\begin{rem}
One can prove that the $s_k$-hyperplane $H$ in the previous lemma must intersect the ray $\alpha$ in the first occurrence of the edge $s_k$. However, we do not need this feature of $H$.
\end{rem}

The intersection of $s_2$-hyperplanes with $2$--corners in $X_2$ is even more restricted.

\begin{lem}
\label{coolhyp1}
Let $(\alpha, \beta)_x$ be a $2$--corner in $X_2$. There is no $s_2$-hyperplane in $X_2$ that intersects both $\alpha$ and $\beta$.
\end{lem}

\begin{proof}
 Assume that some hyperplane $H$ dual to an edge of $\beta$ intersects $\alpha$. We note that $star(H)-ostar(H)$ consists of two infinite trees; one with edge labels from $\{a_1, \ldots, a_p\}$ and the other with edge labels from $\{b_1, \ldots , b_p\}$. Thus there is a loop based at a point in $(star(H)-ostar(H))\cap \beta$ which is labeled by $s_2^ms_1^ns_2^qw$ where $n\neq 0$ and $w$ is a word representing a group element $g_0$ in $F_{xz}$. This implies that $s_2^ms_1^ns_2^qg_0=e$. We observe that there is a group automorphism $\Psi:\!G_2\to \Z$ taking $s_1$ to $1$ and other generators of $G_2$ to $0$. Since $\Psi(s_2^ms_1^ns_2^qg_0)=n\neq 0$, we have $s_2^ms_1^ns_2^qg_0\neq e$ which is a contradiction. Therefore, there is no hyperplane in $X_2$ that intersects both $\alpha$ and $\beta$. 
\end{proof}

The following lemma is analogous to Lemma~4.13 in \cite{BT2020}. It states that $s_1$-geodesic segments can be modified relative to their endpoints to become ball-avoidant paths with linearly bounded lengths. The modifications can be made in the the appropriate coset of the isometrically embedded rank-2 free abelian group $\langle a_1, s_1\rangle$ in $G_m$. The reader may refer to \cite[Lemma~4.13]{BT2020} for the proof.

\begin{lem}
\label{avoidantlem}
Let $m\geq 1$ and $\alpha$ be a geodesic segment in $X_m$ with edges labeled by $s_1$ and endpoints $x$, $y$. Let $z$ be a vertex of $X_m$ and assume that $r=\min\{d_{S_m}(x,z),d_{S_m}(y,z)\} > 0$. Then there exists a path $\beta$ connecting $x$ and $y$ with edges labeled by $s_1$ and $a_1$ such that $\beta$ lies outside the open ball $B(z,r/2)$ and $\ell(\beta)\leq 11 \ell (\alpha)$.
\end{lem}

The goal of the next three lemmas is to establish lower bounds of $r^{n-1}f(r)$ on the lengths of $r$--avoidant paths over $n$--corners in $X_{n}$. The precise bound is given in Lemma~\ref{p3p3}.
We prove this bound by proving a stronger statement: namely, $r^{n-1}f(r)$ is a lower bound for the lengths of $r$--avoidant paths over $n$--corners in the larger space $X_{n+1}$. 
This is achieved by proving the following more general version of the latter statement by induction on $d$: \begin{quote}
    For each $n \geq d$ the lengths of $r$--avoidant paths over $n$--corners in $X_{n+1}$ are bounded below by $r^{d-1}f(r)$.
\end{quote}
Note that setting $d=n$ yields the strong statement above. 
The induction step is proven in Lemma~\ref{p1p1}. The base case of the induction breaks into two parts; one for $2$--corners in $X_3$ and the other for $m$--corners in $X_{m+1}$ when $m \geq 3$.

\begin{lem}[Base case]
\label{okokok}
For each integer $m \geq 2$ the length of $r$--avoidant paths over $m$--corners in $X_{m+1}$ is at least $(r/16)f(r/4)$ for $r$ sufficiently large. 
\end{lem}

\begin{figure}
  
 \tikzset{->-/.style={decoration={
  markings,
  mark=at position .5 with {\arrow[ultra thick]{>}}},postaction={decorate}}}
   \begin{tikzpicture}[scale=0.7]
   
   \draw[dashed, line width=1.0pt] (11,-2) arc (40:127:12); \draw[dashed, line width=1.0pt] (11,-2) to (12,-3.5) to (12.5,-5) to (13,-7.7);
   
   \draw[line width=1.0pt] (-2,-8) node[circle,fill,inner sep=1.5pt, color=black](1){} to (11,-8);
   \draw[line width=1.0pt] (-2,-8) node[circle,fill,inner sep=1.5pt, color=black](1){} to (-5,-2);
    
    \draw[line width=1.0pt] (2,-8) node[circle,fill,inner sep=1.5pt, color=black](1){} to (1,-8) node[circle,fill,inner sep=1.5pt, color=black](1){}; \draw[->-] (1,-8) node[circle,fill,inner sep=1.5pt, color=black](1){} to (2,-8) node[circle,fill,inner sep=1.5pt, color=black](1){};
    
    \draw[line width=1.0pt] (3,-8) node[circle,fill,inner sep=1.5pt, color=black](1){} to (2,-8) node[circle,fill,inner sep=1.5pt, color=black](1){}; \draw[->-] (2,-8) node[circle,fill,inner sep=1.5pt, color=black](1){} to (3,-8) node[circle,fill,inner sep=1.5pt, color=black](1){};
    
    \draw[line width=1.0pt] (2,-8) node[circle,fill,inner sep=1.5pt, color=black](1){} to (2,1) node[circle,fill,inner sep=1.5pt, color=black](1){};\draw[line width=1.0pt] (1,-8) node[circle,fill,inner sep=1.5pt, color=black](1){} to (1,1) node[circle,fill,inner sep=1.5pt, color=black](1){}; \draw[line width=1.0pt] (1,1) node[circle,fill,inner sep=1.5pt, color=black](1){} to (2,1) node[circle,fill,inner sep=1.5pt, color=black](1){}; \draw[->-] (1,1) node[circle,fill,inner sep=1.5pt, color=black](1){} to (2,1) node[circle,fill,inner sep=1.5pt, color=black](1){};
    
      \draw[line width=1.0pt] (-0.5,2.1) node[circle,fill,inner sep=1.5pt, color=black](1){} to (0.1,1.6) to (1,1);\draw[line width=1.0pt] (2,1) to (2.7,1.6) to  (3,2.2)node[circle,fill,inner sep=1.5pt, color=black](1){};

    \draw[line width=1.0pt] (3,-8) node[circle,fill,inner sep=1.5pt, color=black](1){} to (8,-1.2) node[circle,fill,inner sep=1.5pt, color=black](1){};\draw[line width=1.0pt] (2,-8) node[circle,fill,inner sep=1.5pt, color=black](1){} to (7.5,-.8) node[circle,fill,inner sep=1.5pt, color=black](1){};\draw[->-] (7.5,-0.8) node[circle,fill,inner sep=1.5pt, color=black](1){} -- (8,-1.2) node[circle,fill,inner sep=1.5pt, color=black](1){};
    
     \draw[line width=1.0pt] (9.5,-0.5) node[circle,fill,inner sep=1.5pt, color=black](1){} to (9,-0.8) to (8,-1.2) node[circle,fill,inner sep=1.5pt, color=black](1){};
     
     \draw[line width=1.0pt] (7.2,1) node[circle,fill,inner sep=1.5pt, color=black](1){} to (7.3,0.1) to (7.5,-0.8) node[circle,fill,inner sep=1.5pt, color=black](1){};
     
     \draw[line width=1.0pt] (9.5,-0.5) node[circle,fill,inner sep=1.5pt, color=black](1){} to (10.5,-1.4) to (11.1,-2.1) to (11.5,-2.8) node[circle,fill,inner sep=1.5pt, color=black](1){};
     
     \draw[line width=1.0pt] (7.2,1) node[circle,fill,inner sep=1.5pt, color=black](1){} to (6.3,1.43);
     
     \draw[line width=1.0pt] (3,2.2) to (3.9,2.13);
     
     \draw[line width=1.0pt] (-0.5,2.1) node[circle,fill,inner sep=1.5pt, color=black](1){} to (-2, 1.67) to (-3.5,1.1) node[circle,fill,inner sep=1.5pt, color=black](1){};
    
  \node at (2,-8.5) {$s_m^j$}; \node at (2.4,-3.1) {$\beta_j$}; \node at (4.7,-3.7) {$\alpha_j$}; \node at (6.8,-0.8) {$u'_{j+1}$}; \node at (7.3,1.4) {$u_{j+1}$};  \node at (2.4,0.75) {$v'_j$};\node at (3,2.5) {$v_j$}; \node at (5,-8.5) {$\beta$}; \node at (-4.7,-3.5) {$\alpha$}; \node at (1.5,2.6) {$\tau_{\ell(j)}$}; \node at (9,0.5) {$\tau_{\ell(j+1)}$}; \node at (3,1.4) {$\eta_1$}; \node at (7,0.1) {$\eta_{2}$};

      \end{tikzpicture}
\caption{The paths $\alpha_j$ and $\beta_j$ determined by consecutive $s_m$-hyperplanes.} 
\label{Fi1'}      
\end{figure}

\begin{proof}
Suppose that $\gamma$ is an $r$--avoidant path of minimal length over all $m$--corners in $X_{m+1}$. By left translating if necessary, assume that $\gamma$ is $r$--avoidant over an $m$--corner $(\alpha, \beta)_e$ based at the identity. Now, $\gamma \cap \alpha = u$ is a vertex of the form $s_1^ss_{m}^t$ for some integers $s\neq 0$ and $t$ such that $|s|+|t|\geq r$ and $\gamma \cap \beta = v$ is a vertex of the form $s_m^p$ for some integer $p$ such that $|p|\geq r$. We assume that $p>0$; the proof for $p<0$ is similar.

The intersection $\gamma \cap X_m$ is a union of closed, connected sub-paths of $\gamma$ some of which may be single vertices. These sub-paths together with the closures of components of $\gamma - X_m$ mean that $\gamma$ can be written as a concatenation 
 $$\sigma_1\tau_1\sigma_2\tau_2\cdots \sigma_{\ell}\tau_{\ell}\sigma_{\ell+1}$$ such that:
\begin{enumerate}
    \item Each $\tau_i$ intersects $X_{m}$ only at its endpoints $x_i =  \sigma_i \cap \tau_i $ and $y_i = \tau_i \cap \sigma_{i+1}$; 
    \item Each $\sigma_i$ lies completely in the $1$--skeleton of $X_{m}$.
\end{enumerate}

We observe that each $x_i^{-1}y_i$ is a group element in the cyclic subgroup $\langle s_m \rangle$ or in the cyclic subgroup $\langle s_{1} \rangle$. If $x_i^{-1}y_i$ is a group element in the cyclic subgroup $\langle s_m \rangle$, then we replace $\tau_i$ by a geodesic $\tau'_i$ labeled by $s_m$. Otherwise, $x_i^{-1}y_i$ is a group element in the cyclic subgroup $\langle s_1 \rangle$, then we replace $\tau_i$ by a geodesic $\tau'_i$ labeled by $s_1$. The new path $\gamma'=\sigma_1\tau'_1\sigma_2\tau'_2\cdots \sigma_{\ell}\tau'_{\ell}\sigma_{\ell+1}$ lies completely in the $1$--skeleton of $X_m$, has the same initial and end points as $\gamma$, and $\ell(\gamma')\leq \ell(\gamma)$. Note that the new path $\gamma'$ may intersect the open ball $B(e,r)$. We call each subpath $\tau'_i$ of $\gamma'$ a short-cut segment.

We claim that $\gamma'\cap \beta$ does not contain an edge. Assume to the contrary that $\gamma'\cap \beta$ contains an edge. Then some short-cut segment $\tau'_i$ of $\gamma'$ with edges labeled by $s_m$ must contain an edge of $\beta$. Therefore, the endpoint $x_i$ of $\tau'_i$ has the form $s_m^q$ for some $|q|\geq r$. This implies that the subpath $\zeta$ of $\gamma$ connecting $u$ and $x_i$ is also an $r$--avoidant path over an $m$--corner in the complex $X_{m+1}$. Also, $|\zeta|<|\gamma|$ which contradicts to the minimality of $\gamma$.

For each positive integer $j$ we call $e_j$ the $j^{th}$ edge of $\beta$. We know that all edges $e_j$ are labeled by $s_m$. Let $H_j$ be the $s_m$-hyperplane of the complex $X_m$ dual to the edge $e_j$. We now consider the edges $e_j$ for $r/8\leq j\leq r/4$. Assume that $r\geq 16$ so that $j\geq 2$. Therefore, by Lemma~\ref{coolhyp} (in the case $m\geq 3$) or Lemma~\ref{coolhyp1} (in the case $m= 2$) each hyperplane $H_j$ will not intersect $\alpha$ and so must intersect $\gamma'$. Let $m_j$ be the point in the intersection $H_j\cap \gamma'$ such that the subpath of $\gamma'$ connecting $m_j$ and $v$ does not intersect $H_j$ at point other than $m_j$. Then $m_j$ is an interior point of an edge $f_j$ labeled by $s_m$ in $\gamma'$. Since $\gamma'\cap \beta$ does not contain an edge, two edges $e_j$ and $f_j$ are distinct. 

Let $\beta_j$ be the path in $star(H_j)-ostar(H_j)$ that connects the terminal $s_m^j$ of the edge $e_j$ to some endpoint $v'_j$ of the edge $f_j$. Then edges of $\beta_j$ are labeled by $s_{m-1}$ (in the case $m \geq 3$) or by generators of the free group $F_{yz}$ (in the case $m=2$). If $f_j$ is not an edge of a short-cut segment of $\gamma'$, then we let $v_j=v'_j$. Otherwise, $f_j$ is an edge in some short-cut segment $\tau'_{\ell_j}$ of $\gamma'$. In this case, we let $v_j$ be the endpoint $y_{\ell_j}$ of $\tau'_{\ell_j}$. Let $\alpha_j$ be the path in $star(H_{j+1})-ostar(H_{j+1})$ that connects the initial endpoint $s_m^j$ of the edge $e_{j+1}$ to some endpoint $u'_{j+1}$ of the edge $f_{j+1}$. Then edges of $\alpha_j$ are labeled by $s_1$ (in the case $m \geq 3$) or by generators of $F_{xz}$ (in the case $m=2$). If $f_{j+1}$ is not an edge of a short-cut segment of $\gamma'$, then we let $u_{j+1}=u'_{j+1}$. Otherwise, $f_{j+1}$ is an edge in some $\tau'_{\ell_{j+1}}$. In this case, we let $u_{j+1}$ be the endpoint $x_{\ell_{j+1}}$ of $\tau'_{\ell_{j+1}}$. We refer the reader to Figure~\ref{Fi1'} to help visualize the construction.

Let $\gamma'_j$ be the subpath of $\gamma'$ that connects $v'_j$ and $u'_{j+1}$. We will prove that the length of $\gamma'_j$ is at least $f(r/4)$. By the construction, the path $\gamma'_j$ contains two subpaths $\eta_1$ and $\eta_2$ of $\gamma'$, where $\eta_1$ (resp. $\eta_2$) is the (possibly degenerate) subsegment of $\gamma'_j$ connecting $v'_j$ and $v_j$ (resp. $u'_{j+1}$ and $u_{j+1}$). If either the length of $\eta_1$ or the length of $\eta_2$ is greater than $r/4$, then the length of $\gamma'_j$ is greater than $r/4$ and therefore greater than $f(r/4)$. 

We now assume that the lengths of $\eta_1$ and $\eta_2$ are both less than or equal to $r/4$. In this case $$d_{S_m}(v'_j,v_j)\leq r/4 \text{ and } d_{S_m}(u'_{j+1},u_{j+1})\leq r/4.$$ Also, $v_j$ and $u_{j+1}$ lies outside the open ball $B(e,r)$. Then $v'_j$ and $u'_{j+1}$ lies outside the open ball $B(e,3r/4)$.
At this point the proof splits into two cases; the case $m \geq 3$ and the case $m=2$. 

\noindent
\emph{Case $m \geq 3$.} 
In this case, we note that $v'_j=s_m^js_{m-1}^{q_1}$ and $u'_{j+1}=s_m^js_{1}^{q_2}$ for some integers $q_1$ and $q_2$. Moreover, $d_{S_m}(e,s_m^j)=j\leq r/4$. Therefore, we have $|q_1|\geq r/2$ and $|q_2|\geq r/2$. This implies that $$\ell(\gamma'_j)\geq d_{S_m}(v'_j,u'_{j+1})=|s_{m-1}^{-q_1}s_{1}^{q_2}|_{S_m}=|q_1|+|q_2|\geq r\geq r/4\geq f(r/4).$$

\noindent
\emph{Case $m = 2$.} 
In this case, we note that $v'_j=s_2^jg_1$ for some $g_1\in F_{yz}$ and $u'_{j+1}=s_2^jg_0$ for some $g_0\in F_{xz}$. Moreover, $d_{S_2}(e,s_2^j)=j\leq r/2$. Therefore, we have $|g_0|_{S_2}\geq r/4$ and $|g_1|_{S_2}\geq r/4$. By Lemma~\ref{ll1} and Lemma~\ref{st1} we have $$\ell(\gamma'_j)\geq d_{S_2}(v'_j,u'_{j+1})=|g_1^{-1}g_0|_{S_2}\geq f(r/4).$$

Therefore, in both the $m=2$ and $m \geq 3$ cases we have $$\ell(\gamma)\geq \ell(\gamma')\geq \sum_{r/8\leq j\leq r/4}\ell(\gamma'_j)\geq (r/16)f(r/4).$$

\end{proof}

The following lemma is analogous to Proposition~4.14 in \cite{BT2020}. The key technical ingredients Lemma~4.11 and Lemma~4.13 used in the proof of Proposition~4.14 in \cite{BT2020} are replaced in this paper by Lemma~\ref{coolhyp} and Lemma~\ref{avoidantlem} respectively. 
We give the full details of the proof of Lemma~\ref{p1p1} for the convenience of the reader.

\begin{lem}[The induction step]
\label{p1p1}
Let $d\geq 2$ be an integer and $g\!:(0,\infty)\to(0,\infty)$ be a function. Assume that for all $m\geq d$ all $r$--avoidant paths over an $m$--corner in $X_{m+1}$ have length at least $g(r)$ for $r$ sufficiently large. Then for all $n\geq d+1$ all $r$--avoidant paths over an $n$--corner in the complex $X_{n+1}$ have length at least $(r/180)g(r/4)$ for $r$ sufficiently large. 
\end{lem}

\begin{proof}
Let $n\geq d+1$ and let $\gamma$ be an $r$--avoidant path over an $n$--corner $(\alpha,\beta)_x$ in the complex $X_{n+1}$.
We assume $\gamma$ is an $r$--avoidant path of minimal length over all $n$--corners in $X_{n+1}$. This implies that $\gamma \cap \alpha$ is a vertex $u$ and $\gamma \cap \beta$ is a vertex $v$. By translating the corner if necessary, we assume that $x$ is the identity $e$. Therefore, the vertex $u$ has the form $s_1^ss_{n}^t$ for some integers $s\neq 0$ and $t$ such that $|s|+|t|\geq r$. Similarly, the vertex $v$ has the form $s_n^p$ for some integer $p$ such that $|p|\geq r$. We give the proof in the case $p>0$; the proof for $p<0$ is similar.

As in the proof of Lemma~\ref{okokok}, $\gamma$ can be written as a concatenation 
 $$\sigma_1\tau_1\sigma_2\tau_2\cdots \sigma_{\ell}\tau_{\ell}\sigma_{\ell+1}$$ such that:
\begin{enumerate}
    \item Each $\tau_i$ intersects $X_{n}$ only at its endpoints $x_i =  \sigma_i \cap \tau_i $ and $y_i = \tau_i \cap \sigma_{i+1}$; 
    \item Each $\sigma_i$ lies completely in the $1$--skeleton of $X_{n}$.
\end{enumerate}


We also observe that $x_i^{-1}y_i$ is a group element in the cyclic subgroup $\langle s_n \rangle$ or in the cyclic subgroup $\langle s_{1} \rangle$ for each $i$. 

If $x_i^{-1}y_i$ is a group element in the cyclic subgroup $\langle s_n \rangle$, then we replace $\tau_i$ by a geodesic $\tau'_i$ labeled by $s_n$. Call this path $\tau'_i$ an $X_n$-replacement segment of type 1. Note that $\tau'_i$ in this case may have non-empty intersection with the open ball $B(e,r/2)$. These $X_n$-replacement segments of type 1 will be handled in \textbf{Case 2} below; see also Figure~\ref{Fi2}. The technical formulation of the induction statement is needed to handle the situation where these $X_n$-replacement segments of type 1 meet $B(e,r/2)$.

If $x_i^{-1}y_i$ is a group element in the cyclic subgroup $\langle s_1 \rangle$, then by Lemma~\ref{avoidantlem} we can replace $\tau_i$ a path $\tau'_i$ with edges labeled by $s_1$ and $a_1$ such that $\tau'_i$ lies outside the open ball $B(e,r/2)$ and $\ell(\tau'_i)\leq 11\text{ }d_{S_n}(x_i,y_i)\leq 11 \ell (\tau_i)$. 
Call this path an $X_n$-replacement segment of type 2. 

The new path $\gamma'=\sigma_1\tau'_1\sigma_2\tau'_2\cdots \sigma_{\ell}\tau'_{\ell}\sigma_{\ell+1}$ lies completely in the $1$--skeleton of $X_n$, shares two endpoints $u$ and $v$ with $\gamma$, and $\ell(\gamma')\leq 11 \ell(\gamma)$.
We claim that $\gamma'\cap \beta$ does not contain an edge. In fact, assume by the way of contradiction that $\gamma'\cap \beta$ contains an edge. Then some $X_n$-replacement segment $\tau'_i$ of type $1$ must contain an edge of $\beta$. Therefore, the endpoint $x_i$ of $\tau'_i$ has the form $s_n^q$ for some $|q|\geq r$. Therefore, the subpath $\zeta$ of $\gamma$ connecting $u$ and $x_i$ is also an $r$--avoidant path over an $n$--corner in the complex $X_{n+1}$. Also, $|\zeta|<|\gamma|$ which contradicts to the choice of $\gamma$. 

For each positive integer $j$ we call $e_j$ the $j^{th}$ edge of $\beta$. We know that all edges $e_j$ are labeled by $s_n$. Let $H_j$ be the $s_n$-hyperplane of the complex $X_n$ that corresponds to the edge $e_j$. We now consider $r/8\leq j\leq r/4$. Assume that $r\geq 16$. Then $j\geq 2$. Therefore by Lemma~\ref{coolhyp} each hyperplane $H_j$ must intersect $\gamma'$. Let $m_j$ be the point in the intersection $H_j\cap \gamma'$ such that the subpath of $\gamma'$ connecting $m_j$ and $v$ does not intersect $H_j$ at point other than $m_j$. Then $m_j$ is an interior point of an edge $f_j$ labeled by $s_n$ in $\gamma'$. Since $\gamma'\cap \beta$ does not contain an edge, two edges $e_j$ and $f_j$ are distinct. 

Let $\beta_j$ be the path in $star(H_j)-ostar(H_j)$ that connects the terminal $s_n^j$ of the edge $e_j$ to some endpoint $v'_j$ of the edge $f_j$. Then $\beta_j$ is a part of an $(n-1)$--ray. If $f_j$ is not an edge of an $X_n$-replacement segment of $\gamma'$, then we let $v_j=v'_j$. Otherwise, $f_j$ is an edge in some $X_n$-replacement segment $\tau'_{\ell_j}$ of type 1 of $\gamma'$. In this case, we let $v_j$ be the endpoint $y_{\ell_j}$ of $\tau'_{\ell_j}$. Let $\alpha_j$ be the path in $star(H_{j+1})-ostar(H_{j+1})$ that connects the initial endpoint $s_n^j$ of the edge $e_{j+1}$ to some endpoint $u'_{j+1}$ of the edge $f_{j+1}$. Then $\alpha_j$ is a part of an $1$--ray. If $f_{j+1}$ is not an edge of an $X_n$-replacement segment of $\gamma'$, then we let $u_{j+1}=u'_{j+1}$. Otherwise, $f_{j+1}$ is an edge in some $\tau'_{\ell_{j+1}}$. In this case, we let $u_{j+1}$ be the endpoint $x_{\ell_{j+1}}$ of $\tau'_{\ell_{j+1}}$.

\begin{figure}
  
 \tikzset{->-/.style={decoration={
  markings,
  mark=at position .5 with {\arrow[ultra thick]{>}}},postaction={decorate}}}
   \begin{tikzpicture}[scale=0.7]
   
   \draw[dashed, line width=1.0pt] (11,-2) arc (40:127:12); \draw[dashed, line width=1.0pt] (11,-2) to (12,-3.5) to (12.5,-5) to (13,-7.7);
   
   \draw[line width=1.0pt] (-2,-8) node[circle,fill,inner sep=1.5pt, color=black](1){} to (11,-8);
   \draw[line width=1.0pt] (-2,-8) node[circle,fill,inner sep=1.5pt, color=black](1){} to (-5,-2);
    
    \draw[line width=1.0pt] (2,-8) node[circle,fill,inner sep=1.5pt, color=black](1){} to (1,-8) node[circle,fill,inner sep=1.5pt, color=black](1){}; \draw[->-] (1,-8) node[circle,fill,inner sep=1.5pt, color=black](1){} to (2,-8) node[circle,fill,inner sep=1.5pt, color=black](1){};
    
    \draw[line width=1.0pt] (3,-8) node[circle,fill,inner sep=1.5pt, color=black](1){} to (2,-8) node[circle,fill,inner sep=1.5pt, color=black](1){}; \draw[->-] (2,-8) node[circle,fill,inner sep=1.5pt, color=black](1){} to (3,-8) node[circle,fill,inner sep=1.5pt, color=black](1){};
    
    \draw[line width=1.0pt] (2,-8) node[circle,fill,inner sep=1.5pt, color=black](1){} to (2,1) node[circle,fill,inner sep=1.5pt, color=black](1){};\draw[line width=1.0pt] (1,-8) node[circle,fill,inner sep=1.5pt, color=black](1){} to (1,1) node[circle,fill,inner sep=1.5pt, color=black](1){}; \draw[line width=1.0pt] (1,1) node[circle,fill,inner sep=1.5pt, color=black](1){} to (2,1) node[circle,fill,inner sep=1.5pt, color=black](1){}; \draw[->-] (1,1) node[circle,fill,inner sep=1.5pt, color=black](1){} to (2,1) node[circle,fill,inner sep=1.5pt, color=black](1){};
    
      \draw[line width=1.0pt] (-0.5,2.1) node[circle,fill,inner sep=1.5pt, color=black](1){} to (0.1,1.6) to (1,1);\draw[line width=1.0pt] (2,1) to (2.7,1.6) to  (3,2.2)node[circle,fill,inner sep=1.5pt, color=black](1){};

    \draw[line width=1.0pt] (3,-8) node[circle,fill,inner sep=1.5pt, color=black](1){} to (8,-1.2) node[circle,fill,inner sep=1.5pt, color=black](1){};\draw[line width=1.0pt] (2,-8) node[circle,fill,inner sep=1.5pt, color=black](1){} to (7.5,-.8) node[circle,fill,inner sep=1.5pt, color=black](1){};\draw[->-] (7.5,-0.8) node[circle,fill,inner sep=1.5pt, color=black](1){} -- (8,-1.2) node[circle,fill,inner sep=1.5pt, color=black](1){};
    
     \draw[line width=1.0pt] (9.5,-0.5) node[circle,fill,inner sep=1.5pt, color=black](1){} to (9,-0.8) to (8,-1.2) node[circle,fill,inner sep=1.5pt, color=black](1){};
     
     \draw[line width=1.0pt] (7.2,1) node[circle,fill,inner sep=1.5pt, color=black](1){} to (7.3,0.1) to (7.5,-0.8) node[circle,fill,inner sep=1.5pt, color=black](1){};
     
     \draw[line width=1.0pt] (9.5,-0.5) node[circle,fill,inner sep=1.5pt, color=black](1){} to (10.5,-1.4) to (11.1,-2.1) to (11.5,-2.8) node[circle,fill,inner sep=1.5pt, color=black](1){};
     
     \draw[line width=1.0pt] (7.2,1) node[circle,fill,inner sep=1.5pt, color=black](1){} to (6.3,1.43);
     
     \draw[line width=1.0pt] (3,2.2) to (3.9,2.13);
     
     \draw[line width=1.0pt] (-0.5,2.1) node[circle,fill,inner sep=1.5pt, color=black](1){} to (-2, 1.67) to (-3.5,1.1) node[circle,fill,inner sep=1.5pt, color=black](1){};
    
  \node at (2,-8.5) {$s_n^j$}; \node at (2.4,-3.1) {$\beta_j$}; \node at (4.7,-3.7) {$\alpha_j$}; \node at (6.8,-0.8) {$u'_{j+1}$}; \node at (7.3,1.4) {$u_{j+1}$};  \node at (2.4,0.75) {$v'_j$};\node at (3,2.5) {$v_j$}; \node at (5,-8.5) {$\beta$}; \node at (-4.7,-3.5) {$\alpha$}; \node at (1.5,2.6) {$\tau_{\ell(j)}$}; \node at (9,0.5) {$\tau_{\ell(j+1)}$};

      \end{tikzpicture}
\caption{$(\alpha_j,\beta_j)_{s_n^j}$ is a part of a $(n-1)$--corner at $s_n^j$}
\label{Fi1}      
\end{figure}

We see that $(\alpha_j,\beta_j)_{s_n^j}$ is a part of an $(n-1)$--corner at $s_n^j$ (see Figure~\ref{Fi1}). Let $\gamma'_j$ be the subpath of $\gamma'$ that connects $v'_j$ and $u'_{j+1}$. Let $\gamma''_j$ be the subpath of $\gamma'$ that connects $v_j$ and $u_{j+1}$. Let $\gamma_j$ be the subpath of $\gamma$ that connects $v_j$ and $u_{j+1}$. Then by the construction of $\gamma'$ we have $$\ell(\gamma''_j)\leq 11\ell(\gamma_j).$$
Also, $\gamma'_j=\eta_1\cup \gamma''_j\cup \eta_2$, where $\eta_1$ (resp. $\eta_2$) is the (possibly degenerate) subsegment of $\gamma'_j$ connecting $v'_j$ and $v_j$ (resp. $u'_{j+1}$ and $u_{j+1}$). Therefore, $$\ell(\gamma''_j)=\ell(\gamma'_j)-\bigl(\ell(\eta_1)+\ell(\eta_2)\bigr).$$
Since $\eta_1$ and $\eta_2$ are subpaths of $X_n$-replacement segments of type 1 which are also geodesics, we have $$\ell(\eta_1)=d(v'_j,v_j) \text{ and } \ell(\eta_2)=d(u'_{j+1},u_{j+1}).$$
This implies that $$\ell(\gamma'_j)-\bigl(d(v'_j,v_j)+d(u'_{j+1},u_{j+1})\bigr)\leq 11 \ell(\gamma_j).$$
Therefore, $$d(v'_j,v_j)+\ell(\gamma_j)+d(u'_{j+1},u_{j+1})\geq \ell(\gamma'_j)/11.$$

We note that each $X_n$-replacement segment $\tau'$ of type $2$ of $\gamma'$ lies outside the open ball $B(e,r/2)$ by the construction. Therefore, $\tau'$ also lies outside each open ball $B(s_n^{j},r/4)$ for $r/8 \leq j\leq r/4$. We now consider the case of $X_n$-replacement segments $\tau'$ of type 1 and we have two cases:

\textbf{Case 1}: For each $X_n$-replacement segment $\tau'$ of type $1$ of $\gamma'$ such that $\tau'\cap \gamma'_j\neq \emptyset$ the intersection $\tau'\cap \gamma'_j$ lies outside the open ball $B(s_n^{j},r/4)$. Then, $\gamma'_j$ is an $(r/4)$--avoidant path over the $(n-1)$--corner containing $(\alpha_j,\beta_j)_{s_n^j}$ in $X_n$. Also, $n-1\geq d$. Therefore, $\ell(\gamma'_j)\geq g(r/4)$ for $r$ sufficiently large by the induction hypothesis. This implies that $$d(v'_j,v_j)+\ell(\gamma_j)+d(u'_{j+1},u_{j+1})\geq g(r/4)/11.$$

\begin{figure}
  
 \tikzset{->-/.style={decoration={
  markings,
  mark=at position .5 with {\arrow[ultra thick]{>}}},postaction={decorate}}}
   \begin{tikzpicture}[scale=0.8]
   
   \draw[dashed, line width=1.0pt] (11,-2) arc (40:130:12);
   
   \draw[line width=1.0pt] (1,-6) node[circle,fill,inner sep=1.5pt, color=black](1){} to (1,-6); \node at (0.7,-6.4) {$s_n^j$};
  \draw[dashed,line width=1.0pt] (3,-5) arc (0:160:2); 
   \draw[dotted, line width=1.0pt][<->] (0.9,-5.9) -- (-0.7,-4); \node at (-0.1,-5.3) {$r/4$};
 
    \draw[line width=1.0pt] (-0.5,2.1) node[circle,fill,inner sep=1.5pt, color=black](1){} to (1,-5);\draw[line width=1.0pt] (2,-5) arc (0:-180:0.5);\draw[line width=1.0pt] (2,-5) to (3,2.2)node[circle,fill,inner sep=1.5pt, color=black](1){};\draw[->-] (2.1,-4.4) node[circle,fill,inner sep=1.5pt, color=black](1){} -- (2,-5) node[circle,fill,inner sep=1.5pt, color=black](1){};
    
    \draw[line width=1.0pt] (2,-5) node[circle,fill,inner sep=1.5pt, color=black](1){} to[out=-2,in=-100] (8,-1.2) node[circle,fill,inner sep=1.5pt, color=black](1){};\draw[line width=1.0pt] (2.1,-4.4) node[circle,fill,inner sep=1.5pt, color=black](1){} to[out=-2,in=-100] (7.5,-.8) node[circle,fill,inner sep=1.5pt, color=black](1){};\draw[->-] (7.5,-0.8) node[circle,fill,inner sep=1.5pt, color=black](1){} -- (8,-1.2) node[circle,fill,inner sep=1.5pt, color=black](1){};
    
     \draw[line width=1.0pt] (9.5,-0.5) node[circle,fill,inner sep=1.5pt, color=black](1){} to (9,-0.8) to (8,-1.2) node[circle,fill,inner sep=1.5pt, color=black](1){};
     
     \draw[line width=1.0pt] (7.2,1) node[circle,fill,inner sep=1.5pt, color=black](1){} to (7.3,0.1) to (7.5,-0.8) node[circle,fill,inner sep=1.5pt, color=black](1){};
     
     \draw[line width=1.0pt] (9.5,-0.5) node[circle,fill,inner sep=1.5pt, color=black](1){} to (10.5,-1.4) node[circle,fill,inner sep=1.5pt, color=black](1){};
     
     \draw[line width=1.0pt] (7.2,1) node[circle,fill,inner sep=1.5pt, color=black](1){} to (6.3,1.43);
     
     \draw[line width=1.0pt] (3,2.2) to (3.9,2.13);
     
     \draw[line width=1.0pt] (-0.5,2.1) node[circle,fill,inner sep=1.5pt, color=black](1){} to (-2, 1.67) to (-3.5,1.1) node[circle,fill,inner sep=1.5pt, color=black](1){};
    

    
    
   

  \node at (2.7,-2.1) {$s_n^*$}; \node at (4.3,-3.7) {$s_{1}^*$}; \node at (7,0.2) {$s_n^*$}; \node at (1.9,-4.2) {$w$}; \node at (7.1,-0.8) {$w'$}; \node at (7.2,1.4) {$\widetilde{w}$}; \node at (1.8,-4.75) {$f$}; \node at (8,-0.8) {$f'$}; \node at (10.9,-1.2) {$u_{j+1}$}; \node at (-3.5,1.4) {$v_j$}; \node at (3,2.5) {$\widetilde{v}$}; \node at (5.3,2) {$\eta$};

      \end{tikzpicture}
\caption{Some $X_n$-replacement segment of $\gamma'_j$ intersect the open ball $B(s_n^j,r/4)$ and the subsegment $\eta$ of $\gamma_j$ that connects $\widetilde{v}$ and $\widetilde{w}$ is an $(r/4)$--avoidant path over an $n$--corner in $X_{n+1}$.}
\label{Fi2}
\end{figure}
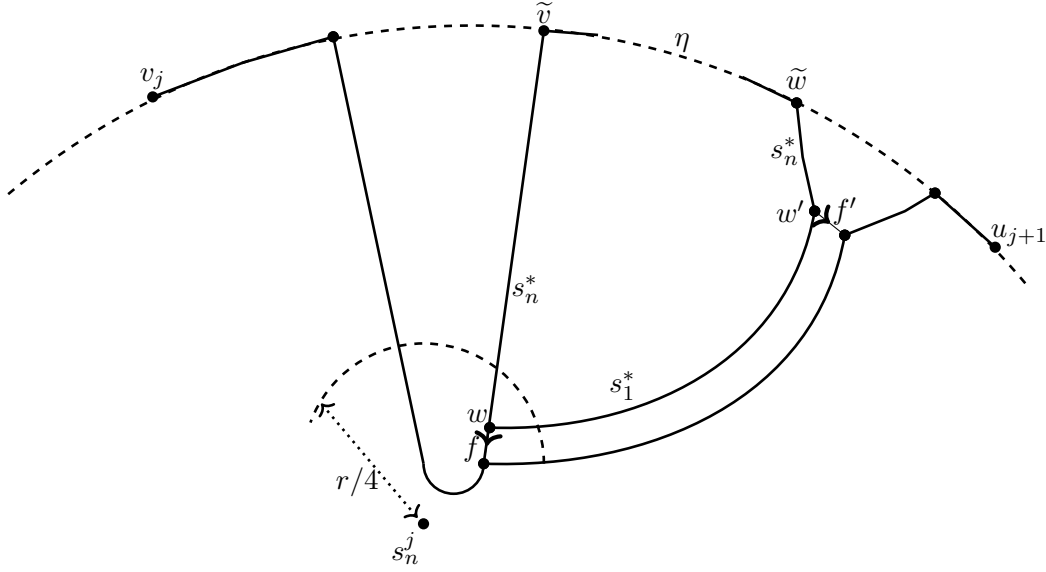

\textbf{Case 2}: We now assume that there is an $X_n$-replacement segment $\tau'$ of type $1$ of $\gamma'$ such that $\tau'\cap \gamma'_j\neq \emptyset$ and it intersects the open ball $B(s_n^j,r/4))$ (see Figure~\ref{Fi2}). We will prove that some subsegment of $\gamma_j$ is an $(r/4)$--avoidant path over an $n$--corner in $X_{n+1}$. Let $f$ be an edge of $\tau'\cap \gamma'_j$ that lies inside the open ball $B(s_n^j,(r/4)+1))$. Then $f$ is labeled by $s_n$. Let $H$ be the hyperplane in $X_n$ that is dual the edge $f$. Then $H$ must intersect $\gamma'_j$. Let $x$ be the point in the intersection $H\cap \gamma'_j$. Then $x$ is an interior point of an edge $f'$ labeled by $s_n$ in $\gamma'_j$. Let $\alpha'$ be the path labeled by $s_1$ in $star(H)-ostar(H)$ that connect a vertex $w$ of $f$ to a vertex $w'$ of $f'$. If $f'$ is not an edge in any $X_n$-replacement segment of $\gamma'$, then $w'$ is a vertex of $\gamma$. In this case, we let $\widetilde{w}=w'$ and $\widetilde{\alpha}=\alpha'$. In the case $f'$ is an edge in some $X_n$-replacement segment $\tau''$ of $\gamma'$, we let $\widetilde{w}$ is an endpoint of $\tau''$ that belongs to $\gamma'_j$, let $\alpha''$ be the subsegment $\tau''$ connecting $w'$ and $\widetilde{w}$, and let $\widetilde{\alpha}=\alpha'\cup \alpha''$. Therefore, $\widetilde{\alpha}$ is a part of an $(1,n)$--ray. Let $\widetilde{v}$ be the endpoint of $\tau'$ that belongs to $\gamma_j$ and let $\widetilde{\beta}$ is a subsegment of $\tau'$ that connects $w$ and $\widetilde{v}$. Then $\widetilde{\beta}$ is a part of an $n$--ray and $(\widetilde{\alpha},\widetilde{\beta})_w$ is a part of an $n$--corner.

Let $\eta$ be the subsegment of $\gamma_j$ that connects $\widetilde{v}$ and $\widetilde{w}$. We note that $\eta$ lies outside the open ball $B(e,r)$ in $X_{n+1}$ and therefore it lies outside the open ball $B(s_n^j,3r/4)$ in $X_{n+1}$. Also, $d(s_n^j,w)<r/4+1\leq r/2$ if we assume that $r>4$. Therefore, $\eta$ lies outside the open ball $B(w,r/4)$. Thus, $\eta$ is an $(r/4)$--avoidant path over an $n$--corner in $X_{n+1}$. Also, $n\geq d+1>d$. Therefore, $$\ell(\gamma_j)\geq \ell(\eta)\geq g(r/4).$$ 

In both cases we have shown that $$d(v'_j,v_j)+\ell(\gamma_j)+d(u'_{j+1},u_{j+1})\geq g(r/4)/11.$$ Therefore, $$\ell(\gamma)\geq \sum_{r/8\leq j\leq r/4} \bigl(d(v'_j,v_j)+\ell(\gamma_j)+d(u'_{j+1},u_{j+1})\bigr)\geq (\frac{r}{16})\bigl(\frac{g(r/4)}{11}\bigr)\geq \frac{r}{180}g(r/4).$$ 
\end{proof}


 \begin{lem}
\label{p3p3}
For each integer $d\geq 2$ there is a positive number $n_d$ such that the following holds. Let $k\geq d$ be an integer and let $(\alpha,\beta)_x$ be a $k$--corner and let $\gamma$ be an $r$--avoidant path over the $k$--corner $(\alpha,\beta)_x$ in the complex $X_{k}$. Then the length of $\gamma$ is at least $(r^{d-1}/n_d)f(r/{n_d})$ for $r$ sufficiently large. 
\end{lem}

\begin{proof}
We first prove the following claim. For each integer $d\geq 2$ there is a positive number $n_d$ such that the following holds. Let $k\geq d$ be integers and let $(\alpha,\beta)_x$ be a $k$--corner. Let $\gamma$ be an $r$--avoidant path over the $k$--corner $(\alpha,\beta)_x$ in the complex $X_{k+1}$. Then the length of $\gamma$ is at least $(r^{d-1}/n_d)f(r/{n_d})$ for $r$ sufficiently large.

The above claim can be proved by induction on $d$. In fact, the claim is true for the base case $d=2$ due to Lemma~\ref{okokok}. Then Lemma~\ref{p1p1} establishes the inductive step and the above claim is proved. By Lemma~\ref{ll1}, we observe that if $\gamma$ is an $r$--avoidant path over a $k$--corner in the complex $X_k\subset X_{k+1}$ then $\gamma$ is also an $r$--avoidant path over the same $k$--corner in the complex $X_{k+1}$. Therefore, the lemma follows from the above claim.
\end{proof}

We now prove the lower bound for the divergence of the groups $G_m$ for $m\geq 3$.

\begin{prop}[Lower bound for $\Div_{G_m}$, $m\geq 3$]
\label{lowerkey1}
Let $m\geq 3$ be an integer. Let $\{\delta_\rho\}$ be the divergence of the Cayley graph $\Gamma(G_m,S_m)$. Then $r^{m-1}f(r)\preceq \delta_\rho(r)$ for each $\rho \in (0,1/2]$.
\end{prop}

\begin{proof}
Let $n_m$ be the positive number in Lemma~\ref{p3p3}. We will prove that $$\delta_\rho (r/\rho)\geq (r^{m-1}/n_m)/f(r/n_m)$$ for $r$ sufficiently large. Let $\alpha$ be a $1$--ray and let $\beta$ be a $m$--ray such that they share the initial point at the identity $e$. Then $(\alpha,\beta)_e$ is an $m$--corner. Let $\gamma$ be an arbitrary path which connects $\alpha(r/\rho)$ and $\beta(r/\rho)$ and lies outside the open ball $B(e, r)$. 
Then by Lemma~\ref{p3p3}, the length of the path $\gamma$ is bounded below by $(r^{m-1}/n_m)/f(r/n_m)$ for $r$ sufficiently large. This implies that $\delta_\rho (r/\rho)\geq (r^{m-1}/n_m)/f(r/n_m)$ for $r$ sufficiently large. Therefore, $r^{m-1}f(r)\preceq \delta_\rho(r)$ for each $\rho \in (0,1/2]$. 
\end{proof}

The group $G_2$ is a special case. Although the upper bound on $r$--avoidant paths over 2--corners in $X_2$ is $rf(r)$ (this can be seen by arguing similar to Lemma~\ref{cool2}), there are pairs of rays in $X_2$ which diverge at a quadratic rate. The next result describes two such rays and establishes a quadratic lower bound on their divergence. 

\begin{lem}
\label{pcoolcool}
Let $n$ be an arbitrary integer greater than $16$ and $s$ be a generator of $G_2$ in $H-F$. Let $\gamma$ be a path with endpoints $(s_2s)^{-n}$ and $(s_2s)^{n}$ which avoids the open ball $B(e,n)$. Then the length of $\gamma$ is at least $n^2/16$.
\end{lem}

\begin{proof}
For each $0\leq i \leq n/8$ let $e_i$ be an edge labeled by $s_2$ with endpoints $(s_2s)^{i}$ and $(s_2s)^{i}s_2$. Then the hyperplane $H_i$ of the complex $X_2$ that corresponds to $e_i$ intersects $\gamma$. Let $u_i$ be the point in this intersection such that the subpath of $\gamma$ connecting $u_i$ and $(a_2s)^{n}$ does not intersect $H_i$ at point other than $u_i$. Then $u_i$ is the midpoint of an edge $f_i$ of $\gamma$. 

Let $\alpha_i$ be the path in $star(H_i)-ostar(H_i)$ that connects $(s_2s)^{i}$ to an endpoint of $f_i$. Therefore, $\alpha_i$ traces a word representing a group element $g_i\in F_{xz}$ and the endpoint $v_i$ of $\alpha_i$ in $\gamma$ has the form $(s_2s)^{i}g_i$. Since $|(s_2s)^{i}|_{S_2}\leq 2i\leq n/2$ and $(s_2s)^{i}g_i$ lies outside the open ball $B(e,n)$, then $|g_i|_{S_2}\geq n-n/2\geq n/2$. Let $\beta_i$ be the path in $star(H_i)-ostar(H_i)$ that connects $(s_2s)^{i}s_2$ to an endpoint of $f_i$. Therefore, $\beta_i$ traces a word representing a group element $g'_i\in F_{yz}$ and the endpoint $w_i$ of $\beta_i$ in $\gamma$ has the form $(s_2s)^{i}s_2g'_i$. Since $|(s_2s)^{i}s_2|_{S_2}\leq 2i+1\leq n/2$ and $(s_2s)^{i}s_2g'_i$ lies outside the open ball $B(e,n)$, then $|g'_i|_{S_2}\geq n-n/2\geq n/2$.

For each $1\leq i \leq n/8$ let $\gamma_i$ be the subpath of $\gamma$ that connects $w_{i-1}$ and $v_i$. Therefore, the length of $\gamma_i$ is at least $d_{S_2}(w_{i-1}, v_i)$. Also, $d_{S_2}(w_{i-1}, v_i)=|w_{i-1}^{-1}v_i|_{S_2}=|g_{i-1}'^{-1}sg_i|_{S_2}$ and the length of the element $g_{i-1}'^{-1}sg_i$ in $G_2$ is $|g'_{i-1}|_{S_1} + 1  + |g_i|_{S_1}$. We see this as follows
\begin{eqnarray*}
d_{S_2}(e, g_{i-1}'^{-1}sg_i)  & = & d_{S_1}(e, g_{i-1}'^{-1}sg_i) \\
& = & d_{S'_1}(e, g_{i-1}'^{-1}sg_i) \\
& \geq & d_{S'_1}(e, g_{i-1}'^{-1}F) + d_{S'_1}(F, sF) + d_{S'_1}(F, g_i)\\
& = & |g'_{i-1}|_{S_1} + 1  + |g_i|_{S_1},
\end{eqnarray*}
where $S'_1=S_1-\{s_1\}$ is the generating set of the amalgamation factor $$H \ast_{\langle d_i=a_ib_i^{-1}\rangle}(F_x\times F_y\times F_z)$$ of $G_1$.
The first two equalities hold because the subgroup inclusions are isometric embeddings with the respective generating sets. The inequality holds from Bass-Serre theory (of free products with amalgamation). For the last equality, we have $d_{S'_1}(F, sF) = 1$ because $s \not\in F$. The remaining parts are easily seen from Lemma~\ref{ll0}, the fact $F$ is a subgroup of $H$, and the fact $H \ast_{\langle d_i=a_ib_i^{-1}\rangle}(F_x\times F_y\times F_z)$ is isometrically embedded into $G_1$. 

Therefore, $$\ell(\gamma_i)\geq|g'_{i-1}|_{S_1} + 1  + |g_i|_{S_1}=|g'_{i-1}|_{S_2} + 1  + |g_i|_{S_2}\geq n/2+n/2+1\geq n.$$
This implies that $$\ell(\gamma)\geq \sum_{1\leq i\leq n/8} \ell(\gamma_i)\geq n^2/16.$$
\end{proof}

We now prove the quadratic lower bound for the divergence of the group $G_2$.

\begin{prop}[Quadratic lower bound for $\Div_{G_2}$]
\label{lowerkey2}
Let $\{\delta_\rho\}$ be the divergence of the Cayley graph $\Gamma(G_2,S_2)$. Then $r^2\preceq \delta_\rho(r)$ for each $\rho \in (0,1/2]$.
\end{prop}

\begin{proof}
We will prove that $\delta_\rho (r/\rho)\geq r^2/256-2r/\rho$ for $r$ sufficiently large. Let $s$ be a generator of $G_2$ in $H-F $. Let $\alpha$ be bi-infinite geodesic containing the identity element $e$ with edges labeled by $s_2$ and $s$ alternately. Let $x$ and $y$ be the two points in the intersection $\alpha \cap S(e,r/\rho$). We assume that the subsegment of $\alpha$ from $e$ to $x$ traces each edge of $\alpha$ in the positive direction and the subsegment of $\alpha$ from $e$ to $y$ traces each edge of $\alpha$ in the negative direction. Let $\beta$ be an arbitrary path with endpoints $x$ and $y$ that lies outside the ball $B(e,r)$. 
Let $n$ be the largest integer such that $n\leq r/2$. Therefore, $n\geq r/2-1\geq r/4$ for $r$ sufficiently large. Let $\alpha_1$ be a subsegment of $\alpha$ that connects $(s_2s)^n$ to $x$. Therefore, $\alpha_1$ lies outside the open ball $B(e,n)$ and has the length bounded above by $r/\rho$. Similarly, let $\alpha_2$ be a subsegment of $\alpha$ that connects $(s_2s)^{-n}$ to $y$. Therefore, $\alpha_2$ lies outside the open ball $B(e,n)$ and has the length bounded above by $r/\rho$. Let $\gamma=\alpha_1 \cup \beta \cup \alpha_2$. Then, $\gamma$ is a path with endpoints $(s_2s)^{-n}$ and $(s_2s)^{n}$ which avoids the open ball $B(e,n)$. Therefore, the length of $\gamma$ is at least $n^2/16$ by Lemma~\ref{pcoolcool}. Therefore, $$\ell(\beta)\geq \ell(\gamma)-2r/\rho\geq n^2/16-2r/\rho\geq r^2/256-2r/\rho.$$ Thus, $\delta_\rho (r/\rho)\geq r^2/256-2r/\rho$ for $r$ sufficiently large. This implies that $r^2\preceq \delta_\rho(r)$.
\end{proof}



\subsection{Geodesic divergence}
\label{sg}
 In this subsection we establish the geodesic divergence statements in Theorem~\ref{main_thm}; namely, that ${\rm Div}^{G_m}_{\langle s_m\rangle}$ is equivalent to $r^{m-1}f(r)$ where $f$ is the inverse of the distortion function. We also prove that $\Div_{\langle s_1s_2\rangle}^{G_m}$ is equivalent to $rf(r)$. As an application we have that, for $f(r) \prec r$. the element $s_1s_2$ is Morse but not contracting in $G_m$.

\medskip

\noindent
\emph{The divergence of $\langle s_m\rangle$ in $G_m$.} 
The proof of the following proposition is similar to the proof of Proposition 5.1 in \cite{BT2020} although the proof for the case $m=2$ is slightly different.

\begin{prop}
\label{cyclicdivergence}
For each $m \geq 2$, the divergence of $\langle s_m \rangle$ in $G_m$ is equivalent to the function $r^{m-1}f(r)$. 
\end{prop}

\begin{proof}
Let $\alpha_m$ be a bi-infinite geodesic in the Cayley graph $\Gamma(G_m,S_m)$ corresponding to the subgroup $\langle s_m\rangle$.
Without loss of generality we can assume that $\alpha_m(0)=e$ and $\alpha_m(1)=s_m$. Let $\beta:[0,\infty)\to \Gamma(G_m,S_m)$ be a $1$--ray with $\beta(0)=e$. By Lemma~\ref{cool2} and Lemma~\ref{ag} (Statement ($Q_m$)) there is a number $M>0$ such that the following hold. Let $r>0$ be an arbitrary number. There are a path $\gamma_1$ outside the open ball $B(\alpha_m(0),r)$ connecting $\alpha_m(-r)$ and $\beta(r)$ and a path $\gamma_2$ outside the open ball $B(\alpha_m(0),r)$ connecting $\alpha_m(r)$ and $\beta(r)$ such that the lengths of $\gamma_1$ and $\gamma_2$ are both bounded above by $Mr^{m-1} \bigl(f(Mr)+1\bigr)$. Therefore, the path $\gamma=\gamma_1\cup\gamma_2$ lies outside the open ball $B(\alpha_m(0),r)$, connects $\alpha_m(-r)$ and $\alpha_m(r)$, and has length at most $2Mr^{m-1} \bigl(f(Mr)+1\bigr)$. This implies that ${\rm Div}_{\alpha_m}(r)\leq 2Mr^{m-1} \bigl(f(Mr)+1\bigr)$.

We now prove a lower bound for ${\rm Div}_{\alpha_m}$. The proof for the case $m=2$ is analogous to the proof of Lemma~\ref{okokok} and we leave it to the reader for checking the details. We now prove a lower bound for ${\rm Div}_{\alpha_m}$ for $m\geq 3$.
Let $n_m$ be the constant in Lemma~\ref{p3p3}. Let $\gamma'$ be an arbitrary path outside the open ball $B(\alpha_m(0),r)$ connecting $\alpha_m(-r)$ and $\alpha_m(r)$. Let $e_1$ be the edge of $\alpha_m$ with endpoints $e$ and $s_m$. Then the hyperplane $H$ of the complex $X_m$ corresponding to $e_1$ must intersect $\gamma'$. Therefore, there is a $1$--ray $\beta_1$ with initial point at $e$ that intersects $\gamma'$ at some vertex $v$. This implies that the subpath $\gamma_1$ of $\gamma'$ connecting $\alpha_m(r)$ and $v$ is an $r$--avoidant path over the $m$--corner $({\alpha_m}_{|[0,\infty)},\beta_1)_e$. By Lemma~\ref{p3p3} for $r$ sufficiently large we have $$\ell(\gamma')\geq \ell(\gamma_1)\geq (r^{m-1}/n_m)f(r/n_m).$$ This implies that ${\rm Div}_{\alpha_m}(r)\geq (r^{m-1}/n_m)f(r/n_m)$ for $r$ sufficiently large. Therefore, the divergence of $\alpha_m$ is equivalent to the function $r^{m-1}f(r)$.
\end{proof}

\medskip

\noindent
\emph{The divergence of $\langle s_1s_2\rangle$ in $G_m$.} 
For each $m \geq 2$ we let $\alpha$ be a bi-infinite geodesic with edges labeled $s_1$ and $s_2$ in $\Gamma(G_m,S_m)$ corresponding to the infinite cyclic subgroup $\langle s_1s_2\rangle$. In the remainder of this section, we prove that the divergence of $\alpha$ in $\Gamma(G_m,S_m)$ for $m\geq 2$ is equivalent to $rf(r)$. This implies that if $F$ is nonlinearly distorted in $H$ (equivalently $f(r) \prec r$), then the group element $s_1s_2$ is Morse but not contracting in $G_m$ for $m\geq 2$. Therefore, by Theorem~2.14 of \cite{MR3339446}, the groups $G_m$ are not $\CAT(0)$ for $m\geq 2$. 

In what follows, we assume that $\alpha(0)=e$ and $\alpha(1)=s_1$. The next two lemmas establish the upper bound, $\Div_{\langle s_1s_2\rangle}^{G_m} \preceq rf(r)$. 

\begin{lem}
\label{coolcoolcool}
Let $r_0 \geq 1$ and $D$ be the constants in Remark~\ref{bbb} and let $r\geq r_0$ be an arbitrary integer. Then there is a group element $g_0$ in $F_{xz}$ such that the following hold.
\begin{enumerate}
    \item $4r\leq |g_0|_{S_2}\leq 4Dr$; and
    \item There is a path $\beta$ in $\Gamma(G_2,S_2)$ connecting $g_0$ and $s_2g_0$ with length at most $Df(4Dr)+1$ and avoids the open ball $B(e,2r)$. 
\end{enumerate}
\end{lem}

\begin{proof}
Let $m=4Dr$. By Lemma~\ref{a1}, there are group elements $g_0$ in $F_{xz}$ and $g_1$ in $F_{yz}$ such that the following hold:
\begin{enumerate}
    \item $s_2^{-1}g_0s_2=g_1$; 
    \item $m/D\leq |g_0|_{R_{xz}}\leq m$ and $m/D\leq |g_1|_{R_{yz}}\leq m$;
    \item There is a path $\gamma$ connecting $g_0$ and $g_1$ with the length at most $Df(m)$ and each vertex in $\gamma$ is a group element $g_0h$ for $h\in H$;
\end{enumerate}
By Lemma~\ref{ll0} and Lemma~\ref{ll1} we have $|g_0|_{R_{xz}} = |g_0|_{S_1}=|g_0|_{S_2}$ and $|g_1|_{R_{yz}} = |g_1|_{S_1}=|g_1|_{S_2}$. Since $m/D=4r$, we have
$$4r\leq |g_0|_{S_2}\leq 4Dr \text{ and } 4r\leq |g_1|_{S_2}\leq 4Dr.$$
Another application of Lemma~\ref{ll0} and Lemma~\ref{ll1} implies that the path $\gamma$ lies outside the open ball $B(e,4r)$ in the group $G_2$. 

Let $\beta_1=s_2\gamma$. Then the path $\beta_1$ connects two points $s_2g_1$ and $s_2g_0$, lies outside the open ball $B(s_1,4r)$ (therefore outside the open ball $B(e,2r)$) and has length at most $Df(m)$. Also, $s_2g_1=g_0s_2$. Then we connect $g_0$ and $s_2g_1$ by an edge $e_1$ labeled by $s_2$. It is clear that this edge also lies outside the open ball $B(e,2r)$. Therefore, $\beta=e_1\cup \beta_1$ is a desired path. 
\end{proof}

\begin{lem}
\label{okokok0}
Let $m\geq 2$ be an integer. Then the divergence of $\alpha$ in $\Gamma(G_m,S_m)$ is dominated by the function $rf(r)$.
\end{lem}

\begin{proof}
By Lemma~\ref{ll1} the path avoids the $r$--ball about $e$ in $G_2$ also lies outside the $r$--ball about $e$ in $G_m$ for $m\geq 3$. Therefore, we only prove the above lemma for the case $m=2$. Let $D$ and $r_0$ be constants in Remark~\ref{bbb}. We will prove that 
$$\Div_\alpha(r)\leq 2r\bigl(Df(4Dr)+1\bigr)+8Dr\text{ for each $r\geq r_0$}\,.$$
We assume that $r$ is an integer. Let $x_i=\alpha(i)$ for each $i\in [-r,r]$. Let $g_0\in F_{xz}$ be a group element in Lemma~\ref{coolcoolcool}. Then
\begin{enumerate}
    \item $4r\leq |g_0|_{S_2}\leq 4Dr$; and
    \item There is a path $\beta$ in $\Gamma(G_2,S_2)$ connecting $g_0$ and $s_2g_0$ with length at most $Df(4Dr)+1$ and avoids the open ball $B(e,2r)$. 
\end{enumerate}

We will show that for each $r\in [-r,r-1]$ there is a path $\beta_i$ which connects $x_ig_0$ and $x_{i+1}g_0$, lies outside the open ball $B(e,r)$, and has length at most $Df(4Dr)+1$. Since edges of $\alpha$ are labeled by $s_1$ and $s_2$ alternately, we see that $x_{i+1}=x_is_1$ or $x_{i+1}=x_is_2$. If $x_{i+1}=x_is_1$, then 
$$x_{i+1}g_0=(x_is_1)g_0=x_i(s_1g_0)=x_i(g_0s_1)=(x_ig_0)s_1\,.$$
There we connect $x_ig_0$ and $x_{i+1}g_0$ by an edge $\gamma_i$ labelled $s_1$ and this edge obviously lies outside the open ball $B(x_i, 4r-1)$. Since $d_{S_2}(e,x_i)\leq |i|\leq r$, the edge $\beta_i$ also lies outside the open ball $B(e,4r-2)$ (therefore outside $B(e,r)$). We now assume that $x_{i+1}=x_is_2$. By Lemma~\ref{coolcoolcool} there is a path $\beta_i$ (a translate of the edge $\beta$ above by $x_i$) connecting $x_ig_0$ and $x_{i+1}g_0$ with length at most $Df(4Dr)+1$ and avoids the open ball $B(x_i,2r)$. Again, we note that $d_{S_2}(e,x_i)\leq |i|\leq r$. Therefore, $\beta_i$ also lies outside the open ball $B(e,r)$.

Let $\gamma_0=\bigcup_{-r\leq i\leq r-1} \beta_i$. Then $\gamma_0$ lies outside the open ball $B(e,r)$, connects two points $x_{-r}g_0$ and $x_{r}g_0$, and has length at most $2r\bigl(Df(4Dr)+1\bigr)$. Let $\eta_1$ be the path labeled by the word in $R_{xz}$ representing $g_0$ which connects $x_{-r}$ and $x_{-r}g_0$. Then the length of $\eta_1$ is exactly $|g_0|_{R_{xz}}\leq 4Dr$. We observe that each vertex in $\eta_1$ is a group element $gu$ where $u\in F_{yz}$ and $g$ is a group element with length exactly $r$ in the subgroup $K$ generated by $\{s_1,s_2\}$. Also the group automorphism $G_{2} \to K$ taking each $s_1$ and $s_2$ to itself, and all other generators of $G_{2}$ to the identity shows that $gu$ lies outside the open ball $B(e,r)$. In other words, $\eta_1$ lies outside the open ball $B(e,r)$. Similarly, there is a path $\eta_2$ which connects $x_r$ and $x_rg_0$, lies outside the open ball $B(e,r)$, and has length at most $4Dr$. 

Let $\gamma=\eta_1\cup \gamma_0 \cup \eta_2$. Then $\gamma$ connects $\alpha(-r)=x_{-r}$ and $\alpha(r)=x_{r}$, lies outside the open ball $B(e,r)$, and has length at most $2r\bigl(Df(4Dr)+1\bigr)+8Dr$. Therefore, $$\Div_\alpha(r)\leq 2r\bigl(Df(4Dr)+1\bigr)+8Dr\,.$$ This implies that the divergence of $\alpha$ is dominated by $rf(r)$.

\end{proof}

The lower bound, $rf(r) \preceq \Div_{\langle s_1s_2 \rangle}^{G_m}$ is proven by induction on $m$. The next lemma establishes the base cases $m=3$ and (as a consequence) $m=2$. 

\begin{lem}
\label{okokok1}

Let $\gamma$ be the path in $X_3$ which lies outside the open ball $B(e,r)$ and connects a vertex in the ray $\alpha[0,-\infty)$ to a vertex in the ray $\alpha[0,\infty)$. Then the length of $\gamma$ is at least $(r/16)f(r/4)$ for $r$ sufficiently large. 
\end{lem}

\begin{proof}

Suppose that $\gamma$ connects a vertex $u$ in $\alpha[0,-\infty)$ to a vertex $v$ in $\alpha[0,\infty)$. As in the proof of Lemma~\ref{okokok}, $\gamma$ can be written as a concatenation 
 $$\sigma_1\tau_1\sigma_2\tau_2\cdots \sigma_{\ell}\tau_{\ell}\sigma_{\ell+1}$$ such that:
\begin{enumerate}
    \item Each $\tau_i$ intersects $X_{2}$ only at its endpoints $x_i =  \sigma_i \cap \tau_i $ and $y_i = \tau_i \cap \sigma_{i+1}$; 
    \item Each $\sigma_i$ lies completely in the $1$--skeleton of $X_{2}$.
\end{enumerate}

We observe that each $x_i^{-1}y_i$ is a group element in the cyclic subgroup $\langle s_2 \rangle$ or in the cyclic subgroup $\langle s_{1} \rangle$. If $x_i^{-1}y_i$ is a group element in the cyclic subgroup $\langle s_2 \rangle$, then we replace $\tau_i$ by a geodesic $\tau'_i$ labeled by $s_2$. Otherwise, $x_i^{-1}y_i$ is a group element in the cyclic subgroup $\langle s_1 \rangle$, then we replace $\tau_i$ by a geodesic $\tau'_i$ labeled by $s_1$. The new path $\gamma'=\sigma_1\tau'_1\sigma_2\tau'_2\cdots \sigma_{\ell}\tau'_{\ell}\sigma_{\ell+1}$ lies completely in the $1$--skeleton of $X_2$, has the same initial and end points as $\gamma$, and $\ell(\gamma')\leq \ell(\gamma)$. Note that the new path $\gamma'$ may intersect the open ball $B(e,r)$. We call each subpath $\tau'_i$ of $\gamma'$ a short-cut segment.


For each positive integer $j$ we call $e_j$ the $(2j-1)^{th}$ edge of the ray $\alpha[0,\infty)$. We know that each edges $e_j$ is labeled by $s_2$ with endpoints $(s_1s_2)^{j-1}s_1$ and $(s_1s_2)^{j}$. Let $H_j$ be the $s_2$-hyperplane of the complex $X_2$ dual to the edge $e_j$. We assume that $r\geq 8$ and we consider the edges $e_j$ for $1\leq j\leq r/8$. Then each hyperplane $H_j$ intersect $\gamma'$ and we let $m_j$ be the point in the intersection $H_j\cap \gamma'$ such that the subpath of $\gamma'$ connecting $m_j$ and $v$ does not intersect $H_j$ at point other than $m_j$. Note that $m_j$ is an interior point of an edge $f_j$ labeled by $s_2$ in $\gamma'$. 

Let $\beta_j$ be the path in $star(H_j)-ostar(H_j)$ that connects the terminal $(s_1s_2)^j$ of the edge $e_j$ to some endpoint $v'_j$ of the edge $f_j$. Then edges of $\beta_j$ are labeled by generators of the free group $F_{yz}$. If $f_j$ is not an edge of a short-cut segment of $\gamma'$, then we let $v_j=v'_j$. Otherwise, $f_j$ is an edge in some short-cut segment $\tau'_{\ell_j}$ of $\gamma'$. In this case, we let $v_j$ be the endpoint $y_{\ell_j}$ of $\tau'_{\ell_j}$. Let $\alpha_j$ be the path in $star(H_{j+1})-ostar(H_{j+1})$ that connects the initial endpoint $(s_1s_2)^js_1$ of the edge $e_{j+1}$ to some endpoint $u'_{j+1}$ of the edge $f_{j+1}$. Then edges of $\alpha_j$ are labeled by generators of $F_{xz}$. If $f_{j+1}$ is not an edge of a short-cut segment of $\gamma'$, then we let $u_{j+1}=u'_{j+1}$. Otherwise, $f_{j+1}$ is an edge in some $\tau'_{\ell_{j+1}}$. In this case, we let $u_{j+1}$ be the endpoint $x_{\ell_{j+1}}$ of $\tau'_{\ell_{j+1}}$.

Let $\gamma'_j$ be the subpath of $\gamma'$ that connects $v'_j$ and $u'_{j+1}$. We will prove that the length of $\gamma'_j$ is at least $f(r/4)$. By the construction, the path $\gamma'_j$ contains two subpaths $\eta_1$ and $\eta_2$ of $\gamma'$, where $\eta_1$ (resp. $\eta_2$) is the (possibly degenerate) subsegment of $\gamma'_j$ connecting $v'_j$ and $v_j$ (resp. $u'_{j+1}$ and $u_{j+1}$). If either the length of $\eta_1$ or the length of $\eta_2$ is greater than $r/4$, then the length of $\gamma'_j$ is greater than $r/4$ and therefore greater than $f(r/4)$. 

We now assume that the lengths of $\eta_1$ and $\eta_2$ are both less than or equal to $r/4$. Therefore, $$d_{S_2}(v'_j,v_j)\leq r/4 \text{ and } d_{S_2}(u'_{j+1},u_{j+1})\leq r/4.$$ Also, $v_j$ and $u_{j+1}$ lies outside the open ball $B(e,r)$. Then $v'_j$ and $u'_{j+1}$ lies outside the open ball $B(e,3r/4)$. We note that $v'_j=(s_1s_2)^jg_1$ for some $g_1\in F_{yz}$ and $u'_{j+1}=(s_1s_2)^js_1g_0$ for some $g_0\in F_{xz}$. Moreover, $d_{S_2}(e,(s_1s_2)^j)=2j\leq r/2$ and $d_{S_2}(e,(s_1s_2)^js_1)=2j+1\leq r/2$. Therefore, we have $|g_0|_{S_2}\geq r/4$ and $|g_1|_{S_2}\geq r/4$. By Lemma~\ref{ll1} and Lemma~\ref{st1} we have $$\ell(\gamma'_j)\geq d_{S_2}(v'_j,u'_{j+1})=|g_1^{-1}s_1g_0|_{S_2}=|(g_1^{-1}g_0)s_1|_{S_1}=|(g_1^{-1}g_0)|_{S_1}+1\geq f(r/4)+1\geq f(r/4)\,.$$
Therefore, $$\ell(\gamma)\geq \ell(\gamma')\geq \sum_{1\leq j\leq r/8}\ell(\gamma'_j)\geq (r/16)f(r/4).$$

\end{proof}

Here is the proof that $rf(r) \preceq \Div_{\langle s_1s_2\rangle}^{G_m}$ for all $m \geq 2$. 

\begin{lem}
\label{okokok2}
For each integer $m\geq 2$ there is a positive integer $k_m$ such that the following hold. Let $\gamma$ be the path in $X_m$ which lies outside the open ball $B(e,r)$ and connects a vertex in the ray $\alpha[0,-\infty)$ to a vertex in the ray $\alpha[0,\infty)$. Then the length of $\gamma$ is at least $(r/k_m)f(r/k_m)$.
\end{lem}

\begin{proof}
We prove the above lemma by induction on $m$. By Lemma~\ref{okokok1} the above lemma is true for the case $m=2$ and $m=3$ with $k_2=k_3=16$. We assume the above lemma is true for $m\geq 3$ and we will prove that it is also true for $m+1$.

Suppose that $\gamma$ is a path outside then open ball $B(e,r)$ in $X_{m+1}$ which connects a vertex $u$ in $\alpha[0,-\infty)$ to a vertex $v$ in $\alpha[0,\infty)$. As in the proof of Lemma~\ref{okokok}, $\gamma$ can be written as a concatenation 
 $$\sigma_1\tau_1\sigma_2\tau_2\cdots \sigma_{\ell}\tau_{\ell}\sigma_{\ell+1}$$ such that:
\begin{enumerate}
    \item Each $\tau_i$ intersects $X_{m}$ only at its endpoints $x_i =  \sigma_i \cap \tau_i $ and $y_i = \tau_i \cap \sigma_{i+1}$; 
    \item Each $\sigma_i$ lies completely in the $1$--skeleton of $X_{m}$.
\end{enumerate}

We also observe that $x_i^{-1}y_i$ is a group element in the cyclic subgroup $\langle s_m \rangle$ or in the cyclic subgroup $\langle s_{1} \rangle$ for each $i$. 

If $x_i^{-1}y_i$ is a group element in the cyclic subgroup $\langle s_m \rangle$, then we replace $\tau_i$ by a geodesic $\tau'_i$ labeled by $s_m$. Call this path $\tau'_i$ an $X_m$-replacement segment of type 1. Note that $\tau'_i$ in this case may have non-empty intersection with the open ball $B(e,r/4)$. 

If $x_i^{-1}y_i$ is a group element in the cyclic subgroup $\langle s_1 \rangle$, then by Lemma~\ref{avoidantlem} we can replace $\tau_i$ a path $\tau'_i$ with edges labeled by $s_1$ and $a_1$ such that $\tau'_i$ lies outside the open ball $B(e,r/2)$ (therefore outside the open ball $B(e,r/4)$) and $\ell(\tau'_i)\leq 11\text{ }d_{S_m}(x_i,y_i)\leq 11 \ell (\tau_i)$. Call this path an $X_m$-replacement segment of type 2. 

The new path $\gamma'=\sigma_1\tau'_1\sigma_2\tau'_2\cdots \sigma_{\ell}\tau'_{\ell}\sigma_{\ell+1}$ lies completely in the $1$--skeleton of $X_m$, shares two endpoints $u$ and $v$ with $\gamma$, and $\ell(\gamma')\leq 11 \ell(\gamma)$. We first assume that $\gamma'$ lies outside the open ball $B(e,r/4)$. By inductive hypothesis we see that $$\ell(\gamma')\geq \bigl(\frac{r}{4k_m}\bigr)f\bigl(\frac{r}{4k_m}\bigr)\text{ for $r$ sufficiently large}\,.$$
This implies that $$\ell(\gamma)\geq \frac{1}{11}\ell(\gamma')\geq \frac{1}{11}\bigl(\frac{r}{4k_m}\bigr)f\bigl(\frac{r}{4k_m}\bigr)\geq \bigl(\frac{r}{44k_m}\bigr)f\bigl(\frac{r}{44k_m}\bigr)\,.$$

We now assume that $\gamma'$ intersect the open ball $B(e,r/4)$.
Then there is an $X_m$-replacement segment $\tau'$ of type $1$ of $\gamma'$ that intersects the open ball $B(e,r/4)$. Let $f$ be an edge of $\tau'$ that lies in the open ball $B(e,r/4+1)\subset B(e,r/2)$. Then $f$ is label by $s_{m}$. Let $H$ be the hyperplane in $X_m$ that is dual the edge $f$. Then $H$ must intersect $\gamma'$. Let $x$ be the point in the intersection $H\cap \gamma'$. Then $x$ is an interior point of an edge $f'$ labeled by $s_m$ in $\gamma'$. Let $\alpha'$ be the path labeled by $s_1$ in $star(H)-ostar(H)$ that connect a vertex $w$ of $f$ to a vertex $w'$ of $f'$. If $f'$ is not an edge in any $X_m$-replacement segment of $\gamma'$, then $w'$ is a vertex of $\gamma$. In this case, we let $\widetilde{w}=w'$ and $\widetilde{\alpha}=\alpha'$. In the case $f'$ is an edge in some $X_m$-replacement segment $\tau''$ of $\gamma'$, we let $\widetilde{w}$ is an endpoint of $\tau''$, let $\alpha''$ be the subsegment $\tau''$ connecting $w'$ and $\widetilde{w}$, and let $\widetilde{\alpha}=\alpha'\cup \alpha''$. Therefore, $\widetilde{\alpha}$ is a part of an $(1,m)$--ray. Let $\widetilde{v}$ be the endpoint of $\tau'$ and let $\widetilde{\beta}$ is a subsegment of $\tau'$ that connects $w$ and $\widetilde{v}$. Then $\widetilde{\beta}$ is a part of an $m$--ray and therefore $(\widetilde{\alpha},\widetilde{\beta})_w$ is a part of an $m$--corner.

Let $\eta$ be the subsegment of $\gamma$ that connects $\widetilde{v}$ and $\widetilde{w}$. We note that $\eta$ lies outside the open ball $B(e,r)$ in $X_{m+1}$. Also, $w$ lies in the open ball $B(e,r/2)$. Therefore, $\eta$ lies outside the open ball $B(w,r/2)$ in $X_{m+1}$. In other words, $\eta$ is an $(r/2)$--avoidant path over an $m$--corner in $X_{m+1}$. By the proof of Lemma~\ref{okokok} we have 
$$\ell(\gamma)\geq \ell(\eta)\geq \bigl(\frac{r}{32}\bigr)f\bigl(\frac{r}{8}\bigr)\geq \bigl(\frac{r}{32}\bigr)f\bigl(\frac{r}{32}\bigr)\,.$$
for $r$ sufficiently large. Therefore, the lemma is true for $m+1$ with $k_{m+1}=\max\{44k_m,32\}$ and the lemma is proved.
\end{proof}

\begin{prop}
\label{ppppp3}
For each $m\geq 2$ the divergence of the cyclic subgroup $\langle s_1s_2\rangle$ in $G_m$ is equivalent to $rf(r)$. In particular, if $F$ is distorted in $H$ then the group element $s_1s_2$ is Morse but not contracting and therefore $G_m$ is not a $\CAT(0)$ group. 
\end{prop}

\begin{proof}
By Lemma~\ref{okokok0} and Lemma~\ref{okokok2} the divergence of the cyclic subgroup $\langle s_1s_2\rangle$ is equivalent to $rf(r)$. We now assume that $F$ is distorted in $H$. Therefore, the divergence of the cyclic subgroup $\langle s_1s_2\rangle$ is super linear but strictly by a quadratic function. By Theorem 1.3 in \cite{ACGH} and Proposition~\ref{lem_CS} the group element $s_1s_2$ is Morse but not contracting. This implies that $G_m$ is not a $\CAT(0)$ group by Theorem 2.14 in \cite{MR3339446}.
\end{proof}

\begin{rem}
In section~\ref{sbasic} we show that for the particular examples of $\CAT(0)$ groups $B_m$ and subgroups $G_m$, none of the groups $G_m$ ($m\geq 1$) are $\CAT(0)$. 
\end{rem}

\section{Embeddings in $\CAT(0)$ groups}
\label{sbasic}

The main embedding result of this section relies on $\CAT(0)$ free-by-cyclic groups with palindromic monodromies. These are described in the first subsection. The second subsection includes the construction of a sequence of $\CAT(0)$ groups $B_m$ which mirrors the construction of the groups $G_m$ in Theorem~\ref{main_thm}. The groups $B_m$ are shown to contain the $G_m$ of Theorem~\ref{main_thm}. In the third subsection we provide explicit examples of $B_m$ containing $G_m$ with divergence $r^\alpha$ and $r^q\log(r)$. 

We refer the reader to \cite{MR1744486} for definitions and background of non-positively curved (n.p.c.) and $\CAT(0)$ metric spaces and $\CAT(0)$ groups. 

\subsection{$\CAT(0)$ Free-by-cyclic building blocks}
\label{sec:fiveone}
In \cite{MR1334311}, T. Brady showed that each group of the form $F_2 \rtimes_{\varphi} \Z$ is CAT(0) by describing how to build a non-positively curved, piecewise euclidean 2--complex with this as fundamental group. Many of the structures in \cite{MR1334311} enjoy additional symmetry properties which reflect symmetries of the corresponding automorphisms of $F_2$. 

In this section we will make extensive use of the palindromic automorphism 
$$
\varphi: F_{\{a,b\}} \to F_{\{a,b\}}: a \mapsto aba, \; b \mapsto a
$$
and the corresponding free-by-cyclic group
$$F\rtimes_{\varphi}\langle t_0 \rangle=\langle a,b,t_0|t_0at_0^{-1}=aba, t_0bt_0^{-1}=a\rangle$$
as well as its index $n$ subgroup $\langle a, b, t_0^n\rangle$.
In \cite{MR3705143} the non-positively curved, piecewise euclidean
2--complex $Z_0$ corresponding to the automorphism $\varphi$ was seen to admit an isometric involution 
$\tau_{Z_0}\!:Z_0\to Z_0$ which induces on automorphism $(\tau_{Z_0})_\ast: \pi_1(Z_0) \to \pi_1(Z_0)$ on the fundamental group taking $a$ to $a^{-1}$, $b$ to $b^{-1}$, and $t_0$ to itself.
This lifts to an isometric involution $\tau_Z: Z \to Z$ of the $n$--fold cyclic cover $Z$ of $Z_0$ corresponding to the subgroup $A = F \rtimes_{\varphi^n}\langle t\rangle$ where $t=(t_0)^n$. The map $\tau_Z$ induces an automorphism 
$\tau_A: A \, \to \, A: a \mapsto a^{-1}, \, b\mapsto b^{-1}, \, t \mapsto t$.
 
 We will build new non-positively curved complexes by isometrically glueing simpler non-positively curved complexes along isometrically embedded (scaled) copies of $Z$. Since an isometry $Z_0 \to Z_0$ lifts (modulo choice of basepoints) to a unique isometry $Z  \to Z$, we can describe these glueing maps on the level of $Z_0$. Therefore, it will be useful to have a notation for different copies of $Z_0$. To this end, let 
 $Z_{0,x}$ denote a copy of $Z_0$ with edges relabeled as follows
 $$
 a \longrightarrow x_1\, ,  \quad \quad 
b \longrightarrow x_2\, , \quad \quad 
t_0 \longrightarrow t_{0,x}\, .
 $$
 This relabeling map determines an isometry $Z_0 \to Z_{0,x}$. This lifts to a unique (modulo choice of basepoints) isometry of the $n$--fold cyclic covers 
 $Z \to Z_x$. On the level of fundamental groups this induces the relabeling isomorphism $A \to A_x$ where 
 $$
 A \;=\; \langle a, b, t \, |\,  tat^{-1} = \varphi^n(a), tbt^{-1} = \varphi^n(b) \rangle$$ and $$ 
  A_x \; =\; \langle x_1, x_2, t_x \, |\,  t_xx_1t_x^{-1} = \varphi^n(x_1), t_xx_2t_x^{-1} = \varphi^n(x_2) \rangle\, .
 $$
 Likewise, we define spaces $Z_y$, $Z_z$, $Z_d$, $Z_{xz}$, $Z_{yz}$, and $Z_{x\bar{y}}$. 
Double subscript labels are defined as follows
$$
 a \longrightarrow x_1z_1\, ,  \quad \quad 
b \longrightarrow x_2z_2\, , \quad \quad 
t_0 \longrightarrow t_{0,x}t_{0,z}
 $$
 and
  $$
 a \longrightarrow x_1\bar{y}_1\, ,  \quad \quad 
b \longrightarrow x_2\bar{y}_2\, , \quad \quad 
t_0 \longrightarrow t_{0,x}t_{0,y}\, .  
 $$
 They induce relabeling isomorphisms 
 $$A \; \to\;  A_{xz}:\;  a \mapsto x_1z_1, \; b \mapsto x_2z_2,\;  t \mapsto t_xt_z$$
 and
  $$A \; \to\;  A_{x\bar{y}}:\;  a \mapsto x_1y_1^{-1}, \; b \mapsto x_2y^{-1}_2,\;  t \mapsto t_xt_y\, .$$
 
 

\subsection{The $\CAT(0)$ construction}
\label{sec:fivetwo}

The goal of this subsection is to construct a sequence of $\CAT(0)$ groups $B_m$ which will be shown to contain the groups $G_m$ of Theorem~\ref{main_thm} as subgroups. The group $B_m$ are constructed inductively via amalgams and HNN extensions, mirroring the construction of the $G_m$. Note that $G_1$ is defined by starting with an amalgamation 
$$
H \; \ast_{\langle d_i= x_iy^{-1}_i\rangle} (F_x \times F_y \times F_z)
$$
where the $F_d = \langle d_1, \ldots, d_p\rangle$ subgroup has non-trivial distortion in $H$. The $\CAT(0)$ analogue will be constructed as an amalgamation over an undistorted edge subgroup. This is achieved by embedding $F_d$ as a distorted fiber in a $\CAT(0)$ free-by-cyclic group $A_d$ and using $A_d$ as the edge group for the $\CAT(0)$ amalgam. We work in the case $p=2$. Here are the details.


Recall that the $G_m$ are defined inductively via HNN extensions starting from $G_1$. The group $G_1$ is built from an amalgam of a group $H$ and a direct product of three free groups amalgamated over a free group. We describe $\CAT(0)$ groups containing each of these components. 

\smallskip
\noindent
\emph{The $\CAT(0)$ vertex group containing the product of free groups.}
Let $F_x$, $F_y$, $F_z$ be rank-two free groups generated by the sets $\{x_1,x_2\}$, $\{y_1,y_2\}$, and $\{z_1,z_2\}$ respectively. Let $\varphi$ be the automorphism of section~\ref{sec:fiveone}, and define $A_x=F_x\rtimes_{\varphi^n}\langle t_x\rangle$, $A_y=F_y\rtimes_{\varphi^n}\langle t_y\rangle$, and $A_z=F_z\rtimes_{\varphi^n}\langle t_z\rangle$. Note that $A_x$ is the fundamental group of the locally $\CAT(0)$ space $Z_x$ described in section~\ref{sec:fiveone}, and similarly for $A_y$ and $A_z$. 
Note that 
\begin{equation}\label{eq1}
  (F_x \times F_y \times F_z) \, \cap \, A_{x\bar{y}} \; =\; F_{x\bar{y}}  
\end{equation}
This identity will be used in proving that $G_1 \leq B_1$ later on. 

\smallskip

\noindent
\emph{The $\CAT(0)$ vertex group containing $H$.} 
Let $X$ be a locally $\CAT(0)$ metric space whose fundamental group is $C$. Assume that $X$ contains an isometric embedding copy $Z_d$ of $Z$ with the fundamental group $A_d=F_d(d_1,d_2)\rtimes_{\varphi^n}\langle t_d\rangle$. We also assume that the group $C$ contains the group $H$ in $G_1$ such that 
\begin{equation}\label{eq2}
H\, \cap \, A_d\; =\; F_d
\end{equation}
This identity will also be used in the proof that $G_1 \leq B_1$ later on.
In section~\ref{sec:fivethree} we give two explicit examples of $X$ and $C$. For now the reader can just work with the assumption that spaces and group with these properties exist. 

\smallskip

\noindent
\emph{The $\CAT(0)$ edge group.} The $\CAT(0)$ edge group $A_d = F_d \rtimes_{\varphi^n}\Z$ contains the edge group $F_d$ in the amalgamation in the definition of $G_1$. Note that $A_d$ is undistorted in each of the vertex groups above. However, the exponential distortion of $F_d$ in $A_d$ gives room for the distortion of $F_d$ in the vertex group $H$.

We define the family of groups $(B_m)$ by induction on $m$ as follows:
\begin{align*}
    & B_1 \; =\; \bigl[C\ast_{\langle d_i=x_iy_i^{-1},t_d=t_xt_y\rangle} (A_x\times A_y \times A_z)\bigr] \times {\langle s_1 \rangle};\\& B_2\;=\; \langle B_1, s_2 | s_2^{-1}(x_iz_i)s_2=y_{i}z_i, s_2^{-1}(t_xt_z)s_2=t_yt_z\rangle; \\
    & {\hbox{and}}\\
    & B_m \; =\; \langle B_{m-1}, s_m | s_m^{-1}s_1s_m=s_{m-1}\rangle \quad \quad {\hbox{ for $m \geq 3$.}}
\end{align*}

The following proposition will be used in the proof that each group $B_m$ is $\CAT(0)$.

\begin{prop}[\cite{MR1744486}, II.11.13]
\label{b1}
Let $X$ and $Y$ be locally $\CAT(0)$ metric spaces. If $Y$ is compact and $f,g:\! Y \to X$ are locally isometric immersions, then the quotient of $X \bigsqcup (Y \times [0,1])$ by the equivalence relation generated by $(y,0) \sim f(y);(y,1) \sim g(y), \forall y \in Y$ is locally $\CAT(0)$.
\end{prop}


\begin{prop}
\label{cat(0)}
$B_m$ is $\CAT(0)$ for $m \geq 1$. 
\end{prop}

\begin{proof}
There are three cases to consider; $m=1$, $m=2$, and $m \geq 3$.

\smallskip

\noindent
\emph{Case $m=1$.} Let $Z_x$, $Z_y$, and $Z_z$ be copies of $Z$ whose fundamental groups are $A_x$, $A_y$, and $A_z$ respectively. We equip $Z_x\times Z_y\times Z_z$ with the product metric. Then $Z_x\times Z_y\times Z_z$ is a locally $\CAT(0)$ space and $\pi_1(Z_x\times Z_y\times Z_z)=A_x\times A_y\times A_z$.

Define a map $g\!: Z \to Z_x\times Z_y$ by $g(p)=(p,\tau_Z(p))$. We observe that the induced homomorphism $g_*\!:A\to A_x\times A_y$ is given by $a\mapsto x_1y_1^{-1}$, $b\mapsto x_2y_2^{-1}$, and $t\mapsto t_xt_y$. Metrically, the map $g$ behaves as follows:
\begin{align*}
    d\bigl(g(p),g(q)\bigr)&=d\bigl((p,\tau_Z(p)),(q,\tau_Z(q))\bigr)\\&=\Bigl( \bigl( d(p,q) \bigr)^2+\bigl(d(\tau_Z(p),\tau_Z(q)) \bigr)^2\Bigr)^{1/2}=\sqrt{2}d(p,q).
\end{align*}
Hence $g$ induces an isometric embedding of the scaled metric space $(\sqrt{2})Z$ into $Z_x\times Z_y\subset Z_x \times Z_y\times Z_z$ and we let $Z_{x\bar{y}}$ denote the image of $(\sqrt{2})Z$ via this embedding. 

 Let $X$ be the locally $\CAT(0)$ space we discussed above. Recall that the fundamental group of $X$ is $C$ and that $X$ contains an isometrically embedded copy $Z_d$ of $Z$ whose fundamental group is the subgroup $A_d=F_d(d_1,d_2)\rtimes_{\varphi^n}\langle t_d\rangle$ of $C$. Rescale the metric on $X$ such that $Z_d$ is isometric to $(\sqrt{2})Z$. 

Define $$W=\bigl(X\bigsqcup(Z_x\times Z_y\times Z_z)\bigsqcup \bigl(\sqrt{2})Z\times [0,1]\bigr)_{/\sim}$$ where the equivalence relation generated by mapping $(\sqrt{2})Z\times\{0\}$ isometrically onto $Z_d\subset X$ and mapping $(\sqrt{2})Z\times\{1\}$ isometrically onto $Z_{x\bar{y}}\subset Z_0\times Z_1\times Z_2$ appropriately. Therefore, $W$ is a locally $\CAT(0)$ space by Proposition~\ref{b1}. By Van Kampen's theorem, we also have $$\pi_1(W)=C\ast_{\langle d_i=x_iy_i^{-1},t_d=t_xt_y\rangle} (A_x\times A_y \times A_z).$$ Let $S_1$ be an oriented circle labeled by $s_1$ and let $X_1=W\times S_1$. We equip $X_1=W\times S_1$ with the product metric. Then $X_1$ is a locally $\CAT(0)$ space and $\pi_1(X_1)=B_1$.

\smallskip

\noindent
\emph{Case $m=2$.}
 We first show that $Z_x\times Z_z$ contains an isometrically embedded copy $Z_{xz}$ of $(\sqrt{2})Z$ whose fundamental group is the subgroup $A_{xz}$ of $A_x\times A_z$ generated by $x_1z_1$, $x_2z_2$, $t_xt_z$. To see this, define a map $g'\!: Z \to Z_x\times Z_z$ by $g'(p)=(p,p)$. The induced homomorphism $g'_*\!:A\to A_x\times A_z$ is given by $a\mapsto x_1z_1$, $b\mapsto x_2z_2$, and $t\mapsto t_xt_z$. Metrically, the map $g'$ behaves as follows:
 $$   d\bigl(g'(p),g'(q)\bigr)\; =\; d\bigl((p,p),(q,q)\bigr)\; =\; \Bigl( \bigl( d(p,q) \bigr)^2+\bigl(d(p,q) \bigr)^2\Bigr)^{1/2} \; =\; \sqrt{2}d(p,q)\,.
$$
Hence $g'$ induces an isometric embedding of the scaled metric space $(\sqrt{2})Z$ into $Z_x\times Z_z\subset Z_x \times Z_y\times Z_z$. Then the image $Z_{xz}$ of $(\sqrt{2})Z$ via this embedding is the desired subspace of $Z_x\times Z_z$. Similarly, $Z_y\times Z_z$ contains an isometric embedding copy $Z_{yz}$ of $(\sqrt{2})Z$ whose fundamental group is the subgroup $A_{yz}$ of $A_y\times A_z$ generated by $y_1z_1$, $y_2z_2$, $t_yt_z$.


Let $X_2={\bigl(X_1 \bigsqcup \bigl((\sqrt{2})Z \times [0,1]\bigr)\bigr)}_{/\sim}$, where the equivalence relation generated by mapping $(\sqrt{2})Z\times\{0\}$ isometrically onto $Z_{xz}\subset X_1$ and mapping $(\sqrt{2})Z\times\{1\}$ isometrically onto $Z_{yz}\subset X_1$ appropriately. Then $X_2$ is also a locally $\CAT(0)$ space by Proposition~\ref{b1} and $\pi_1(X_2)=B_2$.



\smallskip

\noindent
\emph{Case $m\geq 3$.}
For each $m\geq 3$ we construct the locally $\CAT(0)$ space $X_m$ by induction on $m$ as follows. Let $\Sigma_m$ be an oriented circle. Let $X_m={\bigl(X_{m-1} \bigsqcup \bigl(\Sigma_{m} \times [0,1]\bigr)\bigr)}_{/\sim}$, where the equivalence relation generated by mapping $\Sigma_m\times\{0\}$ isometrically onto the oriented circle labeled by $s_1$ in $X_{m-1}$ and mapping $\Sigma_m\times\{1\}$ isometrically onto the oriented circle labeled by $s_{m-1}$ in $X_{m-1}$ appropriately. Then $X_m$ is a locally $\CAT(0)$ space by Proposition~\ref{b1} and $\pi_1(X_m)=B_m$.
\end{proof}

The following lemma is Lemma~5.1 from \cite{MR3705143}. It is a special case of a result of Bass~\cite{MR1239551}. It will be used to prove that $G_m\leq B_m$ for each $m \geq 1$.

\begin{lem}[Injectivity for graphs of groups] \label{basslemma}
Suppose $\mathcal{A}$ and $\mathcal{B}$ are graphs of groups such that
the underlying graph $\Gamma_{\mathcal{A}}$ of $\mathcal{A}$ is a
subgraph of the underlying graph of $\mathcal{B}$. Let $A$ and $B$ be
their respective fundamental groups. Suppose that there are injective
homomorphisms $\psi_e\! : A_e \to B_e$ and $\psi_v\! : A_v \to B_v$
between edge and vertex groups, for all edges $e$ and vertices $v$ in
$\Gamma_{\mathcal{A}}$, which are compatible with the edge-inclusion
maps. That is, whenever $e$ has initial vertex $v$, the diagram 
$$
 \begin{tikzcd}[row sep=1.8em,column sep=3em]
  A_e \arrow[r, "i_e"] \arrow[d,"\psi_e"] & A_v \arrow[d, "\psi_v"] \\
 B_e \arrow[r, "j_e"] &                 B_v \\
  \end{tikzcd}
$$
commutes. 

If $j_e(\psi_e(A_e)) = \psi_{v} (A_{v}) \cap j_e(B_e)$ whenever $e$ has initial
vertex $v$, then the induced homomorphism $\psi \!: A \to B$ is
injective. 
\end{lem}







\begin{prop}
$G_m$ is a subgroup of $B_m$ for $m\geq 1$.
\end{prop}

\begin{proof}
There are three cases to consider; $m=1$, $m=2$, and $m \geq 3$.

\smallskip

\noindent
\emph{Case $m=1$.}
Recall that 
$$G_1 \; =\; \bigl[H\ast_{\langle d_i=x_iy_i^{-1}\rangle} (F_x\times F_y \times F_z)\bigr] \times {\langle s_1 \rangle}
$$
and 
$$
B_1 \; =\; \bigl[C\ast_{\langle d_i=x_iy_i^{-1}, t_d=t_xt_y\rangle} (A_x\times A_y \times A_z)\bigr] \times {\langle s_1 \rangle}\, .
$$
It sufficient to prove an inclusion $G'_1 \to B'_1$ for 
$$
G'_1 \; =\;H\ast_{\langle d_i=x_iy_i^{-1}\rangle} (F_x\times F_y \times F_z) 
$$
and
$$
B'_1 \; =\; C\ast_{\langle d_i=x_iy_i^{-1}, t_d=t_xt_y\rangle} (A_x\times A_y \times A_z)\, .
$$
The underlying graph in each case is a line segment with one edge and two vertices. 

In the $G'_1$ case the edge group is $F = \langle a, b\rangle$ and monomorphisms are 
$$i_1\!: F \to H: a \mapsto d_1, \, b\mapsto d_2$$ and $$i_2\!: F \to F_x \times F_y\times F_z: a \mapsto x_1y^{-1}_1,\,  b \mapsto x_2y^{-1}_2\, .$$



In the $B'_1$ case the edge group is $$A \;=\; \langle a, b, t \, |\,  tat^{-1} = \varphi^n(a), tbt^{-1} = \varphi^n(b) \rangle$$ and the monomorphisms are 
$$
    j_1\!:\, A\to C: a\mapsto d_1, \,  b\mapsto d_2, \, t\mapsto t_d
$$
and
$$
    j_2\!:\, A\to A_x\times A_y \times A_z: a\mapsto x_1y_1^{-1},\,  b\mapsto x_2y_2^{-1},\,  t\mapsto t_xt_y \, .
$$

With the vertical maps as the obvious inclusion maps, it is immediate that the following two diagrams commute. 

$$
 \begin{tikzcd}[row sep=1.8em,column sep=3em]
  F \arrow[r, "i_1"] \arrow[d] & H \arrow[d] & & F \arrow[d] \arrow[r,"i_2"] & F_x\times F_y \times F_z \arrow[d] \\
 A \arrow[r, "j_1"] &                 C      & & A \arrow[r,"j_2"] & A_x \times A_y \times A_z \\
  \end{tikzcd}
$$


Moreover, $j_1(F)=F_d=H\cap A_d=H\cap j_1(A)$ where the middle equality is by hypothesis (equation~(\ref{eq2})) and $j_2(F)=  F_{x\bar{y}} = (F_x\times F_y\times F_z)\cap A_{x\bar{y}}= (F_x\times F_y\times F_z)\cap j_2(A)$ by equation~(\ref{eq1}). Therefore, $G'_1$ is a subgroup of $B'_1$ by Lemma~\ref{basslemma} and so $G_1=G'_1\times \langle s_1\rangle$ is a subgroup of $B_1=B'_1\times \langle s_1\rangle$.

\smallskip

\noindent
\emph{Case $m=2$.}
 Recall that $G_2$ is an HNN extension along a free group $F=\langle a,b\rangle$ with edge monomorphisms (also labeled) $i_1$ and $i_2$ as follows. 
$$
    i_1\!:\, F\to G_1: \, a\mapsto x_1z_1, \, b\mapsto x_2z_2
$$
and 
$$
    i_2\!:\, F\to G_1:\, a\mapsto y_1z_1, \, b\mapsto y_2z_2\, .
$$
Similarly, $B_2$ is an HNN extension along a free-by-cyclic group $$A \;=\; \langle a, b, t \, |\,  tat^{-1} = \varphi^n(a), tbt^{-1} = \varphi^n(b) \rangle$$ with edge monomorphisms $j_1$ and $j_2$
$$
    j_1\!:\, A\to B_1:\, a\mapsto x_1z_1,\, b\mapsto x_2z_2,\, t\mapsto t_xt_z
$$
and
$$
    j_2\!:\, A\to B_1:\, a\mapsto y_1z_1,\, b\mapsto y_2z_2,\, t\mapsto t_yt_z\, .
$$
In both cases the underlying graph is a circle with one edge and one vertex. 
As before, the vertical maps are inclusions and it is straightforward to verify that the following two diagrams commute
$$
 \begin{tikzcd}[row sep=1.8em,column sep=3em]
  F \arrow[r, "i_1"] \arrow[d] & G_1 \arrow[d] & & F \arrow[d] \arrow[r,"i_2"] & G_1 \arrow[d] \\
 A \arrow[r, "j_1"] &                 B_1      & & A \arrow[r,"j_2"] & B_1 \\
  \end{tikzcd}
$$

%
Moreover, the intersections behave as we want. Note that 
$A_{xz} \cap C$ is trivial by the normal form for the amalgamation $B'_1$. Therefore, $A_{xz} \cap H$ is also trivial and so $$A_{xz} \cap G_1=A_{xz} \cap G'_1 = A_{xz} \cap (F_x \times F_y \times F_z) = F_{xz}\,.$$ This establishes the first equality $$j_1(F)=F_{xz}=G_1\cap A_{xz}=G_1\cap j_1(A)\,.$$ Likewise, one can show that $$j_2(F)=F_{yz}=G_1\cap A_{yz}=G_1\cap j_2(A)\,.$$ Therefore, $G_2$ is a subgroup of $B_2$ by Lemma~\ref{basslemma}.


\smallskip

\noindent
\emph{Case $m\geq 3$.}
Recall that for $m\geq 3$ the $G_m$ and $B_m$ are defined inductively by
$$
G_m \; =\; \langle G_{m-1}, s_m | s_m^{-1}s_1s_m=s_{m-1}\rangle 
$$
and
$$
B_m \; =\; \langle B_{m-1}, s_m | s_m^{-1}s_1s_m=s_{m-1}\rangle 
$$
The inclusion $G_m \leq B_m$ is proved by induction on $m$. In both the $G_m$ and the $B_m$ case, the underlying graph is a circle with one edge and one vertex and the edge group is $\Z = \langle s\rangle$. In both cases, the monomorphisms are given by $s \mapsto s_1$ and $s \mapsto s_{m-1}$. We leave it to the reader to verify that the conditions of 
 Lemma~\ref{basslemma} are satisfied and that $G_m$ is a subgroup of $B_m$ for each $m\geq 3$.
\end{proof}


\subsection{Divergence of subgroups of $\CAT(0)$ groups} 
\label{sec:fivethree}


In this section we provide explicit examples of $\CAT(0)$ groups containing finitely presented subgroups with divergence of the form $r^\alpha$ for $\alpha $ dense in $[2,\infty)$ and $r^q\log(r)$ for integers $q \geq 2$. These are obtained by furnishing $\CAT(0)$ groups containing finitely presented subgroups which contain appropriately distorted free subgroups. We end the section with examples of $\CAT(0)$ groups containing contracting elements which are not contracting in certain finitely presented subgroups. Here are the details.

For the $r^q\log(r)$ divergence, we work with the product of a $\CAT(0)$ free-by-cyclic group and $\Z$. 

\begin{lem}
\label{logdist}
Let $A \;=\; \langle a, b, t \, |\,  tat^{-1} = \varphi(a), tbt^{-1} = \varphi(b) \rangle$ be a free-by-cyclic group and let $F$ be the free subgroup generated by $a$ and $b$. Then the distortion function $\Dist^A_F$ is equivalent to an exponential function and admits an exponentially bounded sequence of palindromic certificates in $F$.
\end{lem}

\begin{proof}
The distortion function $\Dist^A_F$ is exponential because $\varphi$ is an exponentially growing automorphism. Upper bounds on the distortion function are obtained algebraically by successive applications of Britton's Lemma to a word $w(a,b,t)$ representing an element of $F$, and keeping track of how lengths grow under replacements $t^{\pm 1}u(a,b)t^{\mp 1} \to \varphi^{\pm 1}(u
)$. The sequence of elements $(\varphi^n(a))$ grows like $(1+\sqrt{2})^n$ and provides a lower bound for the distortion function.

We now prove $\Dist^A_F$ admits an exponentially bounded sequence of palindromic certificates in $F$. We define $a_n=\varphi^n(a)$ for each positive integer $n$. Since $\varphi$ is a palindromic automorphism, each $a_n$ is a palindromic element in $F$. We also observe that for each positive integer $n$ $$|a_{n+1}|_R\leq 3 |a_{n}|_R \quad \text{ and } \quad |a_{n}|_R\geq 2^n,$$ where $R=\{a,b\}$ is a generating set of $F$. The first inequality follows easily from the fact the lengths $|\varphi(a)|_R$ and $|\varphi(b)|_R$ are both bounded above by three. The second inequality can be proved by estimating the length of $\varphi^n(a)$ in the free abelian group on $a$ and $b$. Eigenvalues give a lower bound for the latter estimate as $(1+\sqrt{2})^n$ which is greater than $2^n$. Moreover,
$$|a_n|_T=|\varphi^n(a)|_T=|t^nat^{-n}|_T\leq 2n+1\leq 3n\text{ for each } n\geq 1,$$
where $T=\{a,b,t\}$ is a generating set of $A$. Therefore, $$2^{|a_n|_T/3}\leq 2^n\leq |a_n|_R\,.$$ This implies that the exponential distortion function $\Dist^A_F$ admits an exponentially bounded sequence of palindromic certificates $(a_n)$.
\end{proof}

\begin{table}[h]
\begin{tabular}{|c | c|}
\hline
{\bf Current paper} & {\bf \cite{MR3705143} paper} \\
\hline
$\CAT(0)$ group $C$ &  $\CAT(0)$ group $G_{T,n}$ in Theorem 4.5\\
\hline
$C= \pi_1(X)$ where $X$ is n.p.c. & $G_{T,n} = \pi_1(K_{T,n})$ where $K_{T,n}$ is n.p.c.\\
\hline
Convex subgroup $A_d \leq C$ & Convex subgroup $B_{v_0} \leq G_{T,n}$\\
\hline
$A_d = \pi_1(Z_d)$ where $Z_d \subset X$ is convex &  
$B_{v_0} = \pi_1(L_0)$ where $L_0 \subset K_{T,n}$ is convex (eqn.\ (4.3)) 
\\
\hline
Subgroup $H$ of $C$ & Subgroup $S_{T,n}$ of $G_{T,n}$ in Theorem 5.6\\

\hline
Subgroup $F_d = H\cap A_d$ & Subgroup $A_{v_0} = B_{v_0} \cap S_{T,n}$\\
\hline
\end{tabular}
\caption{Correspondence between current construction and \cite{MR3705143}. The notation n.p.c.\ denotes non-positively curved.}\label{table}
\end{table}

For the $r^\alpha$ divergence, we work with the snowflake groups of \cite{MR3705143}. The snowflake groups in that paper are constructed using an $n$-th power of the automorphism $$\varphi: F_{\{a,b\}} \to F_{\{a,b\}}: a \mapsto aba, b \mapsto a$$ that was introduced in section~\ref{sec:fiveone} above. Note that the automorphism $\varphi$ has Perron-Frobenius eigenvalue $(1+\sqrt{2})$. The precise correspondence with the construction in \cite{MR3705143} is described in Table~\ref{table}.

\begin{lem}\label{snow-dist}
The distortion function $\Dist^H_{F_d}$ is equivalent to $r^\beta$ where $\beta = n\log_m(1+\sqrt{2})$ is dense in $[1,\infty)$. Moreover, this distortion function $\Dist^H_{F_d}$ admits an exponentially bounded sequence of palindromic certificates in $F_d$.
\end{lem}

\begin{proof}
The upper bound for the distortion function is given in Corollary~9.14 of \cite{MR3705143}. The lower bound is realized by the sequence of palindromic words $(\varphi^{np}(a))_{p=1}^\infty$. These have length equivalent to $(1+\sqrt{2})^{np}$ in $F_d$ and length bounded above by function equivalent to $m^p$ in the snowflake group $H = S_{T,n}$. The latter bound is given by the upper half of a snowflake diagram in section~8 of \cite{MR3705143}. 

We show that the same sequence $(\varphi^{np}(a))_{p=1}^\infty$ is an exponentially bounded sequence of palindromic certificates in $F_d$ as follows. Let $a_p = \varphi^{np}(a)$. Since $\varphi$ is palindromic, each $a_p$ is a palindromic element in $F$. Also, there is a constant $C_1>1$ such that
$$
|a_{p+1}|_R \; \leq \; 3^n|a_p|_R \quad \quad {\hbox{and}} \quad \quad |a_p|_R\geq (1+\sqrt{2})^{np}/C_1, 
$$
where $R=\{a,b\}$ is a generating set of $F$. The first inequality shows that $(a_p)$ is an exponentially bounded sequence in $F$. 

Moreover, there is a constant $C_2>1$ such that
$$|a_p|_T\leq C_2m^p\text{ for each } m\geq 1,$$
where $T$ is the finite generating set of the snowflake group of \cite{MR3705143}. Let $C=\max\{C_1,C_2\}$. 
Then
$$(|a_p|_T/C)^\beta\leq (|a_p|_T/C_2)^{\beta}\leq (m^p)^{n\log_m(1+\sqrt{2})}\leq (1+\sqrt{2})^{np}\leq C_1 |a_p|_R \leq C|a_p|_R\,.$$
This implies that the exponential distortion function $\Dist^A_F$ admits an exponentially bounded sequence of palindromic certificates $(a_p)$.
\end{proof}

\begin{thm}
\label{c2c2}
\textit{There exist $\CAT(0)$ groups containing finitely presented groups whose divergence is equivalent to $r^\alpha$ for a dense set of exponents $\alpha \in [2,\infty)$ and to $r^q\log(r)$ for integers $q\geq 2$.}
\end{thm}

\begin{proof}
We consider the two families of functions separately.

\noindent
\emph{Case 1. $r^q\log(r)$ divergence.}
In this case the $\log$ term is provided by the inverse of an exponential distortion function. We define 

\begin{align*}
    & C \; =\; \langle d_1, d_2, t_d \, |\,  t_dd_1t_d^{-1} = \varphi(d_1), t_dd_2t_d^{-1} = \varphi(d_2) \rangle\times \langle s\rangle;\\& A_d\;=\; \langle d_1, d_2, t_d \, |\,  t_dd_1t_d^{-1} = \varphi(d_1), t_dd_2t_d^{-1} = \varphi(d_2) \rangle; 
\end{align*}
Then $A_d$ is the fundamental group of a locally $\CAT(0)$ space $Z_d$ (see section~\ref{sec:fiveone}) and $C$ is the fundamental group of the locally $\CAT(0)$ space (with the product metric) $X=Z_d\times S^1$. Note that $Z_d$ is isometrically embedded into $X$. Let $H$ be a subgroup of $C$ generated by $d_1$, $d_2$, and $t=t_ds$. Then $H$ is a free-by-cyclic group 
$$ H \; =\; \langle d_1, d_2, t \, |\,  td_1t^{-1} = \varphi(d_1), td_2t^{-1} = \varphi(d_2) \rangle$$
and $H\cap A_d=F_d$, where $F_d$ is the free group generated by $d_1$ and $d_2$.

From the pair of groups $(C,A_d)$ we construct a sequence of $\CAT(0)$ groups $B_m$ as in section~\ref{sec:fivetwo}. Similarly, from the pair of groups $(H,F_d)$ we also construct a sequence of subgroup $G_m$ of $B_m$ as in Theorem~\ref{main_thm}. By Lemma~\ref{logdist} the distortion function $\Dist_{F_d}^H$ is equivalent to an exponential function and it also admits an exponentially bounded sequence of palindromic certificates in $F_d$. Therefore, Theorem~\ref{main_thm} implies that $$\Div_{G_m}\sim r^{m-1}(\Dist_{F_d}^H)^{-1}(r)\sim r^{m-1}\log{r} \text{ for } m\geq 3\,.$$ Taking $q=m-1$ yields the second result.

\medskip

\noindent
\emph{Case 2. $r^\alpha$ divergence.}
In this case we use the snowflake groups and corresponding $\CAT(0)$ groups from \cite{MR3705143}. The precise connection is described in Table~\ref{table}. 

From the pair of groups $(C,A_d)$ in Lemma~\ref{snow-dist} we construct a sequence of $\CAT(0)$ groups $B_m$ as in section~\ref{sec:fivetwo}. Similarly, from the pair of groups $(H,F_d)$ in the same lemma we also construct a sequence of subgroup $G_m$ of $B_m$ as in Theorem~\ref{main_thm}. By Lemma~\ref{snow-dist} the distortion function $\Dist_{F_d}^H$ is equivalent to $r^\beta$ where $\beta$ is dense in $[1,\infty)$ and it also admits an exponentially bounded sequence of palindromic certificates in $F_d$. Therefore, Theorem~\ref{main_thm} implies that $$\Div_{G_m}\sim r^{m-1}(\Dist_{F_d}^H)^{-1}(r)\sim r^{m-1+1/\beta} \text{ for } m\geq 3\,.$$ Taking $\alpha=m-1+1/\beta$ yields the first result.
\end{proof}

In the following proposition, we prove the polynomial divergence functions of $\CAT(0)$ groups $B_m$ and cyclic subgroups $\langle s_m \rangle$ in $B_m$.

\begin{prop}
\label{catdiv}
$${\rm Div}_{\langle s_m\rangle}^{B_m} \sim r^m \text{ and } {\rm Div}_{B_m}\sim r^m.$$ 
\end{prop}
\begin{proof}
We first prove the upper bound for ${\rm Div}_{B_m}\sim r^m$ and ${\rm Div}_{\langle s_m\rangle}^{B_m}\sim r^m$ by using the concept of (geometrically) thick groups (see Definition 7.1 in \cite{MR2501302}). We first note that each group $B_1$ is thick of order zero (also known as unconstricted) because it is a direct product (see Example 1 following Definition 3.4 of \cite{MR2501302}). By induction on $m$ and the recursive construction of the $B_m$, it follows from Definition~7.1 of \cite{MR2501302} that the groups $B_m$ are thick of order at most $m-1$. Therefore, the divergence of the group $B_m$ is at most $r^m$ by Corollary~4.17 of \cite{MR3421592}. This is implies that the divergence ${\rm Div}_{\langle s_m\rangle}^{B_m}$ is also at most $r^m$. Using the same strategy in section~\ref{lbfdg} and section~\ref{sg} we can also prove the divergence ${\rm Div}_{\langle s_m\rangle}^{B_m}$ is at least $r^m$. Therefore, the divergence ${\rm Div}_{B_m}$ is also at least $r^m$. Thus, we have $${\rm Div}_{\langle s_m\rangle}^{B_m} \sim r^m \text{ and } {\rm Div}_{B_m}\sim r^m.$$
\end{proof}

In the following proposition, we use cyclic subgroup divergence to prove the existence of $\CAT(0)$ groups containing contracting elements which are not contracting in some finitely presented subgroups. 

\begin{prop}
\label{prop:contract}
\textit{There exist $\CAT(0)$ groups containing contracting elements which are not contracting in some finitely presented subgroups.}
\end{prop}

\begin{proof}
Consider the $\CAT(0)$ group $B_2$, the subgroup $G_2$ and the element $s_2$. By Proposition~\ref{catdiv} the divergence of the cyclic subgroup $\langle s_2\rangle$ in $B_2$ is quadratic. Therefore, $s_2$ is a contracting element in the $\CAT(0)$ group $B_2$ by Theorem 2.14 in \cite{MR3339446}. We note that the concept of contracting quasi-geodesics in \cite{MR3339446} is even stronger than the concept of contracting quasi-geodesics which is used in this paper. 

By Proposition~\ref{cyclicdivergence} and the proof of Theorem~\ref{c2c2} we can choose the $\CAT(0)$ group $B_2$ and the subgroup $G_2$ such that the divergence of the cyclic subgroup $\langle s_2\rangle$ in $G_2$ is $r\log{r}$ or $r^\alpha$ for some $\alpha\in (1,2)$. Therefore, $s_2$ is not a contracting elements of group $G_2$ by Proposition~\ref{lem_CS}.
\end{proof}

\begin{rem}
\label{dehn}
One can use Dehn functions to see that none of the groups $G_m$ ($m \geq 1$) of Theorem~\ref{c2c2} are $\CAT(0)$. 
We include a sketch of the argument here for readers with some familiarity with Dehn functions. 

Let $w$ denote a word in the generators of $G'_1 = H \ast_{F_d = F_{x\bar{y}}} (F_x \times F_y \times F_z)$ which represents the identity in $G'_1 \leq G_1 \leq G_m$. 

\noindent
\emph{Claim 1.} ${\rm Area}_{G_m}(w) = {\rm Area}_{G'_1}(w)$.

We prove this claim by showing that a minimal area van Kampen diagram for $w$ over $G_m$ cannot have any edges labelled by $e_i$ for $1 \leq i\leq m$ and so is a van Kampen diagram over $G'_1$. 

We first show that a minimal area van Kampen diagram for $w$ in $G_m$ cannot involve any of the $s_i$ for $i \geq 2$. Assume to the contrary that there are such edge labels and let $j$ be the maximum index of such edges. Because $j$ is maximal, the notion of $s_j$--corridors are well defined for this van Kampen diagram. An $s_j$--corridor does not meet the boundary of the van Kampen diagram (because $s_j$ does not appear in the word $w$) and so must form an annulus. An innermost such annulus contradicts minimality of the area because for $j \geq 2$ each $s_j$ conjugates a free group to another free group. The interior of an innermost $s_j$--annulus is just a finite tree. One can excise this $s_j$--annulus to form a new diagram for $w$ of strictly smaller area. 

Next we show by contradiction that a minimal area van Kampen diagram for $w$ cannot contain any edges labelled $s_1$. We have already shown that a minimal area van Kampen diagram does not contain any edges labelled $s_j$ for $j \geq 2$. Therefore, $s_1$--corridors are well defined in this van Kampen diagram. As above, $s_1$--corridors can only form annuli. Because $s_1$ commutes with all of $G'_1$ the inner and outer boundaries of $s_1$--annuli are identical words. We can excise such annuli and glue their inner boundaries to their outer boundaries to obtain a van Kampen diagram for $w$ with smaller area. 

\noindent
\emph{Case $H$ is a snowflake group.} In this case take $(w_n)$ to be a sequence of snowflake words in $H$ which realize the Dehn function of $H$, $\delta_H$, which we assume strictly dominates a quadratic function. By Claim 1, the area of $w$ in $G_m$ is equal to the area of $w$ in $G'_1$. The homomorphism $G'_1 \to H$ defined as the identity on $H$ and sending $x_i$ to $d_i$ and sending $y_i$ and $z_i$ to the identity for $i=1,2$ is a Lipschitz retract of $G'_1$ onto $H$ in the sense of \cite[Prop.\ 12]{MR3073930}. This implies that the area of $w$ in $G'_1$ is bounded below by the area of $w_n$ in $H$. Therefore, the Dehn function of $G_m$, $\delta_{G_m}$, satisfies 
$$
\delta_{G_m}(x) \; \succeq \; \delta_{H}(x) \; \succ \; x^2 
$$
and so $G_m$ is not $\CAT(0)$. 

\noindent
\emph{Case $H= F_2 \rtimes_\varphi \Z$.} 
For concreteness, let $H =  F_{\{d_1,d_2\}}\rtimes_\varphi \langle t\rangle$ and 
define the sequence of words $(v_n)$ by $v_n = t^nd_1t^{-n}$. Note that $v_n = \varphi^n(d_1)$ realizes the exponential distortion of $F$ in $H$.

Since $G'_1 = H \ast_{F_d=F_{x\bar{y}}} (F_x \times F_y \times F_z)$ and $F_z$ commutes with $F_{x\bar{y}}$, the words
$$
w_n \; =\; [v_n, z_1]
$$
have length $(4n+4)$ and represent the identity in $G'_1 \leq G_m$. By Claim 1, the area of $w_n$ in $G_m$ equals the area of $w_n$ in $G'_1$. The latter area is bounded below by the length of $\varphi^n(d_1)$ because the $z_1$--corridor in a van Kampen diagram for $w_n$ over $G'_1$ which connects the two boundary edges labelled $z_1$ has length bounded below by $|\varphi^n(x_1\bar{y}_1)| = |\varphi^n(d_1)|$. This is because each lateral boundary word of the $z_1$--corridor defines an element which is equal to $t^n d_1 t^{-n}$ in $G'_1$ and so to the element $\varphi^n(x_1\bar{y}_1)$ in the group $F_{x\bar{y}}$ and this group is undistorted in $F_x \times F_y \times F_z$. Thus, the area of $w_n$ is bounded below by the area of this $z_1$--corridor which is an exponential function of $n$. Therefore, 
$\delta_{G_m}(x)$ is bounded below by an exponential function 
and so $G_m$ is not $\CAT(0)$.
\end{rem}

\bibliographystyle{alpha}
\bibliography{Tran}
\end{document}